%% file: main.tex
\documentclass[12pt,reqno]{amsart}
\pdfoutput=1

\usepackage[utf8]{inputenc}   
\usepackage[T1]{fontenc}
\usepackage{lmodern}
\usepackage{amsmath}
\usepackage{amsthm}
\usepackage{thmtools}
\usepackage{amssymb}
\usepackage{amscd}
\usepackage[inline]{enumitem}
\usepackage{caption}
\usepackage[dvipsnames]{xcolor}
\usepackage[pdftex, colorlinks, linkcolor = MidnightBlue, citecolor = WildStrawberry, urlcolor = BrickRed]{hyperref}
\usepackage[margin=1.in]{geometry} 

\usepackage{setspace}
\usepackage{tikz-cd}

\newcommand{\Z}{\mathbb{Z}}
\newcommand{\R}{\mathbb{R}}

\newcommand{\C}{\mathbb{C}}

\newcommand{\D}{\mathbb{D}}
\newcommand{\DD}{\mathcal{D}}

\newcommand{\M}{\mathcal{M}}

\newcommand{\HH}{\mathbb{H}}

\newcommand{\eps}{\varepsilon}
\newcommand{\vphi}{\varphi}
\newcommand{\wtilde}{\widetilde}
\newcommand{\what}{\widehat}

\DeclareMathOperator{\Symp}{Symp}
\newcommand{\Smcg}{\operatorname{Mod}_{\omega}}

\DeclareMathOperator{\Diff}{Diff}
\DeclareMathOperator{\Mcg}{Mod}

\newcommand{\Grass}{\mathcal{L}}
\newcommand{\GrassOr}{\mathcal{L}^{\operatorname{or}}}

\newcommand{\cob}{\rightsquigarrow} 

\newcommand{\Cob}{\mathsf{Cob}}
\newcommand{\Lag}{\mathsf{Lag}}

\newcommand{\Gimm}{\Omega_{\operatorname{imm}}}
\newcommand{\Gunob}{\Omega_{\operatorname{unob}}}
\newcommand{\Gred}{\widetilde{\Omega}_{\operatorname{unob}}}
\newcommand{\Gemb}{\Omega_{\operatorname{emb}}}
\newcommand{\Gcob}{\Omega}

\newcommand{\chord}{\mathcal{O}}  	
\newcommand{\data}{\mathcal{D}} 	

\DeclareMathOperator{\Fuk}{Fuk}
\DeclareMathOperator{\DFuk}{DFuk}

\newcommand{\FukCob}{\operatorname{Fuk}_{\operatorname{cob}}}

\newcommand{\Kgrp}{K_0 \DFuk}

\DeclareMathOperator{\Mod}{mod}

\newcommand{\Def}[1]{\textbf{#1}}

\DeclareMathOperator{\id}{id}

\DeclareMathOperator{\Hom}{Hom}
\newcommand{\tens}{\otimes}
\DeclareMathOperator{\Ker}{Ker}

\newcommand{\bdry}{\partial}
\DeclareMathOperator{\Int}{int}

\newcommand{\iso}{\cong}
\newcommand{\del}{\partial}

\DeclareMathOperator{\flux}{flux}
\DeclareMathOperator{\hol}{Hol}

\DeclareMathOperator{\area}{area}

\theoremstyle{plain}

\newtheorem{maintheorem}{Theorem}
\newtheorem{maincor}[maintheorem]{Corollary}

\newtheorem{theorem}{Theorem}[section]
\newtheorem{proposition}[theorem]{Proposition}
\newtheorem{lemma}[theorem]{Lemma}
\newtheorem{corollary}[theorem]{Corollary}
\newtheorem{definition}[theorem]{Definition}

\newtheorem{main_question}{Question}

\theoremstyle{definition}
\newtheorem{remark}[theorem]{Remark}

\theoremstyle{remark}
\newtheorem*{notation*}{Notation}

\author{Dominique Rathel-Fournier}
\email{d.rathelfournier@gmail.com}
\title{Unobstructed Lagrangian cobordism groups of surfaces}
\date{\today}

\begin{document}

\begin{abstract}
	We study Lagrangian cobordism groups of closed symplectic surfaces of genus $g \geq 2$ whose relations
	are given by unobstructed, immersed Lagrangian cobordisms. 
	Building upon work of Abouzaid \cite{Abouzaid08} and Perrier \cite{Perrier19},
	we compute these cobordism groups and show that they are isomorphic to the Grothendieck group of the derived Fukaya category of the surface.
	The proofs rely on techniques from two-dimensional topology to construct cobordisms
	that do not bound certain types of holomorphic polygons.
\end{abstract}

\maketitle

\thispagestyle{empty}

\tableofcontents

\section{Introduction}
\label{section:intro}

\input{sections/intro.tex}

\section{Preliminaries}
\label{section:preliminaries}

\input{sections/preliminaries.tex}


\section{Unobstructed cobordisms and the proof of Theorem \ref{thm:cone_decomposition}}
\label{section:fukaya_cob}

\input{sections/fukaya_cob.tex}

\section{Topologically unobstructed cobordisms}
\label{section:top_unob_cobordisms}

\input{sections/top_unob.tex}

\section{Topological obstruction of surgery cobordisms in dimension 2}
\label{section:obstruction_dim2}

\input{sections/obstruction_dim2.tex}

\section{Lagrangian cobordism invariants of curves}
\label{section:invariants}

\input{sections/invariants.tex}

\section{Relations between isotopic curves}
\label{section:isotopies}

\input{sections/isotopies.tex}

\section{Action of the mapping class group}
\label{section:action_mcg}

\input{sections/action_mcg.tex}

\section{Computation of \texorpdfstring{$\Gunob(\Sigma)$}{the unobstructed cobordism group}}
\label{section:computation}

\input{sections/computation.tex}

\appendix

\section{Action-type invariants of Lagrangians in monotone manifolds}
\label{appendix:action}

\input{sections/appendix_action.tex}

\section{Proof of Proposition \ref{prop:isotopy_relation}}
\label{appendix:isotopy}

\input{sections/appendix_isotopy.tex}

\section{A criterion for marked teardrops}
\label{appendix:teardrops}

\input{sections/appendix_teardrops.tex}

\bibliographystyle{alpha}
\bibliography{bibliography}
	
\end{document}

%% file: sections/intro.tex
Lagrangian cobordism is a relation between Lagrangian submanifolds of a symplectic manifold $(M, \omega)$
that was introduced by Arnold \cite{Arnold}.
A consequence of the Gromov-Lees h\nobreakdash-principle is that
the study of general immersed Lagrangian cobordisms essentially reduces to algebraic topology,
see \cite{Eliashberg, Audin_thesis}.
In contrast, cobordisms satisfying suitable geometric constraints display remarkable rigidity phenomena.
A fundamental example of this is the proof by Biran and Cornea \cite{BC13, BC14}
that monotone, embedded Lagrangian cobordisms 
give rise to cone decompositions in the derived Fukaya category $\DFuk(M)$.
See also \cite{BC-lefschetz,BCS, Pictionary} for further developments, 
as well as the work of Nadler and Tanaka \cite{NadlerTanaka}, who obtained related results from a different perspective.

A natural problem is to establish to what extent Lagrangian cobordisms determine the structure of $\DFuk(M)$.
In this paper, we consider the following formalization of this question.
Fix classes of Lagrangians in $M$ and of Lagrangian cobordisms in $\C \times M$, possibly equipped with extra structures
(the precise classes considered in this paper will be specified shortly).
The associated \emph{Lagrangian cobordism group} $\Gcob(M)$ is the abelian group
whose elements are formal sums of Lagrangians in $M$, modulo relations coming from Lagrangian cobordisms.
On the side of $\DFuk(M)$, part of the information about cone decompositions is encoded
in the \emph{Grothendieck group} $\Kgrp(M)$.
(See Section \ref{section:preliminaries} for precise definitions of $\Gcob(M)$ and $\Kgrp(M)$.)

Whenever the cone decompositions results of \cite{BC14} can be extended to the classes of Lagrangians under consideration,
there is an induced morphism
\[
\Theta: \Gcob(M) \to \Kgrp(M),
\]
see \cite[Corollary 1.2.1]{BC14}.
Biran and Cornea posed the following question.

\begin{main_question}
\label{main_question}
When is the map $\Theta$ an isomorphism?
\end{main_question}
There are only two cases for which $\Theta$ is known to be an isomorphism
(for appropriate classes of Lagrangians):
the torus $T^2$ by work of Haug \cite{Haug15} (see also \cite{Hensel} for a refinement of this result),
and Liouville surfaces by work of Bosshard \cite{Bosshard}.

In this paper, we consider this problem in the case of a closed surface $\Sigma$ of genus $g \geq 2$.
In this case, the group $\Kgrp(\Sigma)$ was computed by Abouzaid \cite{Abouzaid08}, who showed that there is an isomorphism
\begin{equation}
\label{eq:computation_Kgrp}
\Kgrp(\Sigma) \iso H_1(S \Sigma; \Z) \oplus \R,
\end{equation}
where $S\Sigma$ is the unit tangent bundle of $\Sigma$.
The main goal of this paper is to answer Question­~\ref{main_question}
by exhibiting a Lagrangian cobordism group $\Gunob(\Sigma)$ whose relations are given by \emph{unobstructed} immersed cobordisms,
such that the associated map $\Theta: \Gunob(\Sigma) \to \Kgrp(\Sigma)$ is an isomorphism.

\subsection{Main results}

\subsubsection{Unobstructed cobordisms and cone decompositions}

Our first main result is an extension of the cone decomposition results of \cite{BC14} to a class of immersed Lagrangian cobordisms
that are \emph{unobstructed}, in the sense that they have a well-behaved Floer theory.
There exist many different mechanisms to achieve unobstructedness in the litterature, at varying
levels of generality and technical complexity.
The precise notion of unobstructedness considered in this work is the following:
we say that an immersed Lagrangian cobordism is \emph{quasi-exact} if it does not bound $J$-holomorphic disks and teardrops 
for suitable almost complex structures $J$ 
(see Section \ref{subsection:unob_cobordisms} for the precise definition).
We will use the terms \emph{quasi-exact} and \emph{unobstructed} interchangeably.

To state the first theorem, we let $(M, \omega)$ be a closed symplectic manifold which is symplectically aspherical, 
in the sense that $\omega|_{\pi_2(M)} = 0$. 
Recall that a Lagrangian submanifold $L \subset M$ is called weakly exact if $\omega|_{\pi_2(M, L)}=0$.
In this paper, we assume that all Lagrangians and Lagrangian cobordisms are oriented and equipped with a spin structure.

\begin{maintheorem}
\label{thm:cone_decomposition}
Let $V: L \cob (L_1, \ldots, L_k )$ be a quasi-exact Lagrangian cobordism between weakly exact, embedded Lagrangian submanifolds of $M$.
Then $L$ admits a cone decomposition in $\DFuk(M)$ with linearization $(L_1, \ldots, L_k)$. 
\end{maintheorem}

\begin{remark}
In this work, we will only apply Theorem \ref{thm:cone_decomposition} to the case of curves on surfaces of genus $g \geq 2$.
However, since the weakly exact case requires no additional ingredients, we include it for the sake of generality.
\end{remark}

\subsubsection{Computation of $\Gunob(\Sigma)$ and comparison with $K$-theory}

In the remainder of this work, $\Sigma$ denotes a closed symplectic surface of genus $g \geq 2$.
We denote by $\Gunob(\Sigma)$ the Lagrangian cobordism group whose generators are non-contractible, oriented, embedded closed curves in $\Sigma$, 
and whose relations are given by quasi-exact cobordisms.
Our second main result is a computation of $\Gunob(\Sigma)$.

\begin{maintheorem}
\label{thm:computation_Gunob}
There is an isomorphism
\[
\Gunob(\Sigma) \iso H_1(S\Sigma; \Z) \oplus \R,
\]
where $S \Sigma$ is the unit tangent bundle of $\Sigma$. 
\end{maintheorem}

In Theorem \ref{thm:computation_Gunob}, the map to $H_1(S\Sigma; \Z)$
is given by the homology class of lifts of curves to $S\Sigma$.
The map to $\R$ is a generalization to closed surfaces of the symplectic area of plane curves.
At the level of generators, these invariants are the same as those considered by Abouzaid in their computation of $\Kgrp(\Sigma)$.

It follows from Theorem \ref{thm:cone_decomposition} that the Yoneda embedding induces
a well-defined morphism $\Theta: \Gunob(\Sigma) \to \Kgrp(\Sigma)$.
By combining Theorem \ref{thm:computation_Gunob} and Abouzaid's result \eqref{eq:computation_Kgrp}, we immediately obtain
the following corollary,
which gives an answer to Question \ref{main_question} for higher genus surfaces.

\begin{maincor}
\label{cor:isomorphism_Kgroup}
The map $\Theta : \Gunob(\Sigma) \longrightarrow \Kgrp(\Sigma)$ is an isomorphism.
\end{maincor}

\subsubsection{The immersed cobordism group}

Let $\Gimm(\Sigma)$ be the Lagrangian cobordism group whose generators are oriented immersed curves in $\Sigma$,
and whose relations are given by all oriented, immersed Lagrangian cobordisms. 
There is an obvious forgetful morphism $\Gunob(\Sigma) \to \Gimm(\Sigma)$.
As a byproduct of our computation of $\Gunob(\Sigma)$, we also obtain a computation of $\Gimm(\Sigma)$.

\begin{corollary}
\label{cor:computation_Gimm}
The natural map $\Gunob(\Sigma) \to \Gimm(\Sigma)$ is an isomorphism.
\end{corollary}

\begin{remark}
\label{rmk:immersed_group}
Corollary \ref{cor:computation_Gimm} corrects an error in Theorem 1.5 of \cite{Perrier19}, where it is claimed that
$\Gimm(\Sigma) \iso H_1(\Sigma; \Z) \oplus \Z_{\chi(\Sigma)}$, the latter group being isomorphic to $H_1(S\Sigma; \Z)$.
The error originates from a mistaken application of the h-principle for Lagrangian immersions in the proof of Lemma 2.1 of \cite{Perrier19}.
\end{remark}

\begin{remark}
The group $\Gimm(\Sigma)$ can also be computed using the 
Gromov-Lees h-principle and homotopy-theoretic methods; 
see \cite[Theorem 1.7]{DRF-immersions}.
\end{remark}

\subsubsection{Further remarks on the main results}

\begin{remark}
In the case of embedded Lagrangians, 
the notion of quasi-exactness was previously considered in \cite{BCS}, as well as in
\cite{SheridanSmith} under the name of
\emph{unobstructed Lagrangian branes}.
The definition given here is a straightforward generalization to the case of immersed 
Lagrangians, where one must
also have control over holomorphic disks with corners at self-intersection points,
of which teardrops are a special case.
One of the results of \cite{BCS} is a cone decomposition theorem for embedded quasi-exact cobordisms;
Theorem \ref{thm:cone_decomposition} is thus an extension of this result to the immersed case.
\end{remark}

\begin{remark}
	Although the definition of quasi-exactness involves holomorphic curves,
	the Lagrangian surgery cobordisms that will be constructed
	in the proof of Theorem \ref{thm:computation_Gunob}
	are unobstructed for \emph{topological reasons}, i.e.
	they are topologically unobstructed
	in the sense of Section \ref{section:top_unob_cobordisms}.
	The reason we consider the larger class of quasi-exact cobordisms
	is that it is closed under concatenations; see Remark \ref{rmk:def_unob_cob} and Remark \ref{rmk:def_unob_perrier}.
\end{remark}

\begin{remark}
	It is expected that Theorem \ref{thm:cone_decomposition} should
	generalize to immersed Lagrangian cobordisms that may bound disks and teardrops, but are nevertheless
	\emph{Floer-theoretically unobstructed} in a suitable sense.
	See for example the work of Biran-Cornea \cite{Pictionary} and Hicks \cite{Hicks-wall_crossing,Hicks-surgery} 
	for progress in this direction.
	One advantage of the class of quasi-exact cobordisms 
	is that it requires much less technical machinery to set up Floer theory and prove Theorem \ref{thm:cone_decomposition}.
	On the other hand,
	avoiding all disks and teardrops
	is the main source of difficulty in the computation of $\Gunob(\Sigma)$.
\end{remark}

\begin{remark}

The reason we consider immersed Lagrangian cobordisms in this paper is that Lagrangian surgery
-- which is our main source of Lagrangian cobordisms --
generally produces cobordisms that have self-intersections.
For this reason, embedded cobordisms are much harder to construct.
A natural question is whether one could obtain results analogous to Theorem \ref{thm:computation_Gunob} and Corollary \ref{cor:isomorphism_Kgroup}
using only embedded cobordisms.
Namely, is there a Lagrangian cobordism group $\Gemb(\Sigma)$
whose relations are given by unobstructed, oriented, embedded cobordisms,
such that the map $\Gemb(\Sigma) \to \Kgrp(\Sigma)$ is an isomorphism?
\end{remark}

\subsection{Outline of the proofs of the main results}
\label{subsection:outline}

\subsubsection{Outline of the proof of Theorem \ref{thm:cone_decomposition}}

The proof of Theorem \ref{thm:cone_decomposition} follows the same scheme as the proof of the cone decomposition results of \cite{BC14}
and the extension to embedded quasi-exact cobordisms developed in \cite{BCS}.
The main technical argument is the construction of a Fukaya category $\FukCob(\C \times M)$
whose objects are quasi-exact cobordisms.
The main difference between our setup and that of \cite{BC14} and \cite{BCS} is that we consider Lagrangian cobordisms that may have self-intersections.
The Floer theory of Lagrangian immersions was developed in its most general form by Akaho-Joyce \cite{AkahoJoyce}.
Restricting to Lagrangians that do not bound holomorphic disks and teardrops allows
us to use the transversality and compactness arguments of \cite{BC14} and \cite{BCS} with minimal modifications, 
rather than the more complicated virtual perturbation techniques of \cite{AkahoJoyce}.
The extension of the results of \cite{BC14} to this case thus
poses no new technical challenges.

\subsubsection{Outline of the proof of Theorem \ref{thm:computation_Gunob}}
\label{subsubsection:outlineB}

The general structure of the computation of $\Gunob(\Sigma)$ is similar to the computation
of $\Kgrp(\Sigma)$ by Abouzaid \cite{Abouzaid08}. 
In fact, many of our arguments can be seen as translations to the language of Lagrangian cobordisms of arguments appearing in \cite{Abouzaid08}.
The main idea of the proof is to use the action of the mapping class group $\Mcg(\Sigma)$ of $\Sigma$
on (a quotient of) $\Gunob(\Sigma)$ to obtain a simple set of generators of $\Gunob(\Sigma)$.
The computation of the action of $\Mcg(\Sigma)$ is done by realizing Dehn twists as iterated Lagrangian surgeries.
The key technical difficulty in doing this -- and the main contribution of the present work -- is to show that one can 
perform these surgeries in such a way 
that the associated surgery cobordisms are unobstructed, i.e. do not bound disks and teardrops.

In Section \ref{section:obstruction_dim2}, we give topological criteria that ensure that such cobordisms are
unobstructed, which rely on techniques from surface topology.
Roughly speaking, these criteria state that surgery cobordisms are unobstructed provided that the curves being surgered do not bound
certain types of polygons.
These conditions may be hard to ensure in practice for curves intersecting in many points.
To bypass this difficulty, we use an inductive argument, based on a technique of Lickorish \cite{Lick64},
that reduces the computation of the action of Dehn twists
to the simplest cases of curves which have only one or two intersections points,
where unobstructedness can be achieved easily.
This method may have interesting applications to the study of $\DFuk(\Sigma)$.

\subsection{Relation to previous work}

Lagrangian cobordism groups of higher genus surfaces
were studied previously by Perrier in their thesis \cite{Perrier19}.
Here, we make some remarks to clarify the relationship between the present work and \cite{Perrier19}.

The present work is heavily inspired by the results and methods of Perrier.
In particular, the general structure of the computation of $\Gunob(\Sigma)$
and many of the surgery procedures that we use
are the same as in \cite{Perrier19}.
However, it appears that the proofs of the key unobstructedness results in \cite{Perrier19} contain significant gaps,
which the present author was unable to fix in a straightforward way.
It is therefore one of the main purposes of this paper to clarify and correct some of the results and proofs of \cite{Perrier19}.

The main novel contributions of the present work are as follows.
First, we provide a proof that unobstructed cobordisms give rise to cone decompositions, filling a gap in \cite{Perrier19}.
It should be noted that the class of unobstructed cobordisms studied in the present work
differs from the class considered in \cite{Perrier19}.
This allows us to deal with some technical issues with the definition used in \cite{Perrier19}; see Remark \ref{rmk:def_unob_perrier}.

Our second main contribution concerns the proofs that the cobordisms constructed in \cite{Perrier19} are unobstructed.
In Section \ref{section:obstruction_dim2}, we formulate a precise criterion for the existence of teardrops on $2$-dimensional surgery cobordisms,
fixing a gap in the proof of the main obstruction criterion used in \cite{Perrier19} (see Proposition 5.11 therein).
We also treat the case of disks with boundary on surgery cobordisms, which is overlooked in \cite{Perrier19}.
We then apply these obstruction criteria to verify that the cobordisms appearing in the computation of $\Gunob(\Sigma)$
are unobstructed.
As explained in Section \ref{subsubsection:outlineB}, we use an inductive argument that 
allows us to circumvent the delicate combinatorial arguments used in \cite{Perrier19},
greatly simplifying the proofs.

Finally, we give a direct proof that holonomy is an invariant of oriented Lagrangian cobordism,
and use this to correct the computation of $\Gimm(\Sigma)$ appearing in \cite{Perrier19}, see Remark \ref{rmk:immersed_group}.

\subsection{Organization of the paper}

This paper is organized as follows.
In Section \ref{section:preliminaries}, we recall some well-known facts about Lagrangian cobordisms and Fukaya categories.
In Section \ref{section:fukaya_cob}, we construct the Fukaya category of unobstructed cobordisms $\FukCob(\C \times M)$
and use it to prove Theorem \ref{thm:cone_decomposition}.

The rest of the paper is devoted to the proof of Theorem \ref{thm:computation_Gunob}.
In Sections \ref{section:top_unob_cobordisms} and \ref{section:obstruction_dim2}, we establish 
the unobstructedness criteria that we will use repeatedly throughout the computation of $\Gunob(\Sigma)$.
In Section \ref{section:invariants}, we define the morphism $\Gunob(\Sigma) \to H_1(S\Sigma; \Z) \oplus \R$
appearing in Theorem \ref{thm:computation_Gunob}.
In Section \ref{section:isotopies}, we determine relations between isotopic curves and reduce the computation of $\Gunob(\Sigma)$
to that of a quotient, the reduced group $\Gred(\Sigma)$.
In Section \ref{section:action_mcg}, we compute the action of Dehn twists on $\Gred(\Sigma)$.
Finally, in Section \ref{section:computation}, we combine the results of Sections \ref{section:top_unob_cobordisms} -- \ref{section:action_mcg} to
complete the computation of $\Gunob(\Sigma)$.

This paper has three appendices. 
In Appendix \ref{appendix:action}, 
we define cobordism invariants of Lagrangians in monotone symplectic manifolds,
which generalize the holonomy of curves on higher genus surfaces (as defined by Abouzaid \cite{Abouzaid08}).
In Appendix \ref{appendix:isotopy}, we prove a minor modification of a result appearing in \cite{Perrier19},
which we state as Proposition \ref{prop:isotopy_relation}.
Appendix \ref{appendix:teardrops} contains
results on marked teardrops that are not used in the rest of the paper but may be of independent interest.

\subsection{Acknowledgements}

This work is part of the author's PhD thesis at the University of Montreal under the supervision of Octav Cornea and François Lalonde.
The author would like to thank Octav Cornea for his invaluable support and patience throughout this project.
The author would also like to thank Jordan Payette for helpful discussions in the early stages of this project.
The author acknowledges the financial support of NSERC Grant \#{}504711 and FRQNT Grant \#{}300576.

%% file: sections/preliminaries.tex
\subsection{Conventions and terminology}

By a \emph{curve} in a surface $\Sigma$, we mean an equivalence class of smooth immersions $S^1 \to \Sigma$,
where two immersions are equivalent if they differ by an orientation-preserving diffeomorphism of $S^1$.
By convention, the circle $S^1$ is always equipped with its standard orientation
as the boundary of the disk.
A curve is \emph{simple} if it is an embedding.
We will often identify simple curves with their image in $\Sigma$.

We say that a Lagrangian immersion has \emph{generic self-intersections} if its only self-inter\-sections
are transverse double points.
Likewise, two Lagrangian immersions $\alpha$ and $\beta$ are in \emph{general position}
if the union $\alpha \cup \beta$ has generic self-intersections.
As in the case of curves, we will generally not distinguish between two
Lagrangian immersions that differ by a diffeomorphism of the domain isotopic to the identity.

\subsection{Lagrangian cobordisms}
\label{subsection:cobordisms}

Let $(M, \omega)$ be a symplectic manifold. 
We equip the product $\wtilde{M} = \C \times M$ with the symplectic form 
$\wtilde{\omega} = \omega_{\C} \oplus \omega$, where $\omega_{\C}$ is the standard symplectic form on $\C$. 
We denote by $\pi_{\C}$ and $\pi_{M}$ the projections on the first and the second factor, respectively.

Let $(\alpha_i: L_i \to M)_{i=1}^{r}$ and $(\beta_j: L_j' \to M)_{j=1}^{s}$ be two finite collections of 
Lagrangian immersions of compact manifolds into $M$. 
We recall the definition of a Lagrangian cobordism between these collections, following Biran and Cornea \cite{BC13}. 

\begin{definition}
\label{def:lagrangian_cobordism}
A \Def{Lagrangian cobordism} from $(\alpha_i)_{i=1}^{r}$ to $(\beta_j)_{j=1}^{s}$ 
is a compact cobordism $V$ with positive end $\coprod L_i$ 
and negative end $\coprod L_j'$, 
along with a proper 
Lagrangian immersion $\iota_V: V \to [0,1] \times \R \times M$ which 
is cylindrical near $\bdry V$ in the following sense. 

For each positive end $L_i$, there is a collar embedding $(1-\eps, 1] \times L_i \to V$ 
over which the immersion $\iota_V$ is given by
$(t, p) \mapsto (t, i, \alpha_i(p))$. 
Likewise, for each negative end $L_j'$ there is a collar embedding $[0, \eps) \times L_j' \to V$ 
over which $\iota_V$ is given by
$(t, p) \mapsto (t , j, \beta_j(p))$.
\end{definition}

We will use the notation $\iota_V: (\alpha_i) \cob (\beta_j)$
to denote a Lagrangian cobordism, or more succintly $V: (L_i) \cob (L_j')$ when
the immersions are clear from the context.

In this work, we will always assume that the Lagrangians are equipped with
an orientation and spin structure.
Likewise, we always assume that cobordisms are equipped with an orientation and spin structure that restrict 
compatibly to the ends, in the sense that the collar embeddings in Definition \ref{def:lagrangian_cobordism} can be chosen
to preserve the orientations and spin structures.

The basic examples of Lagrangian cobordisms are Lagrangian suspensions and Lagrangian surgeries, which
we recall below.

\begin{figure}
	\centering
	\includegraphics[width = 0.5\textwidth]{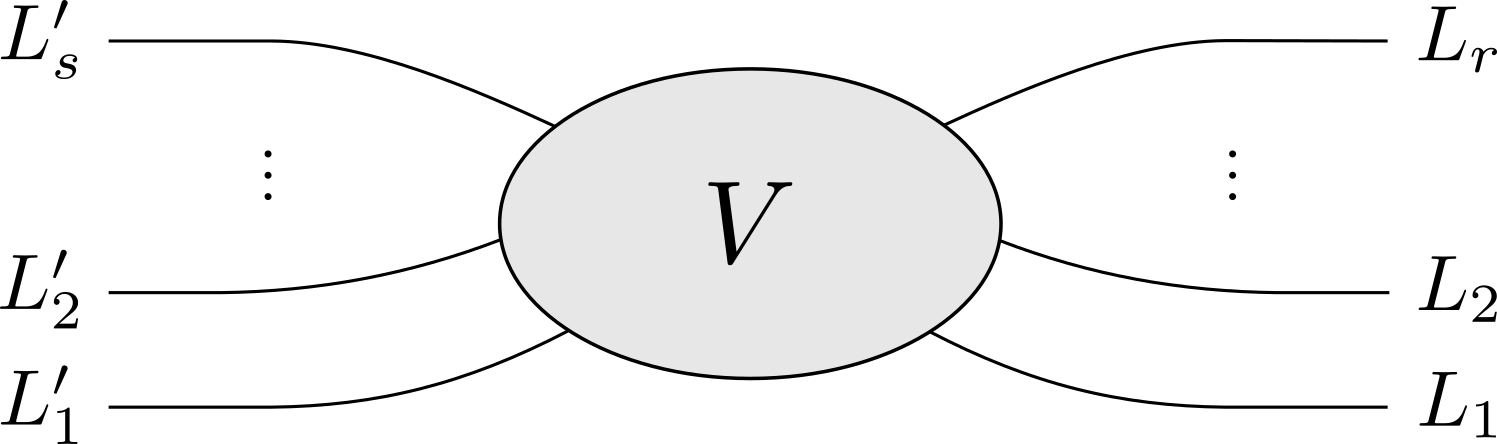}
	\caption{The projection to $\C$ of a cobordism $V: (L_i) \cob (L_j')$.}
\end{figure}

\subsubsection{Lagrangian suspension}
\label{subsection:suspension}

Recall that a Lagrangian regular homotopy $\phi: L \times I \to M$
is called \emph{exact} if for each $t \in I$ the form $\beta_t = \phi_t^* \iota_{\xi_t} \omega$ is exact, where 
$\xi_t$ is the vector field along $\phi$ that generates the homotopy.
A function $H : L \times I \to \R$ such that 
$d H_t = \beta_t$ is called a Hamiltonian generating $\phi$.
We call $\phi$ a \emph{Hamiltonian isotopy} if each map $\phi_t: L \to M$ is an embedding.
Note that in this case $\phi$ extends to an isotopy of Hamiltonian diffeomorphisms of $M$.

To an exact homotopy $\phi$ generated by a Hamiltonian $H$, we can associate a Lagrangian cobordism $\eta: L \times I \to \C \times M$ called the \emph{Lagrangian suspension} of $\phi$.
This cobordism is defined by the following formula:
\[
\eta(x, t) = (t, H_t( \phi_t(x)), \phi_t(x)).
\]
It is easily checked that this is a Lagrangian immersion, and that it is cylindrical provided that $\phi$ is stationary
near the endpoints.

\subsubsection{Lagrangian surgery}
\label{subsection:surgery}

Let $L$ and $L'$ be immersed Lagrangians in general position that intersect at $s \in M$.
One can perform the Lagrangian surgery of $L$ and $L'$ at $s$ to eliminate this intersection point, 
obtaining a new Lagrangian $L \#_s L'$ (which depends on additional choices).
Moreover, there is a Lagrangian cobordism with ends $L$, $L'$ and $L \#_s L'$.
This operation was described by several authors at various levels of generality: see for example the works of Arnold \cite{Arnold},
Audin \cite{Audin_remarques}, Lalonde-Sikorav \cite{LalondeSikorav},
Polterovich \cite{Polterovich91} and Biran-Cornea \cite{BC13}.

The Lagrangian surgery construction will be our main source of Lagrangian cobordisms between curves.
We describe it in more details in Section \ref{subsection:def_surgery}.

\subsubsection{Lagrangian cobordism groups}
\label{subsection:cobordismgroups}

We recall the general definition of the Lagrangian cobordism groups of a symplectic manifold $M$,
following \cite[Section 1.2]{BC14}.
Let $\Lag$ be a class of Lagrangian submanifolds of $M$ and $\Cob$ be a class of Lagrangian cobordisms.
The associated cobordism group $\Gcob(M)$ is the quotient of the free abelian group generated by $\Lag$ by the relations
\[
L_1 + \ldots + L_r = 0
\]
whenever there is a cobordism $(L_1, \ldots, L_r) \cob \emptyset$ that belongs to $\Cob$.

In this paper, we are interested in the Lagrangian cobordism groups of a closed symplectic surface $\Sigma$.
In this case, we take the class $\Lag$ to consist of all oriented, non-contractible, simple curves in $\Sigma$. 
Following the setting of \cite{Abouzaid08}, we assume that all curves 
are equipped with the bounding spin structure on the circle.
Note that
the class $\Lag$ coincides with the objects of $\Fuk(\Sigma)$, as defined in Section \ref{subsection:fukaya_cat_prelim} below.

There are several interesting choices for the class $\Cob$. 
In this paper, we focus on the case where $\Cob$ consists of all
\emph{quasi-exact} cobordisms, which are defined in Section \ref{subsection:unob_cobordisms}.
We denote by $\Gunob(\Sigma)$ the associated cobordism group.

The inverse in $\Gunob(\Sigma)$ is given by reversing the orientation and spin structure of curves.
Note also that any oriented cobordism between oriented curves admits
a spin structure that restricts to the bounding spin structure on its ends (this is easily seen in the case
of the cylinder and the pair of pants; in the general case, consider a pants decomposition of the cobordism).
This means that dropping the assumption that cobordisms carry compatible spin structures
leads to the same cobordism group, so that
we may ignore spin structures in the computation of $\Gunob(\Sigma)$.\footnote{Note that in the context of Theorem \ref{thm:cone_decomposition}
the cone decomposition induced by a cobordism may depend on the choice of spin structure, but
this dependence is not seen at the level of K-theory.}

\subsection{Fukaya categories}
\label{subsection:fukaya_cat_prelim}

In this section, we define the version of the Fukaya category
that will be considered in this paper.
The purpose of this section is mainly to fix our conventions, hence most details will be omitted.
A comprehensive exposition of the material presented here can be found in the book of Seidel \cite{SeidelBook}.
See also \cite{SeidelHMS} and \cite{Abouzaid08} for the specific case of surfaces.

In the following, we assume that $(M, \omega)$ is a closed symplectic manifold which is \emph{symplectically aspherical},
in the sense that $\omega|_{\pi_2(M)} = 0$.

\subsubsection{Coefficients.}

The version of $\Fuk(M)$ that we will consider
is a $\Z_2$-graded $A_{\infty}$-category which is linear over the universal Novikov ring over $\Z$
\[
\Lambda = \Big\{ \sum_{i \geq 0} a_i \, q^{t_i} : a_i \in \Z, t_i \in \R, t_i \nearrow \infty \Big\}.
\]

\subsubsection{Objects.}

Recall that a Lagrangian submanifold $L \subset M$ is \emph{weakly exact} if $\omega|_{\pi_2(M, L)} = 0$. 
The objects of $\Fuk(M)$ are weakly exact, spin, connected, closed Lagrangian submanifolds 
$L \subset M$, equipped with an orientation and spin structure.

\begin{remark}
In the case where $M$ is a surface, we will follow the setting of \cite{Abouzaid08} and restrict
our attention
to the subcategory of $\Fuk(M)$ whose objects are equipped with the \emph{bounding} spin structure on $S^1$.
Note that there is no obstruction to incorporating curves equipped with the trivial spin structure.
However, this would require minor adjustments to some of the statements and proofs from \cite{Abouzaid08}, and we shall not do so here.

\end{remark}

\subsubsection{Morphisms.}

Given two objects $L_0$ and $L_1$, the morphism space $\Hom(L_0, L_1)$ is the Floer cochain
complex $CF(L_0, L_1; H, J)$,
which is defined as follows.

The complex depends on a choice of Floer datum $\DD(L_0, L_1) = (H, J)$, 
which consists of 
a Hamiltonian $H : [0,1] \times M \to \R$ and 
a time-dependent compatible almost complex structure
$(J_t)_{t \in [0,1]}$ on $M$.
We require that $\phi_H^1(L_0)$ be transverse to $L_1$, where $(\phi_H^t)_{t \in [0,1]}$ denotes
the isotopy generated by the Hamiltonian vector field $X_H$.
Moreover, we assume that the pair $(H,J)$ is \emph{regular} in the sense of \cite[Section (8i)]{SeidelBook}.

Let $\chord(L_0, L_1; H)$ be the set of orbits $y: [0,1] \to M$ of the Hamiltonian flow $\phi_H^t$ that satisfy $y(0) \in L_0$ and $y(1) \in L_1$.
Recall that here is a bijection $\chord(L_0, L_1; H) \to \phi_H^1(L_0) \cap L_1$, which associates to $y$ the endpoint $y(1)$.
The \emph{Floer complex} $CF(L_0, L_1; H, J)$ is defined as the free $\Lambda$-module generated by $\chord(L_0, L_1; H)$.

\subsubsection{Grading.}
\label{section:grading}

The Floer complex has a $\Z_2$-grading, which is defined as follows.
Take $y \in \chord(L_0, L_1; H)$ and see it as an intersection point $p \in \phi_{H}^1(L_0) \cap L_1$. 
Let $\lambda_{\operatorname{can}}$ be a canonical short path from $T_p \phi_H^{1}(L_0)$ to $T_p L_1$ in the Lagrangian Grassmannian of $T_p M$.
Then we set $|y| = 0$ if $\lambda_{\operatorname{can}}$ lifts to a path in the oriented Lagrangian Grassmannian of $T_p M$
from $T_p \phi_H^{1}(L_0)$ to $T_p L_1$ (equipped with their given orientations).
Otherwise, we set $|y| = 1$.

Equivalently, we have $|y| = 0$ if the contribution of $p$
to the intersection number $\phi^1_{H}(L_0) \cdot L_1$ is $(-1)^{\frac{n(n+1)}{2}}$. Otherwise, $|y| = 1$.

\subsubsection{Differential.}
\label{prelim:differential}

The Floer differential is defined by counting rigid solutions $u: \R \times [0,1] \to M$ to the \emph{Floer equation}
\begin{equation}
\label{eq:Floer_strip}
\begin{cases}
&\del_s u + J_t(u) (\del_t u  - X_H(t,u)) = 0\\
&u(s, 0) \in L_0, \; u(s,1) \in L_1.
\end{cases}
\end{equation}

For $x, y \in \chord(L_0, L_1; H)$, let $\wtilde{\M}(y, x)$ be
the set of solutions to \eqref{eq:Floer_strip} with asymptotics $\lim_{s \to - \infty} u(s, - ) = y$, $\lim_{s \to \infty} u(s,-) = x$.
For a regular Floer datum $(H, J)$, $\wtilde{\M}(y,x)$ is a manifold,
which may have several connected components of different dimensions.
The dimension of the component containing $u$ is given by the Maslov index $\mu(u)$.
Let $\M(y, x)$ be the quotient of $\wtilde{\M}(y, x)$ by the action of $\R$ by translation in the $s$ variable.

The Floer differential $\mu^1 : CF^{*}(L_0, L_1; H, J) \to CF^{*+1}(L_0, L_1; H,J)$ is
defined by setting
\begin{equation}
\label{eq:definition_Floer_differential}
\mu^1(x) = (-1)^{|x|} \, \sum_{y \in \chord(L_0, L_1; H)} \; \sum_{\substack{u \in \M(y,x) \\ \mu(u) = 1}} 
\eps(u) \, q^{\omega(u)} \, y.
\end{equation}
Here, $\eps(u) = \pm 1$ is a sign which is described in Section \ref{subsubsection:signs}.

\subsubsection{$A_{\infty}$-operations.}
\label{prelim:ainfinity_operations}

More generally, the higher $A_{\infty}$-operations $(\mu^d)_{d=2}^{\infty}$ are defined by counting solutions to
inhomogeneous Cauchy-Riemann equations for maps $u: S \to M$, where $S$ is a disk
with $d+1$ punctures on the boundary.

More precisely, one considers the Deligne-Mumford moduli space $\mathcal{R}^{d+1}$ of
disks with \mbox{$d+1$} boundary punctures $\zeta_0, \ldots, \zeta_{d}$, with the convention that $\zeta_0$ is incoming, 
$\zeta_1, \ldots, \zeta_d$
are outgoing, and the punctures are ordered anti-clockwise along the boundary.
Let $\mathcal{S}^{d+1} \to \mathcal{R}^{d+1}$ be the universal curve over $\mathcal{R}^{d+1}$.
For $r \in \mathcal{R}^{d+1}$, denote by $S_r \subset \mathcal{S}^{d+1}$ the fiber over $r$.
Let $C_0, \ldots, C_d$ be the boundary components of $S_r$ ordered anti-clockwise, with the convention
that $\zeta_0$ is adjacent to $C_0$ and $C_d$.

Fix a consistent universal choice of strip-like ends as in \cite[Section (9g)]{SeidelBook}.
For each family of Lagrangians $(L_0, \ldots, L_d)$, choose a \emph{perturbation datum}
$\DD(L_0, \ldots, L_d) = (\mathbf{\Theta}, \mathbf{J})$. 
Here, $\mathbf{\Theta} = (\Theta^r)_{r \in \mathcal{R}^{d+1}}$ is a smooth family
of $1$-forms  $\Theta^r \in \Omega^1(S^r ; C^{\infty}(M))$.
Moreover, $\mathbf{J} = (J_z)_{z \in \mathcal{S}^{d+1}}$ is a smooth family of compatible almost complex structures on $M$ 
parametrized by $\mathcal{S}^{d+1}$.
The perturbation data is required to restrict to the previously chosen Floer data over the 
strip-like ends. Moreover, it is required to be \emph{consistent} with respect to breaking and gluing of disks 
(see \cite[Section (9i)]{SeidelBook} for the precise definitions).

The perturbation datum $\DD(L_0, \ldots, L_d)$ gives rise to the following inhomogeneous Cauchy-Riemann equation for maps $u: S_r \to M$
\begin{equation}
\label{eq:Floer_polygon}
\begin{cases}
Du + J(z, u) \circ Du \circ j = Y(z,u) + J(z,u) \circ Y(z,u) \circ j\\
u(C_k) \subset L_k.\\
\end{cases}
\end{equation}
Here, $Y \in \Omega^1(S_r; C^{\infty}(TM))$ is the $1$-form with values in Hamiltonian vector fields of $M$ induced by $\Theta^{r}$.
Moreover, $j$ denotes the complex structure of $S_r$.

Fix orbits $y_0 \in \chord(L_0, L_d)$ and $y_k \in \chord(L_{k-1}, L_k)$ for $k = 1, \ldots, d$.
Let $\M^{d+1}(y_0, \ldots, y_{d})$ be the moduli space of pairs $(r,u)$, where $r \in \mathcal{R}^{d+1}$ and 
$u: S_r \to M$ is a solution of \eqref{eq:Floer_polygon} which is asymptotic to
$y_k$ at the puncture $\zeta_k$ in strip-like coordinates.

Under suitable regularity assumptions, $\M^{d+1}(y_0, \ldots, y_{d})$ is a manifold 
whose local dimension near $u$ is $\mu(u) + d - 2$, where $\mu$ is the Maslov index.
Write $\M_0^{d+1}( y_0, \ldots,  y_{d})$
for the $0$-dimensional part of $\M^{d+1}(y_0, \ldots, y_{d})$.
The operation
\[
\mu^d: \Hom(L_{d-1}, L_{d}) \tens \cdots \tens \Hom(L_0, L_1) \to \Hom(L_0, L_d)[2-d]
\]
is now defined by setting
\begin{equation}
\label{eq:definition_mu}
\mu^d(y_d, \ldots, y_1) = (-1)^{\dagger} \; \sum_{y_0 \in \chord(L_0, L_d)}  \; \sum_{u \in \M_0^{d+1}(y_0, y_1, \ldots, y_{d})} 
\eps(u) \, q^{\omega(u)} \, y_0.
\end{equation}
Here, $\eps(u) = \pm 1$ is a sign which is described in Section \ref{subsubsection:signs}. 
Moreover, $(-1)^{\dagger}$ is an additional sign given by $\dagger = \sum_{k=1}^{d} k |y_k|$.

\subsubsection{Orientations and signs.}
\label{subsubsection:signs}

The signs appearing in the definition of the operations $\mu^d$ depend 
on choices of orientations of the moduli spaces of inhomogeneous polygons $\M^{d+1}(y_0, \ldots, y_d)$.
We orient these moduli spaces using the arguments of \cite[Chapter 8]{FOOO} (see also \cite[Section 11]{SeidelBook} for a closely related approach).

The orientations of the moduli spaces depend on the orientation and spin structures of the Lagrangians involved,
as well as some extra data which we now describe.
Fix a pair of objects $(L_0, L_1)$ and a generator of the Floer complex $y \in \chord(L_0, L_1; H)$. 
Up to replacing $L_0$ by $\phi^1_H(L_0)$ in what follows, 
we may assume that $y$ is a constant trajectory at an intersection point $p \in L_0 \cap L_1$.

We introduce the following notations. 
Let $\GrassOr(T_p M)$ be the oriented Lagrangian Grassmannian of $T_p M$. 
Let $\lambda$ be a path in $\GrassOr(T_p M)$
with $\lambda(0) = T_p L_0$ and $\lambda(1) = T_p L_1$.
The path $\lambda$ defines an \emph{orientation operator} $D_{\lambda}$ as follows.
Let $S_{+}$ be the unit disk with one incoming boundary puncture.
Let $E \to S_{+}$ be the trivial vector bundle with fiber $T_p M$.
The path $\lambda$ defines a totally real subbundle $F \subset E|_{\bdry S_{+}}$.
Define $D_{\lambda}$ as the standard Cauchy-Riemann operator on $E$ with boundary conditions given by $F$.
Denote by $\wtilde{\lambda}$ the vector bundle over $[0,1]$ given by
\[
\wtilde{\lambda} = \bigcup_{t \in [0,1]} \{ t \} \times \lambda(t).
\]

For each pair $(L_0, L_1)$ and for each $y \in \chord(L_0, L_1; H)$, we fix the following \emph{orientation data}:
\begin{enumerate}
	\item A path $\lambda$ in $\GrassOr(T_p M)$ from $T_p L_0$ to $T_p L_1$.
	\item An orientation of the determinant line $\det(D_{\lambda})$.
	\item A spin structure on $\wtilde{\lambda}$ that extends the given spin structures of $T_p L_0$ and $T_p L_1$.
\end{enumerate}

By standard gluing arguments (see for example Chapter 8 of \cite{FOOO} or Section 12 of \cite{SeidelBook}),
the above data determine orientations of the 
moduli spaces $\M^{d+1}(y_0, \ldots, y_d)$.
In particular, each isolated element $u \in  \M^{d+1}_0(y_0, \ldots, y_d)$ carries a sign
$\eps(u) = \pm 1$ which determines its contribution to $\mu^d$ (see Equation \eqref{eq:definition_Floer_differential} and Equation \eqref{eq:definition_mu}).

\begin{remark}

In the case where $M$ is a surface, the signs appearing in the definition of $\mu^d$ can also be defined
in a purely combinatorial way. This is described in \cite[Section 7]{SeidelHMS} and is also the definition used in \cite{Abouzaid08}.

\begin{figure}
	\centering
	\includegraphics[width=0.6\textwidth]{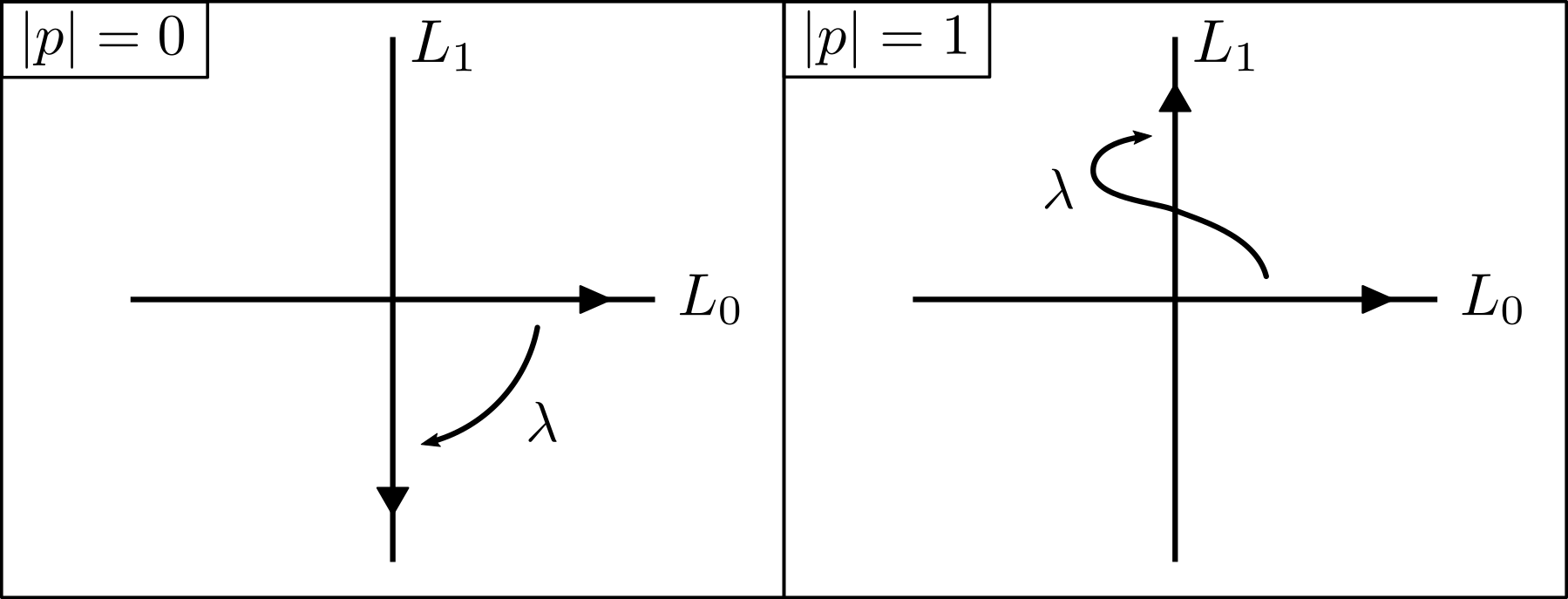}
	\caption{Choices of orientation paths for intersection points between curves.}
	\label{fig:orientation_path}
\end{figure}

The combinatorial definition of the signs in \cite{SeidelHMS} is related to the one described above in the following way.
In \cite{SeidelHMS}, 
each Lagrangian is equipped with a marked point and a trivialization of the spin structure on the complement
of the marked point.
These extra choices determine canonical choices of the orientation data (1)--(3) as follows.
For an intersection point $p \in L_0 \cap L_1$, one takes the path $\lambda$ to be as in Figure \ref{fig:orientation_path}.
More precisely, if $|p|=0$ then $\lambda$ is the canonical short path from $L_0$ to $L_1$.
If $|p|=1$, one takes the reverse of the canonical short path from $L_1$ to $L_0$ and perturbs it to add a positive crossing.
The spin structure on $\wtilde{\lambda}$ is uniquely determined (up to isomorphism relative to the boundary) by the trivializations of the spin structures of $T_p L_0$ and $T_p L_1$.
Finally, $\det(D_{\lambda})$ is oriented as follows.
If $|p|=0$, then $D_{\lambda}$ is invertible, so that $\det(D_{\lambda})$ is canonically oriented.
If $|p|=1$, there is an isomorphism
$
\det(D_{\lambda}) \iso T_p L_1
$ 
given by the spectral flow (see \cite[Lemma 11.11]{SeidelBook}).
As $T_p L_1$ is oriented, this determines an orientation of $\det(D_{\lambda})$.

With these compatible choices of data, it can be checked that the combinatorial definition of the signs given in \cite{SeidelHMS} agrees with the one obtained from gluing arguments, at least up to a global sign (see the discussion following Example 13.5 in \cite{SeidelBook}).

\end{remark}

\subsubsection{Invariance.}

The Fukaya category depends on several auxiliary choices, such as choices of strip-like ends, Floer data, perturbation data and orientation data.
It follows from the constructions in \cite[Section (10a)]{SeidelBook} that the resulting category is independent of these choices up to quasi-isomorphism. With this in mind, we will omit the choice of auxiliary data from the notation and simply write $\Fuk(M)$ for any
of these categories.

\subsubsection{Derived Fukaya categories and $K$-theory.}

We consider the following model for the derived Fukaya category $\DFuk(M)$.
First, consider the Yoneda embedding $\mathcal{Y}: \Fuk(M) \to \Mod( \Fuk(M) )$. 
Then, let $\Fuk(M)^{\wedge}$ be the triangulated closure of the image of $\mathcal{Y}$ inside $\Mod(\Fuk(M))$.
Finally, define $\DFuk(M) = H^0(\Fuk(M)^{\wedge})$.
Note that we do not complete with respect to idempotents.

The derived category $\DFuk(M)$ is a triangulated category (in the usual sense).
Geometrically, the shift functor is realized by the operation of reversing the orientation and spin structure of objects of $\Fuk(M)$.

We let $\Kgrp(M)$ be the Grothendieck group of $\DFuk(M)$.
Recall that for a triangulated category $\mathcal{C}$, the Grothendieck group $K_0 \mathcal{C}$
is the free abelian group generated by the objects of $\mathcal{C}$, quotiented by the relations $A + B - C = 0$
whenever there is an exact triangle $A \to B \to C$ in $\mathcal{C}$.

%% file: sections/fukaya_cob.tex
In this section, we construct a Fukaya category $\FukCob(\C \times M)$ whose objects are quasi-exact cobordisms,
and use it to prove Theorem \ref{thm:cone_decomposition}. 
The proof closely follows the scheme formulated in \cite{BC14}, and its extension to embedded quasi-exact cobordisms in \cite{BCS}. 
As a complete proof is outside of the scope of this paper,
we will content ourselves with describing the 
small modifications that are needed to adapt the framework developed in \cite{BC14} and \cite{BCS} to the present case.
The main differences are as follows:

\begin{enumerate}[label = (\roman*)]
	\item We consider Lagrangian cobordisms that may have self-intersections. 
	In this case, holomorphic curves that appear through bubbling may have corners at self-intersection points.
	We show that ruling out holomorphic curves with at most $1$ corner is sufficient to ensure 
	the compactness
	of the moduli spaces relevant to the definition of $\FukCob(\C \times M)$.	
	
	\item The Fukaya category defined in \cite{BC14} is linear over $\Z_2$ and ungraded.
	The version of the Fukaya category that we consider is linear over the Novikov ring over $\Z$ and is $\Z_2$-graded.
	In our setting, the adjustments needed to deal with signs and gradings were described by Haug \cite[Section 4]{Haug15}.
\end{enumerate}

Following the setting of \cite{Abouzaid08} and \cite{Haug15}, we use cohomological conventions 
for complexes, which leads to some superficial differences with the homological conventions used in \cite{BC14}.

\subsection{Holomorphic maps with boundary on immersed Lagrangians}
\label{subsection:holomorphic_disks}

The Floer theory of Lagrangian immersions was originally developed by Akaho \cite{Akaho} and Akaho-Joyce \cite{AkahoJoyce}.
The main difference with the embedded case is that holomorphic curves
with boundary on an immersed Lagrangian
may have branch jumps at self-intersection points of the Lagrangian.
To describe this behaviour, holomorphic curves are equipped with the data of \emph{boundary lifts}
that record the branch jump type.

We recall the following definitions from \cite{AkahoJoyce}. 
Let $S$ be a nodal disk, that is a compact nodal Riemann surface 
of genus $0$ with $1$ boundary component. 
In order to define boundary lifts, $S$ is equipped with a boundary parametrization, i.e. a continuous orientation-preserving map 
$\ell: S^1 \to \bdry S$
such that 
\begin{itemize}
	\item the preimage of a boundary node consists of two points,
	\item the preimage of a smooth point of $\bdry S$ consists of one point.
\end{itemize}
Note that $\ell$ is unique up to reparametrization.

Let $N$ be a symplectic manifold and let $\iota: L \looparrowright N$ be a Lagrangian immersion of a manifold $L$ 
(which we do not assume to be connected or compact). 
Fix a compatible almost complex structure $J$ on $N$.

\begin{definition}
\label{def:stable_disk_with_corners}
A (genus $0$) \Def{$J$-holomorphic map with corners} with boundary on \mbox{$\iota: L \looparrowright N$}
is a tuple $(S, \ell, u, \Delta, \overline{u})$, where
\begin{enumerate}[font=\normalfont]
	\item $S$ is a nodal disk with boundary parametrization $\ell$,
	\item $u: (S, \bdry S) \to (N, \iota(L))$ is a continuous map,
	\item $\Delta \subset \bdry S$ is a finite set of marked points distinct from the nodes,
	\item $\overline{u} : S^1 \setminus \ell^{-1}(\Delta) \to L$ is a continuous map,
\end{enumerate}
which satisfies the following conditions:
\begin{enumerate}[label = (\roman*), font=\normalfont]
	\item $u$ is $J$-holomorphic on $S \setminus \Delta$,
	\item $u$ has finite energy, i.e. $\int_{S \setminus \Delta} u^* \omega < \infty$,
	\item $u \circ \ell = \iota \circ \overline{u}$ on $S^1 \setminus \ell^{-1}(\Delta)$,
	\item for each $\zeta \in \ell^{-1}(\Delta)$, the one-sided limits of $\overline{u}$ at $\zeta$ are distinct.
\end{enumerate}
\end{definition}

The elements of $\Delta$ are called the \emph{corner points} of $u$ and the map $\overline{u}$ is called the \emph{boundary lift} of $u$. 
Condition (iv) in Definition \ref{def:stable_disk_with_corners} means that the boundary of $u$ has a branch jump at each corner point. 
Note that $u$ cannot have branch jumps at the nodes of $S$. 
However, the individual components of $u$, seen as maps defined on a disk, may have branch jumps at the nodal points.

We introduce the following terminology. We denote by $\D$ the closed unit disk in $\C$.

\begin{definition}
\label{def:disks+teardrops}
\hfill
\begin{enumerate}[label = (\roman*), font=\normalfont]
	\item A \Def{$J$-holomorphic disk} is a $J$-holomorphic map with domain $\D$ and no corners, i.e. $\Delta = \emptyset$.
	\item A \Def{$J$-holomorphic teardrop} is a $J$-holomorphic map with domain $\D$ and $1$ corner.
\end{enumerate}
\end{definition}

A $J$-holomorphic map $u$ is \Def{stable} if its automorphism group is finite.
Equivalently, if $S_{\alpha}$ is a component of $S$ such that $u|_{S_{\alpha}}$ is constant, then 
$S_{\alpha}$ carries at least $3$ special points (with the convention that interior nodes count twice).

\subsection{Definition of quasi-exact cobordisms}
\label{subsection:unob_cobordisms}

Given a Lagrangian cobordism $V \subset [0,1] \times \R \times M$, we can
extend its ends towards $\pm \infty$ by gluing appropriate cylinders of the form $(-\infty, 0] \times \{ k \} \times L$
or $[1, \infty) \times \{k \} \times L$.
This defines a non-compact Lagrangian $\overline{V}$, called the extension of $V$.
When discussing the Floer theory of $V$, we make the convention that $V$ is always to be replaced by its extension.

Recall that we write $\wtilde{M} = \C \times M$. 
We say that a compatible almost complex structure $J$ on $\wtilde{M}$
is \Def{admissible} if there is a compact set $K \subset (0,1) \times \R$ such that the projection $\pi_{\C}: \wtilde{M} \to \C$ 
is $(J, i)$-holomorphic outside $K \times M$.

We now define quasi-exact cobordisms,
which generalize to the immersed case Definition 4.2 of \cite{BCS} and 
the unobstructed Lagrangian brane cobordisms of \cite{SheridanSmith}. 

\begin{definition}[Quasi-exact cobordisms]
\label{def:unob_cobordisms}
Let $V \subset \wtilde{M}$ be an immersed Lagrangian cobordism with generic self-intersections
and $J$ be an admissible almost complex structure.
We say that the pair $(V, J)$ is \Def{quasi-exact} if $V$ does not bound non-constant $J$-holomorphic disks and teardrops.
\end{definition}

\begin{remark}
\label{rmk:def_unob_cob}
\hfill
\begin{enumerate}[label = (\roman*)]
	\item By definition, quasi-exact cobordisms have embedded ends. 
	It is possible in some situations to work with cobordisms having immersed ends; see for example \cite{Pictionary} for an implementation. 
	However, doing so would involve additional technical difficulties and is not necessary for our purpose.
	
	One drawback of this restriction is that surgeries that produce immersed Lagrangians are \emph{never} unobstructed in the sense of Definition \ref{def:unob_cobordisms}. 
	However, concatenations of such surgeries may give rise to unobstructed cobordisms after perturbing their self-intersection locus; a special case of this procedure is explained in Section \ref{subsection:concatenations}.

	\item It follows from the open mapping theorem that the concatenation
	of two quasi-exact cobordisms along two matching ends is quasi-exact (for a suitable choice of $J$); 
	see Proposition 6.2 of \cite{BCS} for a detailed proof.
	In contrast, classes of cobordisms satisfying topological constraints such as weak exactness or monotonicity
	are typically not closed under concatenations.
\end{enumerate}
\end{remark}

\begin{remark}[Comparison with the definition in \cite{Perrier19}]
\label{rmk:def_unob_perrier}
The class of unobstructed cobordisms considered in \cite{Perrier19} consists
of cobordisms that do not bound any \emph{continuous} tear\-drops.
This definition is thus closer to the definition of a topologically unobstructed cobordism which we introduce
in Section \ref{section:top_unob_cobordisms}.

There are two problems with the definition used in \cite{Perrier19}.
The first is that there is no condition on disks with boundary on the cobordism, 
which poses technical issues in the setup of Floer theory and the proofs of cone decompositions.
These issues are not addressed in \cite{Perrier19}.
The second problem is that the non-existence of continuous teardrops is generally not preserved under concatenations of cobordisms.
The use of holomorphic maps in Definition \ref{def:unob_cobordisms} aims to fix this issue.
Indeed, the class of quasi-exact cobordisms is well-behaved with respect
to concatenations; see Remark \ref{rmk:def_unob_cob} and Proposition \ref{prop:concatenation_unobstructed}.
\end{remark}

\begin{lemma}
\label{lemma:no_stable_polygons}
Let $(V,J)$ be a quasi-exact cobordism. 
Then $V$ does not bound stable $J$-holomorphic maps with at most $1$ corner.
\end{lemma}
\begin{proof}
Suppose that $V$ bounds a stable $J$-holomorphic map $u$ with at most $1$ corner.
If $u$ has one component, then this contradicts that $(V,J)$ is quasi-exact.
If $u$ has more than one component,
then there is a component $\D_{\alpha}$ that has $1$ boundary node and no marked point.
By stability, the restriction $u_{\alpha} = u|_{\D_{\alpha}}$ is non-constant.
Moreover, it can only have a corner at the unique nodal point of $\D_{\alpha}$.
Hence $u_{\alpha}$ is a non-constant $J$-holomorphic disk with at most $1$ corner, which is again a contradiction.
\end{proof}

For the purpose of defining suitable classes of Floer and perturbation data, we will
need a version of Lemma \ref{lemma:no_stable_polygons} that holds for a larger class of almost complex structures, which we now define.

\begin{definition}
\label{def:adapted_J}
Let $(V, J_V)$ be a quasi-exact cobordism.
A compatible almost complex structure $J$ on $\wtilde{M}$ is \Def{adapted} to $(V, J_V)$ if
\begin{enumerate}[label = (\roman*), font=\normalfont]
	\item $J = J_V$ on $[0,1] \times \R \times M$.
	\item The projection $\pi_{\C}: \wtilde{M} \to \C$ is $(J, i)$-holomorphic on $U \times M$, where $U$
	is a neighborhood of $\pi_{\C}(V)$.
\end{enumerate}
\end{definition}

\begin{lemma}
\label{lemma:no_stable_polygons_bis}
Suppose that $(V, J_V)$ is a quasi-exact cobordism and that $J$ is adapted to $(V, J_V)$.
Then $V$ does not bound stable $J$-holomorphic maps with at most $1$ corner.
\end{lemma}
\begin{proof}
We prove the any $J$-holomorphic map $u: \D \to \wtilde{M}$ with boundary on $V$ is actually $J_V$-holomorphic.
Let $v = \pi_{\C} \circ u$. 
By the open mapping theorem for holomorphic functions (see Proposition 3.3.1 of \cite{BC14} for the precise version used here), 
$v$ cannot meet the unbounded components of $\C \setminus \pi_{\C}(V)$.
Therefore, either $v(\D) \subset [0,1] \times \R$, or $v(\Int \D)$ meets \mbox{$\pi_{\C}(V) \cap( [0,1] \times \R)^c$}. 
In the first case, it follows from condition (i) in Definition \ref{def:adapted_J} that $u$ is $J_V$-holomorphic.
In the second case, it follows from condition (ii) and the open mapping theorem that $v$ is constant, so that
$u$ is also constant.
\end{proof}

\subsection{Definition of $\FukCob(\wtilde{M})$}
\label{subsection:def_category}

We now define the Fukaya category of quasi-exact cobordisms $\FukCob(\wtilde{M})$,
closely following the construction described in \cite{BC14}.
We shall only describe the modifications required to adapt this construction to the present case,
and refer to \cite{BC14} for further details.

\subsubsection{Objects}
\label{subsection:objects}

An object of $\FukCob(\wtilde{M})$ 
is a quasi-exact Lagrangian cobordism $(V,J)$ with weakly exact ends,
equipped with an orientation and spin structure on $V$.
For notational convenience, we will most of the time
write such an object by only specifying the cobordism $V$ and keeping the other parts of the data implicit.

\begin{remark}
For a given cobordism $V$, each choice of $J$ such that $(V,J)$ is quasi-exact defines a different object of $\FukCob(\wtilde{M})$.
A priori, these objects need not be quasi-isomorphic. 
As a consequence, the cone decomposition of Theorem \ref{thm:cone_decomposition} may depend on the choice of $J$.
Note, however, that the induced relation on the level of $K$-theory does not depend on $J$.
\end{remark}

\subsubsection{Floer data}
\label{subsection:floer_data}

Recall the following notion from \cite{BC14}.
To define the class of admissible Floer data, we fix a \emph{profile function} $h: \R^2 \to \R$,
whose role is to specify the form of the Hamiltonian perturbations near infinity.
The definition is the same as that of \cite[Section 3.2]{BC14},
except that the sign of $h$ is switched, i.e. $h$ is such that $-h$ satisfies the properties i-iv on p.1761 of \cite{BC14}.
The switch in the sign of $h$ is due to our use of cohomology rather than homology.

For each pair of objects $(V_0, V_1)$, 
we fix a choice of Floer datum $\data(V_0, V_1) = (\wtilde{H}, \wtilde{J})$, where $\wtilde{H}: [0,1] \times \wtilde{M} \to \R$
is a Hamiltonian and $\wtilde{J}(t)$ is a time-dependent compatible almost complex structure on $\wtilde{M}$.
The Floer data are required to satisfy the following conditions:

\begin{enumerate}[label = (\roman*), font=\normalfont]
\item $\phi^1_{\tilde{H}}(V_0)$ is transverse to $V_1$ and their intersections are not double points.
\item There exists a compact set $K \subset (-\frac{5}{4}, \frac{9}{4}) \times \R \subset \C$ and a Hamiltonian 
\mbox{$H: [0,1] \times M \to \R$} such that $\wtilde{H}(t, x+iy, p) = h(x,y) + H(t, p)$ for $x + i y \in \C \setminus K$.
\item For all $t \in [0,1]$, the projection $\pi_{\C}: \wtilde{M} \to \C$ is $( \wtilde{J}(t), (\phi_h^t)_* i )$-holomorphic outside $K \times M$.
\item[(iv*)] For $k \in \{0, 1 \}$, $\wtilde{J}(k)$ is adapted to $J_{k}$ in the sense of Definition \ref{def:adapted_J}.
\end{enumerate}

Note that conditions (ii)--(iii) are the same as in \cite{BC14} (see p.1792), and
that condition (i) is also the same except for the restriction concerning the double points.
The purpose of condition (iv*) is to ensure that there is no bubbling of holomorphic disks with boundary on $V_k$.
Note that condition (iv*) does not interfere with condition (iii) since the profile function satisfies
$(\phi_h^t)_* i  = i$ over $[0,1] \times \R$ and
over the projections of the ends of the cobordisms.
We remark however that it is generally not possible to impose that $\wtilde{J}(k)$ 
satisfies condition (i) of Definition \ref{def:unob_cobordisms}.

It follows from condition (ii) and the definition of the profile function $h$ that $\phi^1_{\tilde{H}}(V_0)$ and $V_1$ are distinct at infinity.
Combined with condition (i), this implies that the set of Hamiltonian chords $\chord(V_0, V_1; \wtilde{H})$ is finite.

\subsubsection{Perturbation data}
\label{subsection:perturbation_data}

To define the admissible class of perturbation data, we fix the following choices.
First, fix a consistent universal choice of strip-like ends as in \cite[Section (9g)]{SeidelBook}.
Secondly, fix a family of \emph{transition functions} $\mathbf{a} : \mathcal{S}^{d+1} \to [0,1]$ for $d \geq 1$ as in Section 3.1 of \cite{BC14}.
These transition functions are required to satisfy several conditions; we refer to \cite{BC14} for the precise definitions.

For each collection of objects $(V_0, \ldots, V_d)$, we fix a perturbation datum $\DD(V_0, \ldots, V_d) = (\mathbf{\Theta}, \mathbf{J})$.
Here, $\mathbf{\Theta} = (\Theta^r)_{r \in \mathcal{R}^{d+1}}$ is a smooth family of $1$-forms $\Theta^r \in \Omega^1(S_r; C^{\infty}(\wtilde{M}))$
and $\mathbf{J} = (J_z)_{z \in \mathcal{S}^{d+1}}$ is a smooth family of compatible almost complex structures on $\wtilde{M}$.

The perturbation data are required to satisfy the conditions (i)--(iii) stated
on pp.1763-1764 of \cite{BC14}.
Moreover, in order to rule out bubbling off of holomorphic disks, we impose the following additional condition:
\begin{enumerate}
\item[(iv*)] For $0 \leq k \leq d$ and $z \in C_k$, $J_z$ is adapted to $J_{V_k}$ in the sense of Definition \ref{def:adapted_J}.
\end{enumerate}

As usual, the perturbation data are required to be consistent with breaking and gluing as in \cite[Section (9i)]{SeidelBook}.

\subsubsection{Transversality and compactness}

Having made choices of Floer and perturbation data as described above,
the moduli spaces of Floer polygons are defined as in Section \ref{prelim:differential} and Section \ref{prelim:ainfinity_operations}.
The main point here is that Floer polygons are assumed to have \emph{no corners},
i.e. they are required to be smooth up to the boundary.

Before going further, we need to justify that the Floer and perturbation data can be chosen to
achieve the transversality and compactness up to breaking of the
moduli spaces of Floer polygons.

For transversality, the argument is the same as in Section 4.2 of \cite{BCS}.
The main point is that the additional constraints imposed on the Floer and perturbation data (i.e. condition (iv*) above)
only concern
their restriction to the boundary of
$S_r$, so that arbitrary perturbations are allowed in the interior of $S_r$.

We now address the compactness issues.
The first point is that we must have $C^0$ bounds for Floer polygons of bounded energy. 
In our setting, the proof is the same as in \cite[Section 3.3]{BC14}.
The only difference is that there is no uniform bound on the energy of Floer polygons;
instead the areas of curves are encoded using Novikov coefficients.

The second point is that we must show that no bubbling of holomorphic disks
can occur for sequences of Floer polygons with bounded energy.
This is a consequence of Lemma \ref{lemma:no_stable_polygons_bis}.
Indeed, since Floer polygons are assumed to have no corners, it follows from the Gromov compactness
theorem for curves with boundary lifts (see \cite{IS})
that a sequence of such polygons with bounded energy
converges to a stable map with no corners.
If the limit stable map contains a bubble tree $u$
with boundary on a cobordism $V_k$, then by condition (iv*) in the definition of the Floer and perturbation data
$u$ must be $J$-holomorphic for an almost complex structure $J$ which is adapted to the base structure $J_{k}$ associated to $V_k$.
Moreover, since the limit map has no corners, $u$ has at most $1$ corner.
By restricting to a subtree if necessary, we can assume that $u$ is stable.
Hence, $V_k$ bounds a stable $J$-holomorphic map with at most $1$ corner,
which contradicts Lemma \ref{lemma:no_stable_polygons_bis}.

\subsubsection{Summing it up}

Having made the choices described above, the definition of $\FukCob(\wtilde{M})$ now follows the same
recipe as the definition of $\Fuk(M)$ outlined in Section \ref{subsection:fukaya_cat_prelim}.
For objects $V_0$ and $V_1$, the morphism space $\Hom(V_0, V_1)$ is the
complex generated by Hamiltonian orbits for the chosen Floer datum.
The grading of $\Hom(V_0, V_1)$ is defined as in Section \ref{section:grading}.
The $A_{\infty}$-operations $\mu^d$ are defined by counting rigid Floer polygons asymptotic to Hamiltonian orbits.
The signs appearing in the definition of $\mu^d$ are
defined by fixing orientation data as in Section \ref{subsubsection:signs}.

\subsection{End of the proof of Theorem \ref{thm:cone_decomposition}}

In the previous sections, we described the technical adjustments
that are needed to extend the theory
developed in \cite{BC14, BCS} to immersed quasi-exact cobordisms.
Having made these modifications,
the proof of Theorem \ref{thm:cone_decomposition}
now follows the same arguments as the proof Theorem A of \cite{BC14}.
The only missing ingredient 
is the verification that the $A_{\infty}$-functors defined in \cite{BC14}
are compatible with gradings and signs.
In the present setting,
these verifications were carried out by Haug, see \cite[Section 4]{Haug15}.

%% file: sections/top_unob.tex
In this section, we introduce a class of cobordisms that are
\emph{topologically unobstructed},
in the sense that they do not bound any homotopically non-trivial continuous disks and teardrops.
This condition will play an important role in this paper because it is easier to check than unobstructedness in the sense of 
Section \ref{section:fukaya_cob}.
Moreover, as we will see in subsequent sections,
all the cobordisms appearing in the computation of $\Gunob(\Sigma)$ can be
chosen to either satisfy this stronger property, or at least be concatenations of topologically unobstructed cobordisms.

In Section \ref{subsection:def_top_unob}, we define topologically unobstructed cobordisms.
In Section \ref{subsection:obstruction_suspensions}, we characterize topological obstruction for Lagrangian suspensions.
In Section \ref{subsection:stability}, we show
that topological unobstructedness is an open condition in the $C^1$ topology.

In Section \ref{subsection:concatenations},
we consider concatenations of cobordisms along possibly immersed ends.
There are two issues with
such concatenations: the first is that their
self-intersections are not generic and therefore must be perturbed;
the second is that topological unobstructedness is generally not preserved under
concatenations.
The main result of Section \ref{subsection:concatenations} is that 
appropriate concatenations of topologically unobstructed cobordisms
become unobstructed in the sense of Section \ref{section:fukaya_cob} after a suitable perturbation.

\subsection{Definitions}
\label{subsection:def_top_unob}

In Section \ref{subsection:holomorphic_disks}, we considered holomorphic disks with corners with boundary on Lagrangian immersions.
In this section, we consider \emph{continuous} polygons with boundary on immersions,
which are defined in a similar way. 

\begin{definition}
\label{def:continuous_polygon}
Let $\iota: L \to M$ be an immersion.
A \Def{continuous polygon} with boundary on $\iota$ consists of a continuous map $u : (\D, \bdry \D) \to (M, \iota(L))$,
a finite set $\Delta \subset \bdry \D$
and a continuous map $\overline{u}: \bdry \D \setminus \Delta \to L$
such that 
\begin{enumerate}[label = (\roman*), font=\normalfont]
	\item $\iota \circ \overline{u} = u$ over $\bdry \D \setminus \Delta$,
	\item for each $\zeta \in \Delta$, the one-sided limits of $\overline{u}$ at $\zeta$ are distinct.
\end{enumerate}
\end{definition}
In the rest of this paper, by a \emph{polygon} we will always mean a \emph{continuous polygon} (not necessarily holomorphic).
As before, we call a polygon a \emph{disk} if it has $0$ corners, and a \emph{teardrop} if it has $1$ corner.

Note that it follows from the definition that an immersion $\iota: L \to M$ 
bounds a teardrop if and only if there is a path $\gamma$ in $L$ with
distinct endpoints such that $\iota \circ \gamma$ is a contractible loop in $M$.

We will need to consider homotopy classes of disks with boundary on an immersion.
For this purpose, recall that to a based map of based spaces $f: X \to Y$, 
we can associate relative homotopy groups $\pi_n(f)$,
which are a straightforward generalization of the relative homotopy groups of a pair.
The elements of $\pi_n(f)$ are
homotopy classes of diagrams of based maps
\begin{equation*}
\begin{tikzcd}
\D^n \arrow{r} & Y \\
\bdry \D^n \arrow{u}{\text{inc}} \arrow{r} &X \arrow{u}[swap]{f}
\end{tikzcd}
\end{equation*}
where the homotopies are also required to factor through $f$ over the boundary of the disk.
Equivalently, we may define $\pi_n(f)$ as the relative homotopy group $\pi_n(M_f, X)$,
where $M_f$ is the mapping cylinder of $f$ and $X$ is embedded in $M_f$ in the usual way.
There is then a long exact sequence of homotopy groups
\begin{equation}
\label{eq:LES_relative_homotopy}
\begin{tikzcd}
\ldots \arrow{r} &\pi_n(X) \arrow{r}{f_*} &\pi_n(Y) \arrow{r} & \pi_n(f) \arrow{r} &\pi_{n-1}(X) \arrow{r} &\ldots
\end{tikzcd}
\end{equation}
\begin{definition}
\label{def:incompressible}
A map $f: X \to Y$ is called \Def{incompressible} if the
maps \mbox{$f_*: \pi_1(X) \to \pi_1(Y)$} induced by $f$ are injective (for any choices of basepoints).

A subspace $X \subset Y$ is called \Def{incompressible} if the inclusion $X \to Y$ is incompressible.
\end{definition}

Note that by the long exact sequence \eqref{eq:LES_relative_homotopy}, the induced map $f_*: \pi_1(X) \to \pi_1(Y)$ 
is injective if and only 
the boundary map $\pi_2(f) \to \pi_1(X)$ vanishes.

We can now define topologically unobstructed Lagrangian immersions.
\begin{definition}
\label{def:top_unob_lagrangian}
A Lagrangian immersion $\iota: L \to M$ is \Def{topologically unobstructed} if
$\iota$ is incompressible and does not bound continuous teardrops.
\end{definition}

\begin{remark}
\hfill
\begin{enumerate}[label = (\roman*)]
	\item We emphasize that in Definition \ref{def:top_unob_lagrangian} we do not
assume that the immersion $\iota$ has generic self-intersections.
The reason is that it will be convenient for us to allow cobordisms 
with immersed ends to belong to the class of topologically unobstructed cobordisms.

	\item The cases of interest for this paper are $M = \Sigma$ and $M = \C \times \Sigma$,
	where $\Sigma$ is a closed surface of genus $g \geq 2$.
	In these cases we have $\pi_2(M) = 0$, hence incompressibility of an immersion $\iota$ is equivalent to $\pi_2(\iota) = 0$.
	
	In the case of a surface $\Sigma$, it follows from Lemma 2.2 of \cite{Abouzaid08} that an immersion $S^1 \to \Sigma$
	is topologically unobstructed if and only if it is \emph{unobstructed} in the sense of Definition 2.1 of \cite{Abouzaid08},
	that is if its lifts to the universal cover $\wtilde{\Sigma}$ are proper embeddings. See Lemma \ref{lemma:criterion_top_unob} for
	a generalization.
	Note also that, since $\pi_1(\Sigma)$ has no torsion, an immersion $S^1 \to \Sigma$ is incompressible if and only if it is non-contractible.
\end{enumerate}
\end{remark}

\begin{proposition}
Let $M$ be a symplectically aspherical manifold.
Let $\iota_V: V \to \C \times M$ be a topologically unobstructed Lagrangian cobordism
with embedded ends and generic self-intersections.
Then $(V,J)$ is quasi-exact for any admissible $J$.
\end{proposition}
\begin{proof}
By assumption, $V$ does not bound any continuous teardrops.
Moreover, since the map $\pi_2(\iota_V) \to \pi_1(V)$ is trivial and
$M$ is symplectically aspherical, disks with boundary on $V$ have zero symplectic area.
Hence, there are no non-constant $J$-holomorphic disks with boundary on $V$.
\end{proof}

We now give a useful reformulation of topological unobstructedness,
which is a straightforward generalization of Lemma 2.2 of \cite{Abouzaid08}.
Let $\iota: L \to M$ be an immersion and denote by $\pi: \wtilde{L} \to L $ and $\rho: \wtilde{M} \to M$
the universal covers of $L$ and $M$.
By a \emph{lift} of $\iota$ to $\wtilde{M}$, we mean a lift of $\iota \circ \pi$ to $\wtilde{M}$.

\begin{lemma}
\label{lemma:criterion_top_unob}
Let $L$ be a compact manifold.
An immersion $\iota: L \to M$ is topologically unobstructed if and only if its lifts to $\wtilde{M}$
are proper embeddings.
\end{lemma}
\begin{proof}
The case of immersed curves on surfaces is Lemma 2.2 of \cite{Abouzaid08}. The general case follows from the same argument.
Since the proof is short, we include it here.

Observe that $\iota$ is topologically obstructed if and only if there is a path $\gamma$ in $L$ such that
\begin{enumerate}[label = (\roman*)]
	\item either $\gamma$ has distinct endpoints (in the case of a teardrop) or is a non-contractible loop (in the case of a disk), and
	\item $\iota \circ \gamma$ is a contractible loop in $M$.
\end{enumerate} 
By lifting to the universal covers, paths in $L$ satisfying conditions (i) and (ii) correspond to paths $\wtilde{\gamma}$ in $\wtilde{L}$ with
distinct endpoints 
such that $\wtilde{\iota} \circ \wtilde{\gamma}$ is a loop.
Such paths exist if and only if $\wtilde{\iota}$ has self-intersections.

Hence, we have proved that $\iota$ is topologically unobstructed if and only if its lifts are injective.
To finish the proof, we observe that an injective lift of a proper immersion is also a proper immersion, hence a proper embedding.
\end{proof}

\subsection{Topological obstruction of suspensions}
\label{subsection:obstruction_suspensions}

We now prove a criterion for the topological unobstructedness of Lagrangian suspensions.

\begin{lemma}
\label{lemma:obstruction_suspension}
Let $\iota: L \times [0,1] \to \C \times M$ be the suspension of an exact Lagrangian homotopy $(\phi_t)_{t \in [0,1]}$.
If $\phi_t$ is topologically unobstructed for every $t \in [0,1]$, then $\iota$ is topologically unobstructed.
\end{lemma}
\begin{proof}
This follows easily from the fact the each slice $L \times \{ t \}$ is a strong deformation retract of $L \times [0,1]$.
Indeed, this readily implies that $\iota$ is incompressible if and only if $\phi_t$ is incompressible for all $t$.
Moreover, if $\iota$ bounds a teardrop $u$, then by definition of the suspension the endpoints of the boundary lift of $u$
must be in the same slice $L \times \{ t \}$. Hence, applying the deformation retraction to the 
boundary lift yields a teardrop with boundary on $\phi_t$.
\end{proof}

\subsection{Stability under perturbations}
\label{subsection:stability}

The incompressibility of an immersion $\iota$ is obviously invariant under homotopies of $\iota$.
On the other hand, the non-existence of teardrops with boundary on an immersion is generally not preserved by homotopies or even regular homotopies.
Nevertheless, this property is preserved by sufficiently small deformations, as we now show.

\begin{lemma}
\label{lemma:invariance_teardrops}
Let $L$ be a compact manifold, possibly with boundary.
Suppose that the immersion $\iota: L \to M$ does not bound teardrops. 
Then any immersion sufficiently close to $\iota$
in the $C^1$ topology does not bound teardrops.
\end{lemma}
\begin{proof}
Assume, in view of a contradiction, that there is a sequence of immersions $\theta_n: L \to M$
converging to $\iota$ in the $C^1$ topology, such that each $\theta_n$ bounds a teardrop.
This means that there exist sequences $(p_n)$ and $(q_n)$ in $L$ and a sequence of paths $\gamma_n$ from $p_n$ to $q_n$
with the following properties:
\begin{enumerate}[label = (\roman*)]
	\item $\theta_n(p_n) = \theta_n(q_n)$ for all $n$.
	\item $q_n \neq p_n$ for all $n$.
	\item The loop $\theta_n \circ \gamma_n$ is contractible in $M$.
\end{enumerate}
By compactness, we can assume that $p_n \to p$ and $q_n \to q$ for some points $p,q \in L$.
Then $\iota(p) = \iota (q)$. Moreover, since $\theta_n \to \iota$ in the $C^1$ topology, we must have $p \neq q$ (otherwise the
existence of the sequences $p_n$ and $q_n$
would contradict that $\iota$ is an immersion).

For every $n$, choose a path $\wtilde{\gamma}_n$ from $p$ to $q$ so that $d_{C^0}(\gamma_n, \wtilde{\gamma}_n) \to 0$
as $n \to \infty$ (where the $C^0$ distance is computed with respect to some fixed Riemannian metric).
Then we also have $d_{C^0}(\theta_n \circ \gamma_n, \iota \circ \wtilde{\gamma}_n) \to 0$ as $n \to \infty$
since $\theta_n \to \iota$ in the $C^0$ topology.
In particular, for $n$ large enough the loops $\theta_n \circ \gamma_n$ and $\iota \circ \wtilde{\gamma}_n$
are freely homotopic.
Therefore the loop $\iota \circ \wtilde{\gamma}_n$ is contractible in $M$ and we conclude that $\iota$ bounds a teardrop.
\end{proof}

\begin{corollary}
\label{cor:top_unob_open_condition}
Let $\iota: L \to M$ be a topologically unobstructed immersion, where $L$ is compact.
Then any immersion sufficiently close to $\iota$
in the $C^1$ topology is topologically unobstructed.
\end{corollary}

Recall that a Lagrangian immersion
has \emph{generic self-intersections} if 
it has no triple points and each double point is transverse.
It is well known that immersions $L \to M$ with generic self-intersections
are dense in the space of all Lagrangian immersions.
As a consequence, we obtain the following approximation result.

\begin{corollary}
\label{cor:generic_perturbation}
Let $\iota: L \to M$ be a topologically unobstructed immersion, where $L$ is compact.
Then there is a $C^1$-close Lagrangian immersion which is topologically unobstructed
and has generic self-intersections.
\end{corollary}

In the case of a cobordism $V$ with immersed ends, this corollary allows us to perturb $V$
so that its ends are generic, and so that the only non-generic self-intersections of $V$ are the intervals of double points corresponding
to the double points of its ends.
The double points over the ends of $V$ will be handled by a more careful choice of perturbations,
which are described in the next section.

\subsection{Concatenations of cobordisms along immersed ends}
\label{subsection:concatenations}

Let $V: L \cob (L_1, \ldots, L_r)$ be a cobordism and let $V'$ be a cobordism that has a negative end modelled over $L$.
Then $V$ and $V'$ can be concatenated by gluing them along the ends corresponding to $L$, producing
a cobordism $V \# V'$.
Assume that $V$ and $V'$ are topologically unobstructed.
Then $V \# V'$ needs not be topologically unobstructed,
since a disk or teardrop on $V \# V'$ needs not correspond to a disk or teardrop on $V$ or $V'$.
Despite this, we will prove
that, after a suitable perturbation, the concatenation $V \# V'$ 
does not bound $J$-holomorphic disks or teardrops for appropriate choices of almost complex structures $J$.
The precise statement is as follows.

\begin{proposition}
\label{prop:concatenation_unobstructed}
Let $V$ be a Lagrangian cobordism with embedded ends which is the concatenation
of topologically unobstructed cobordisms. 
Then $V$ is exact homotopic relative to its boundary to a quasi-exact cobordism.
\end{proposition}

Before proving the proposition, we describe the class of perturbations that will be used in the proof.
We will use a construction from \cite[Section 3.2.1]{Pictionary}
which replaces the self-intersections of an immersed end of a cobordism by transverse double points.
The outcome of this perturbation is a \emph{cobordism with bottlenecks} in the sense of \cite{MakWu}.
The presence of these bottlenecks is the key feature that allows to control the behaviour of $J$-holomorphic maps.

Let $\iota_V: V \to \C \times M$ be an immersed cobordism.
For notational convenience, we replace $V$ by its extension as in Section \ref{subsection:unob_cobordisms}.
Consider a positive end of $\iota_V$, which for definiteness we assume is lying
over the interval $I = [1 - \eps, \infty) \times \{ 0 \} \subset \C$ for some small $\eps > 0$.
This means that there is a Lagrangian immersion $\iota_L: L \to M$ and a proper embedding 
$j: [0, \infty) \times L \to V$ such that
$V|_{I} = j( [1 - \eps, \infty) \times L)$ and
\[
\iota_V \circ j (x, p) = (x, 0 , \iota_L(p))
\] 
for $x \geq 1 - \eps$. From now on, we identify $[0, \infty) \times L$ with its image in $V$.
Moreover, we assume that $\iota_L$ has generic self-intersections
and that $\iota_V$ does not have double points in $[0, 1-\eps) \times L$.

The perturbations we consider are of the following form. 
Extend $\iota_V$ to a symplectic immersion $\Psi: U \to \C \times M$, where $U$ is a neighborhood of the zero-section in $T^* V$.
The perturbed immersion will be obtained by composing $\Psi$ with a Hamiltonian isotopy of the zero-section inside $U$.

To describe the relevant Hamiltonian isotopy, consider a smooth function $h: [0,\infty) \to \R$
that has the profile shown in Figure \ref{fig:bottleneck_perturbation}.
More precisely, $h(x)$ vanishes on $[0, 1 - 2 \eps]$, has a unique non-degenerate critical point in $(1 - 2 \eps, \infty)$
located at $x = 1$, and is affine with positive slope for $x > 2$.

\begin{figure}
	\centering
	\includegraphics[width = 0.8\textwidth]{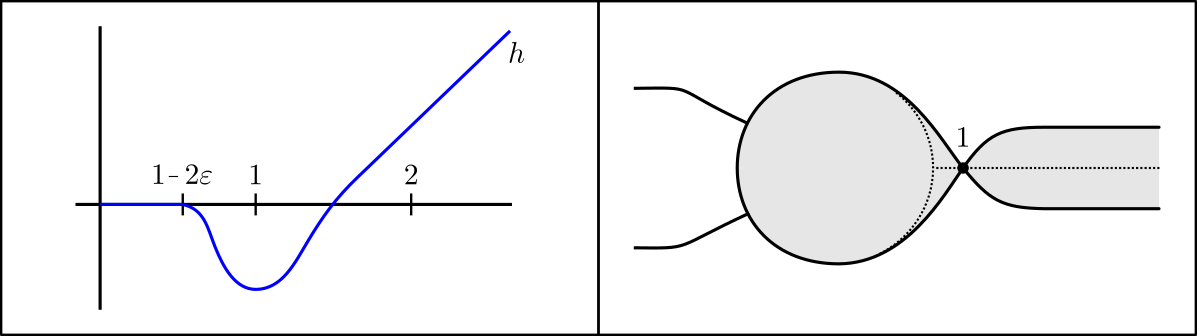}
	\caption{Left: the behaviour of the profile function $h$. Right: the projection to $\C$ of a cobordism after
	perturbation of a positive end.}
	\label{fig:bottleneck_perturbation}
\end{figure}

Fix a double point of $\iota_L$ and pick an ordering $(P_{-}, P_{+})$ of the two preimages. 
Let $D_{\pm} \subset L$ be disjoint open disks centered at $P_{\pm}$.
Fix bump functions $\chi_{\pm}$ on $D_{\pm}$ such that $\chi_{\pm} \geq 0$, $\chi \equiv 1$ near $P_{\pm}$ and $\chi \equiv 0$ near $\bdry D_{\pm}$.
We now define functions $H_{\pm} : [0, \infty) \times D_{\pm} \to \R$
by setting $H_{\pm}(x, p) = \pm  h(x)\, \chi_{\pm}(p)$.
Doing this for each pair of double points and extending by zero, we obtain a function $H$ on $V$.
Finally, extend $H$ to a neighborhood of $V$ in $T^*V$ by pulling back by the projection $T^*V \to V$.

Consider the family of Lagrangian immersions $\iota_t = \Psi \circ \phi^{t}_H |_V$ for small $t$.
By our choice of $h$, for $t > 0$ the projection of $\iota_t$ to $\C$ has the shape shown in Figure \ref{fig:bottleneck_perturbation}.
The key feature is that for $t$ small enough each pair of intervals of double points of $\iota_V$ inside $[0, \infty) \times L$ 
has been replaced by two pairs of transverse double points.
One of these pairs projects to the "bottleneck" at $x = 1$, while 
the other pair is a small perturbation of the double points $(1 - \eps, P_{\pm})$ of $\iota_V$.

Assuming that $\iota_V$ does not bound teardrops, then by Lemma \ref{lemma:invariance_teardrops}, for $t$ small enough $\iota_t$ also
does not bound teardrops. We will always assume that $t$ is chosen small enough so that this is the case.

In the case of a negative end of $V$, the perturbation is defined in a similar way, with the profile function $h$
replaced by $x \mapsto -h(-x)$.

We now proceed to the proof of Proposition \ref{prop:concatenation_unobstructed}.

\begin{proof}[Proof of Proposition \ref{prop:concatenation_unobstructed}.]
We only provide details for the case where $V$ is the concatenation of two cobordisms; the general case is similar.

Let $V$ be the concatenation of topologically unobstructed cobordisms $V_0$ and $V_1$ along a matching end modelled over the immersed Lagrangian $L$.
After generic perturbations, we may assume that $L$ has generic self-intersections and 
that the only 
non-generic self-intersections of $V_0$ and $V_1$ are along the end corresponding to $L$.
By Lemma \ref{lemma:invariance_teardrops}, for a small enough perturbation the resulting cobordisms are still topologically unobstructed.

Next, perturb $V$ by splicing together compatible perturbations of the
ends of $V_0$ and $V_1$ which were concatenated, as described in the beginning of this section.
The outcome is a cobordism (which we still call $V$) with generic self-intersections
and two bottlenecks $\zeta_{\pm} \in \C$ as in Figure \ref{fig:perturbed_concatenation}. 
The bottlenecks separate the cobordism in three parts which we label $V_{\ell}$, $V_{c}$ and $V_{r}$.
Note that $V_{\ell}$, $V_r$ and $V_c$ are topologically unobstructed
since they are small perturbations of 
$V_0$, $V_1$ and a product cobordism $I \times L$, respectively.

\begin{figure}
	\centering
	\includegraphics[width = 0.6\textwidth]{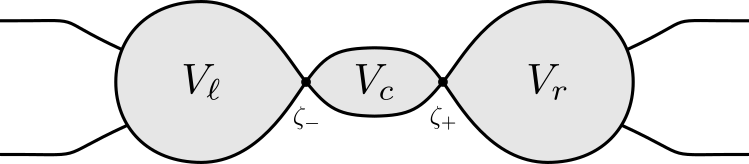}
	\caption{The projection to $\C$ of the cobordism obtained by perturbing the concatenation of $V_0$ and $V_1$.}
	\label{fig:perturbed_concatenation}
\end{figure}

Let $J$ be a compatible almost complex structure on $\C \times M$ with the following properties:
\begin{enumerate}[label = (\roman*)]
	\item The projection $\pi_{\C}$ is $(J,i)$-holomorphic outside $K \times M$ for some compact set $K \subset \C$.
	\item The projection $\pi_{\C}$ is $(J,i)$-holomorphic on $U_{\pm} \times M$, where $U_{\pm}$ is a neighborhood of the bottleneck $\zeta_{\pm}$.
\end{enumerate}
We claim that $(V,J)$ is quasi-exact.
Indeed, if there is a $J$-holomorphic teardrop with boundary on $V$, then by the open mapping theorem this teardrop cannot cross the bottlenecks, 
i.e. it must have boundary on one of the three parts $V_{\ell}$, $V_c$ or $V_r$. This contradicts
that these immersions are topologically unobstructed.
By the same argument, $V$ does not bound non-constant $J$-holomorphic disks.
\end{proof}

%% file: sections/obstruction_dim2.tex
We now turn to the computation of $\Gunob(\Sigma)$ for a closed surface $\Sigma$ of genus $g \geq 2$,
which will take up the rest of this paper.
As a first step, in this section we prove topological unobstructedness results
for the cobordisms associated to the surgery of immersed curves in $\Sigma$.
These results will form the basis of the proofs
that the cobordisms appearing in the computation of $\Gunob(\Sigma)$ are unobstructed.
We note that the proofs in this section make heavy use of 
special features of the topology of surfaces.

Throughout this section, we fix immersed curves $\alpha$ and $\beta$ in $\Sigma$ 
that are in general position and intersect at a point $s \in \Sigma$.
In Section \ref{subsection:def_surgery}, we give a precise description of the Lagrangian surgery of $\alpha$ and $\beta$ at $s$ and of the associated
surgery cobordism, which we denote $S(\alpha, \beta ; s)$.
In Section \ref{subsection:teardrops} we consider teardrops with boundary $S(\alpha, \beta; s)$,
and in Section \ref{subsection:disks} we characterize the existence of non-trivial disks with boundary on $S(\alpha, \beta; s)$.

\subsection{Construction of the surgery cobordism}
\label{subsection:def_surgery}

We recall the construction of the surgery cobordism $S(\alpha, \beta; s)$ 
in order to fix notations for the proofs of the obstruction results of 
the next sections.
We closely follow the construction given by Biran and Cornea \cite[Section 6.1]{BC13},
although our presentation differs slightly since we will need some 
control over the double points of the surgery cobordism
in order to investigate teardrops.

We start with a local model, which is the surgery of the Lagrangian subspaces $\R$ and $i \R$ in $\C$.
Fix an embedding $c: \R \to \C$ with the following properties.
Writing $c = c_1 + ic_2$, we have for some $\eps >0$
\begin{itemize}
	\item $c(t) = t$ for $t \leq -\frac{\eps}{2}$ and $c(t) = i t$ for $t \geq \frac{\eps}{2}$,
	\item $c_1'(t) >0$ and $c_2'(t) >0$ for $t \in (- \frac{\eps}{2}, \frac{\eps}{2})$,
	\item $c_1''(t) < 0$ and $c_2''(t) > 0$ for $t \in (- \frac{\eps}{2}, \frac{\eps}{2})$.
\end{itemize}
The local model for the surgery is then the Lagrangian $\R \# i\R = c(\R) \cup -c(\R)$.

\begin{figure}[h]
	\centering
	\includegraphics[width=0.3\textwidth]{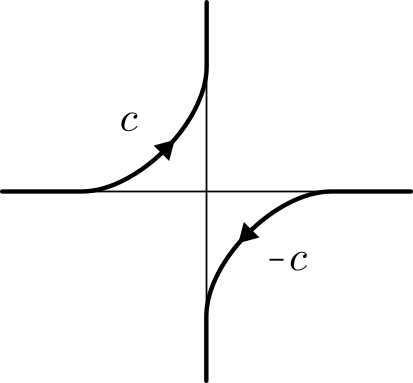}
	\caption{The shape of the curve $c$ used to define the local surgery model.}
	\label{fig:surgery_curve}
\end{figure}

To define the local model for the surgery cobordism, consider the Lagrangian embedding
\begin{align*}
\psi: \R \times S^1 &\longrightarrow \C^2 \\
(t, (x, y)) &\longmapsto (c(t) x, c(t) y).
\end{align*}
The model for the surgery cobordism is the Lagrangian $1$-handle $\psi(W)$, where
\begin{equation}
\label{eq:local_model_surgery}
W = \{ (t,(x,y)) : x \geq 0 , |tx| \leq \eps \}.
\end{equation}
Note that $\psi(W)$ is a submanifold of $\C^2$ with boundary 
$(\{ -\eps \} \times \R )\cup (\{i \eps \} \times i \R ) \cup (\{0 \} \times \R \# i\R)$.

Suppose now that $\alpha$ and $\beta$ are immersed curves in $\Sigma$ in general position.
Let $s$ be a positive intersection\footnote{According to the grading convention of Section \ref{section:grading}, this means that $s$ has degree $1$ in $CF(\alpha, \beta)$.} 
for the pair $(\alpha, \beta)$.
Let $B(0, \eps)$ be the open disk of radius $\eps$ in $\C$. 
For a small enough $\eps >0$, 
fix a Darboux chart $\Phi: B(0, \eps) \to U \subset \Sigma$ centered at $s$ with the property that 
$\Phi^{-1} \circ \alpha$ parametrizes $\R$ in the positive direction and 
$\Phi^{-1} \circ \beta$ parametrizes $i\R$.
The surgery $\alpha \#_s \beta$ is defined by taking the union $\alpha \cup \beta$ and replacing its intersection 
with $U$ by the local model $\R \# i \R$. 
Note that the resulting immersion is independent of the above choices up to Hamiltonian isotopy.
Moreover, since $s$ is positive the surgery has a canonical orientation
which is compatible with those of $\alpha$ and $\beta$.

To construct the surgery cobordism, start with the product cobordisms 
$\what{\alpha} = [-1,0] \times \alpha$
and
$\what{\beta} = i[0,1] \times \beta$.
These cobordisms are oriented so that $[-1,0]$ is oriented from $-1$ to $0$, and $i[0,1]$ is oriented from $i$ to $0$.
Let $\what{\iota}: S \to \C \times \Sigma$ be the Lagrangian immersion
obtained from $\alpha \cup \beta$ 
by removing its intersection with $B(0, \eps) \times U$ and gluing the handle $\psi(W)$
using the Darboux chart $\id \times \Phi: B(0, \eps) \times B(0, \eps) \to B(0, \eps) \times U$. 
Then $\what{\iota}$ is a Lagrangian immersion of a pair of pants whose restriction to the boundary coincides with the immersions
$\{ -1 \} \times \alpha$, $\{ i \} \times \beta$ and $\{ 0 \} \times (\alpha \#_s \beta)$.
Note that the handle $\psi(W)$ is compatible with the orientations of $\what{\alpha}$ and $\what{\beta}$, 
so that $S$ is canonically oriented.

The immersion $\what{\iota}$ is almost the required cobordism $\alpha \#_s \beta \cob (\alpha, \beta)$, except
that it is not cylindrical near the boundary component that projects to $0 \in \C$. 
We perturb $\what{\iota}$ to make it cylindrical using the argument of \cite[Section 6.1]{BC13}.

To describe the perturbation, we first introduce some notations. 
Write 
\[
P = S - \what{\iota}^{-1}(B(0, \eps) \times U),
\]
so that $\what{\iota}$ coincides over $P$ with restrictions of the products $\what{\alpha}$ and $\what{\beta}$.
Let $C$ be the boundary component of $S$ that projects to $0 \in \C$.
Let $\mathcal{U}$ be a collar neighborhood of $C$.
Taking $\mathcal{U}$ smaller if necessary, assume
that there is an isomorphism $\mathcal{U} \iso(-\delta, 0] \times S^1$ that extends
the obvious identification over $P$.

Extend the immersion $\what{\iota}: S \to \C \times \Sigma$
to a symplectic immersion $\vphi: \mathcal{N} \to \C \times \Sigma$, 
where $\mathcal{N}$ is a convex neighborhood of the zero-section in $T^*S$.
We make the following assumptions on $\vphi$.
First, taking $\mathcal{N}$ smaller if necessary, we may assume that $\vphi$ is an embedding on each fiber
and that $\vphi$ does not have self-intersections that project to the subset $\what{\iota}^{-1}( \C \times \overline{U})$.
Secondly, 
we choose $\vphi$ so that over $P$ it is
compatible with the splitting of $\what{\iota}$ into a product.

Write $\gamma = \alpha \#_{s} \beta$ and consider the immersed Lagrangian submanifold
$I \times \gamma$, where $I = \{ x + i y  \in \C : y = -x \}$.
Taking $\mathcal{U}$ smaller if necessary, we may assume that $\vphi^{-1}(I \times \gamma)$
is the graph of a closed $1$-form $\eta$ on $\mathcal{U}$ that vanishes on $C$.
Since $\mathcal{U}$ deformation retracts onto $C$, $\eta$ is exact. 
Write $\eta = d G$ for a function $G$ on $\mathcal{U}$.
Let $r$ be the collar coordinate on $\mathcal{U}$ and let $\chi(r)$ be an increasing smooth function such that $\chi(r) = 0$
near $-\delta$
and $\chi(r) = 1$ on a neighborhood of $0$.
Extend the function $\chi(r) G$ to $S$ by $0$.
Finally, the desired immersion $\iota: S \to \C \times \Sigma$ 
is the composition of $\vphi$ with the graph of $d(\chi G)$, which is then isotoped 
to make its ends horizontal.

\begin{notation*}
In the remainder of this section, we will use the following notations to describe $S(\alpha, \beta; s)$.
We denote by $S$ the domain of the cobordism and by $\iota: S \to \C \times \Sigma$
the immersion of $S$.
We denote by $A$, $B$ and $C$ the boundary components of $S$ which correspond, respectively, to the curves $\alpha$, $\beta$ and 
$\gamma := \alpha \#_{s} \beta$.
We call $H = \iota^{-1}(\C \times U)$ the handle region, where $U$ is the Darboux chart around $s$ used to define the surgery.
We call $K = \iota^{-1}( \C \times \{ s \})$
the core of the handle.
See Figure \ref{fig:surgery_domain} for a schematic representation of $S$.

\end{notation*}
\begin{figure}
	\centering
	\includegraphics[width=0.5\textwidth]{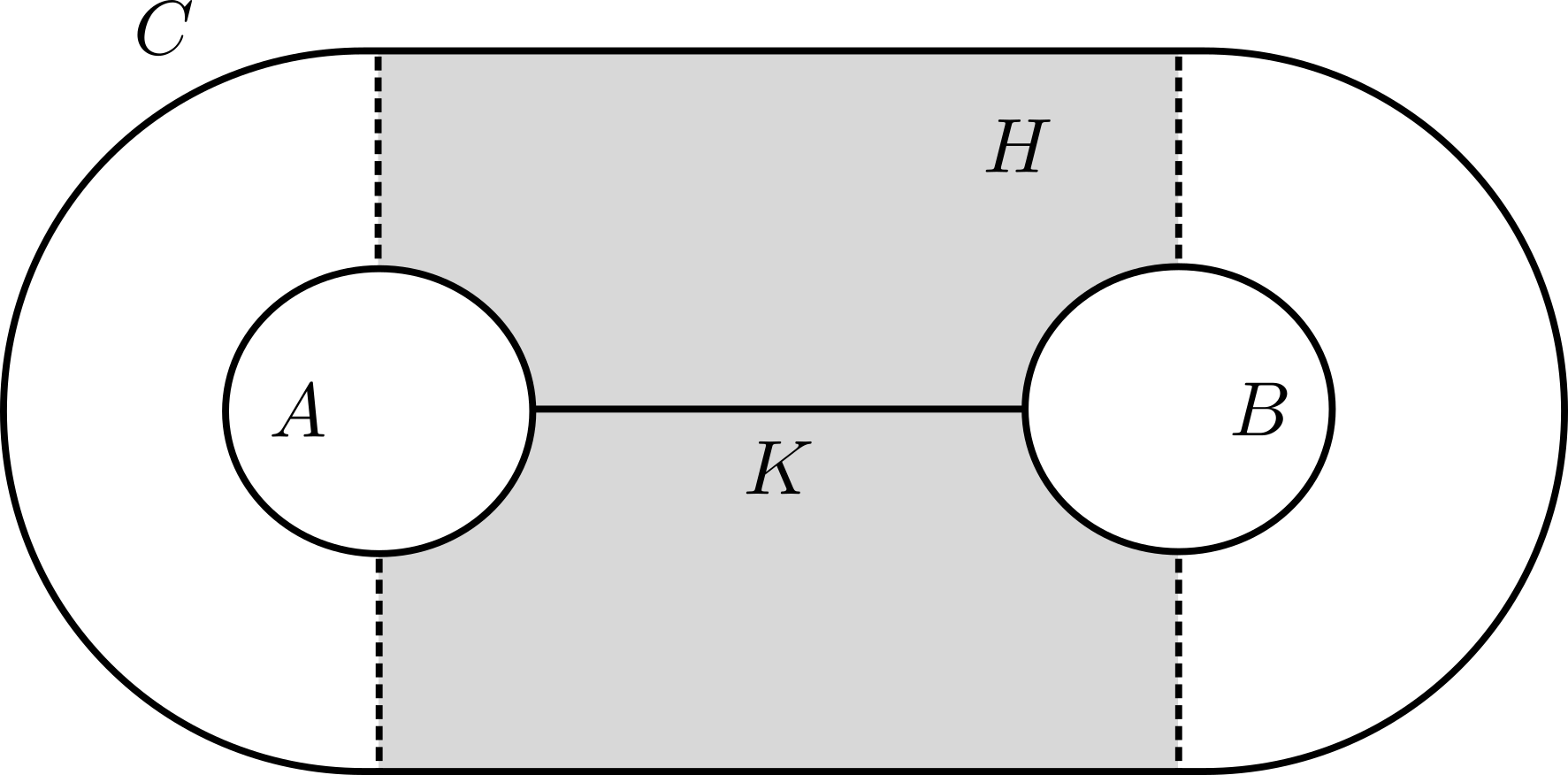}
	\caption{Schematic representation of the domain of a surgery cobordism. The handle $H$ is shaded. On the complement of $H$,
	the cobordism is equivalent to a product.}
	\label{fig:surgery_domain}
\end{figure}

The following properties of $S(\alpha, \beta; s)$ are straightforward consequences of the construction.

\begin{lemma}
The surgery cobordism $S(\alpha, \beta; s)$ satisfies the following properties:
\begin{enumerate}[label = (\roman*), font=\normalfont]
	\item All the double points of $\iota$ belong to $S \setminus H$.
	\item Over $S \setminus H$, $\iota$ is equivalent to restrictions of the
	product immersions $\lambda \times \alpha$ and $\lambda' \times \beta$,
	where $\lambda$ and $\lambda'$ are embedded paths in $\C$.
\end{enumerate}
\end{lemma}
\begin{proof}
Property (i) follows from the fact that $H$ does not contain double points of the original immersion $\what{\iota}$.
Our assumptions about the Weinstein neighborhood $\vphi$ imply that this is still true after perturbation.

Property (ii) follows from the fact that the $\Sigma$ component of $\what{\iota}$
agrees with the $\Sigma$ component of $I \times \gamma$ over $P$, and
that the Weinstein immersion $\vphi$ is chosen to respect
the splitting of $\iota$ into a product over $P$.
It follows that over $P$ the $1$-form $\eta$ used to define the perturbation is given by $\eta = f(r)dr$, 
where $r$ is the collar coordinate on $\mathcal{U}$.
This implies that the perturbed immersion has the claimed form over $S \setminus H$.

\end{proof}

\subsection{Teardrops on surgery cobordisms}
\label{subsection:teardrops}

By construction, the projection of the surgery cobordism $S(\alpha, \beta; s)$ to $\Sigma$
consists of $\alpha \cup \beta$ and a small neighborhood of $s$.
In particular, the projection deformation retracts onto $\alpha \cup \beta$. 
Therefore, the projection of a polygon with boundary on $S(\alpha, \beta; s)$
should correspond, via this deformation, to a polygon with boundary on $\alpha \cup \beta$
that may have additional corners at $s$.
This leads to the following definition.

\begin{definition}
\label{def:marked_teardrop}
Let $\alpha$ and $\beta$ be immersed curves intersecting at $s$. 
A polygon with boundary on $\alpha$ and $\beta$ is called an 
\Def{$s$-marked teardrop} 
if it has one corner at a point $c$ distinct from $s$, and all remaining corners at $s$. 
\end{definition}

See Figure \ref{fig:marked_teardrops} for some examples of marked teardrops.
Note that in the preceding definition we allow the case where there is no corner at $s$;
the polygon is then an ordinary teardrop on one of the curves.

\begin{figure}
	\centering
	\includegraphics[width= \textwidth]{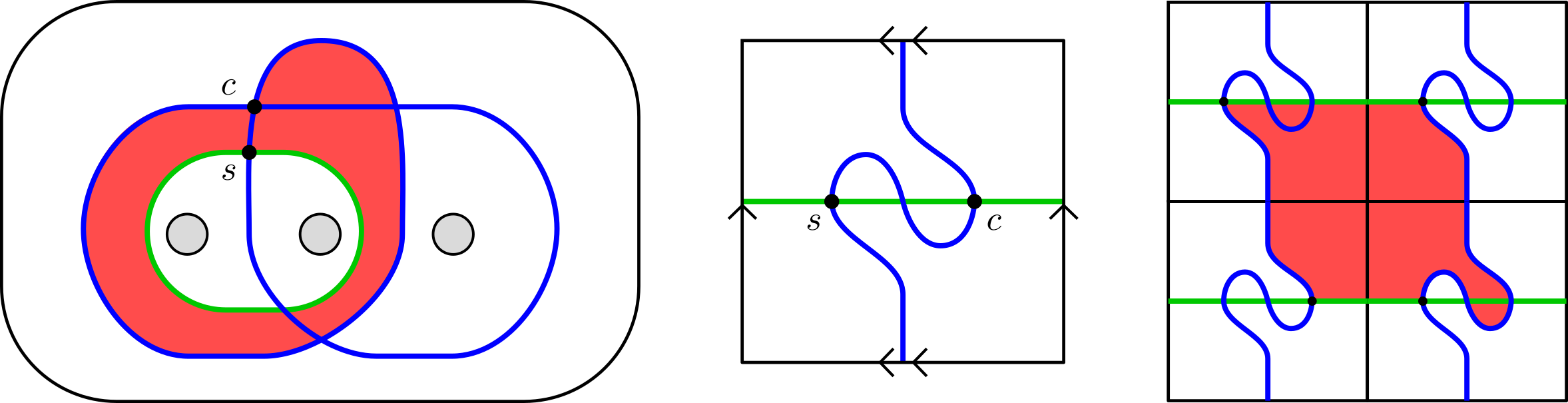}
	\caption{Left: curves on the four-holed sphere that bound an $s$-marked triangle. Center and right: curves on the torus that bound an $s$-marked square, with the square drawn in the universal cover of the torus.}
	\label{fig:marked_teardrops}
\end{figure}

\begin{remark}
In Appendix \ref{appendix:teardrops},
we show that curves that bound marked teardrops necessarily bound
teardrops or bigons.
We shall not make use of this result in the present paper, though it may be of independent interest.
\end{remark}

To relate teardrops on $S(\alpha, \beta; s)$ with $s$-marked teardrops on $\alpha \cup \beta$,
we will use the following lemma, which makes the above deformation argument precise.

\begin{lemma}
\label{lemma:deformation_surgery}
There is a strong deformation retraction $r: S \times [0,1] \to S$
of $S$ onto $A \cup B \cup K$ that satisfies
\begin{equation}
\label{eq:condition_double_points}
\pi_{\Sigma} \circ \iota \circ r_s (p) = \pi_{\Sigma} \circ \iota(p)	
\end{equation}
for all $s \in [0,1]$ and for all $p \in S \setminus H$.
\end{lemma}
\begin{proof}
Recall that the cobordism $\iota$
is obtained by perturbing an immersion $\what{\iota}$.
Moreover, over $S \setminus H$ 
we have $\pi_{\Sigma} \circ \iota = \pi_{\Sigma} \circ \what{\iota}$.
Hence, it suffices to construct the required deformation retraction for $\what{\iota}$.

For the immersion $\what{\iota}$,
the claim follows from a standard argument,
using the fact that $\what{\iota}$ is obtained from the product cobordisms $\what{\alpha}$
and $\what{\beta}$ by attaching the index $1$ handle $H$.
This can be made precise by the following general argument. 
Observe that the function $(x+iy, p) \mapsto x-y$ on $\C \times \Sigma$
restricts to a Morse function $f$ on $S$ with a single critical point of index $1$, which lies on $K$.
The required deformation retraction can then be obtained by using the negative gradient flow of $f$ away from
$K$ and a straight-line homotopy near $K$, as explained in \cite[Theorem 3.14]{Milnor-hcobordism}.

\end{proof}

The purpose of condition \eqref{eq:condition_double_points} in the previous lemma is to control
the behaviour of the double points of $\iota$ throughout the deformation.

\begin{proposition}
\label{prop:teardropcriterion}
Let $\alpha$ and $\beta$ be immersed curves intersecting at $s$. 
Then the surgery cobordism $S(\alpha, \beta; s)$ bounds a teardrop 
if and only if
$\alpha$ and $\beta$ bound an $s$-marked teardrop.
\end{proposition}
\begin{proof}
Suppose first that $S(\alpha, \beta; s)$ bounds a teardrop $u$.
We see the boundary lift of $u$ as a path $\delta:[0,1] \to S$ with distinct endpoints and with $\iota(\delta(0)) = \iota(\delta(1))$.
Since the handle $H$ does not contain double points of $\iota$, the endpoints of $\delta$ lie in $S \setminus H$.
By applying the retraction of Lemma \ref{lemma:deformation_surgery}, we obtain a new path 
$\delta' := r_1 \circ \delta$ 
on $A \cup B \cup K$
whose endpoints are distinct 
and belong to $A \cup B$.
By a further homotopy, we can assume that this path is locally embedded.
By condition \eqref{eq:condition_double_points} 
of Lemma \ref{lemma:deformation_surgery}, the path $\lambda:= \pi_{\Sigma} \circ \iota \circ \delta'$ is a loop
which is homotopic relative endpoints to the loop $\pi_{\Sigma} \circ \iota \circ \delta$.
It follows that $\lambda$ is null-homotopic, hence is the boundary of a disk $v$ with boundary on $\alpha \cup \beta$.
The disk $v$ is an $s$-marked teardrop; indeed, it has one corner corresponding to projection of the corner of $u$,
and the remaining corners at $s$ correspond to the (finitely many) times when $\lambda$
crosses $K$.

Conversely, suppose that $\alpha \cup \beta$ bounds an $s$-marked teardrop $v$
with corner at $c$.
Identify $A$ and $B$ with the domains of $\alpha$ and $\beta$, respectively. 
Then the boundary lifts of $v$ can be seen as paths in $A \cup B$ with endpoints at the preimages of $s$ and $c$.
Whenever $v$ has a corner that maps to $s$, the boundary lifts on each side of the corner can be glued together by concatenating them with a path going along the core $K$. 
Since all the corners of $v$ except one are mapped to $s$, by doing this we obtain a single path $\delta$ in $A \cup B \cup K$ with endpoints on the two preimages of the other corner $c$.

If the endpoints of $\delta$ lie on the same component of $A \cup B$, then $\iota \circ \delta$ is a loop. 
This loop is null-homotopic in $\C \times \Sigma$ since by construction $\pi_{\Sigma} \circ \iota \circ \delta$ 
is a reparametrization of $v|_{\bdry \D}$.
Hence $\iota \circ \delta$ extends to a disk $u: \D \to \C \times \Sigma$ with boundary on $\iota(S)$. 
This disk is a teardrop since the path $\delta$ is a boundary lift of $u$.

If the endpoints of $\delta$ de not lie on the same component of $A \cup B$, we close up $\iota \circ \delta$
into a loop in the following way. By Lemma \ref{lemma:deformation_surgery}, 
for each endpoint of $\delta$,
there is a path in $(\pi_{\Sigma}\circ \iota)^{-1}(c)$ that connects that endpoint to the boundary component $C$.
Concatenating $\delta$ with these paths yields a path $\delta'$ with endpoints on $C$ with the property that $\iota \circ \delta'$ is a 
loop.
By the same argument as in the previous case, $\delta'$ is the boundary lift of a teardrop on $\iota$.
\end{proof}

Next, we give a simple algebraic obstruction to representing a class in $\pi_2(\Sigma, \alpha \cup \beta, s)$ by an $s$-marked teardrop.
Since $\alpha$ and $\beta$ are in general position,
we can see the union $\alpha \cup \beta$ as a $4$-valent oriented graph (i.e. a $1$-dimensional CW complex) embedded in $\Sigma$, 
whose vertices are the intersections and self-intersections of the curves.
For each vertex $v$ of this graph, there is an integral cellular $1$-cocycle $\rho_v$ 
supported on the edges incident to $v$, whose values
are prescribed by Figure \ref{fig:corner_cocycle}. 
This cocycle can be thought of as giving an algebraic count of the number of corners that a cycle has at $v$.

\begin{figure}
	\centering
	\includegraphics[width=0.3\textwidth]{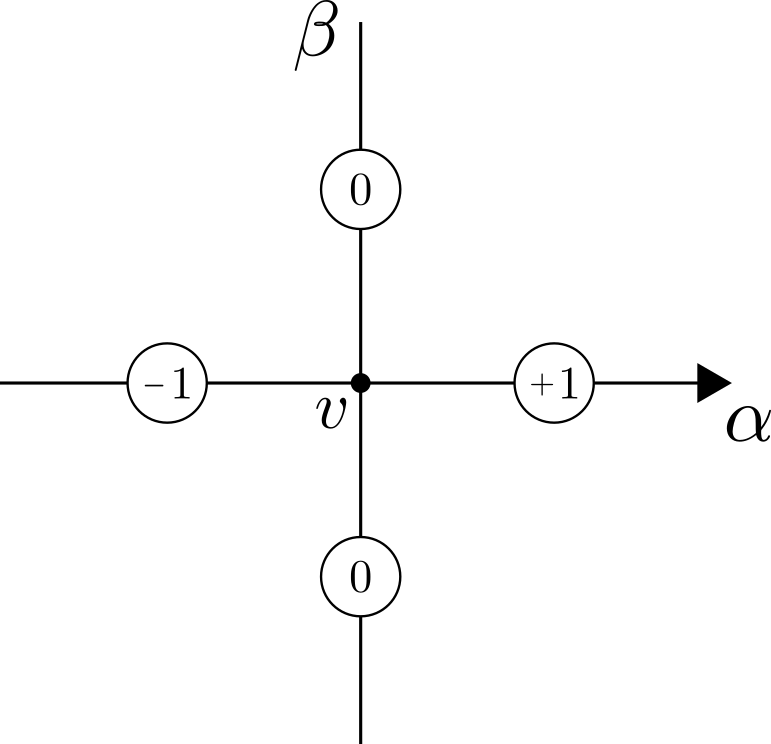}
	\caption{The definition of the cocycle $\rho_v$ that counts the algebraic number of corners of a cycle at $v$.}
	\label{fig:corner_cocycle}
\end{figure}

For a class $A \in \pi_2(\Sigma, \alpha \cup \beta, s)$, we denote by $\bdry A \in H_1(\alpha \cup \beta ;\Z)$ the image of its boundary by the Hurewicz morphism.

\begin{lemma}
\label{lemma:countingcorners}
Suppose that the class $A \in \pi_2(\Sigma, \alpha \cup \beta, s)$ can be represented 
by an $s$-marked teardrop with corner at $c$.
Then $\rho_c (\bdry A) = \pm 1$ and $\rho_v(\bdry A) =0$ for every 
vertex $v$ distinct from $c$ and $s$.
\end{lemma}
\begin{proof}
Suppose that $A$ can be represented by a polygon $u$ that has no corners at the vertex $v$.
Let $G$ be the graph obtained by pulling apart the two branches of $\alpha \cup \beta$ meeting at $v$. 
To be precise, $G$ is obtained from $\alpha \cup \beta$ by adding a vertex $v'$ and has the same set of edges, except that the two $\beta$ edges incident to $v$ in $\alpha \cup \beta$ are incident to $v'$ in $G$.
That $u$ has no corner at $v$ means that $\bdry u$ factors through the quotient map $G \to \alpha \cup \beta$ that identifies $v'$ and $v$. 
Since the pullback of $\rho_v$ to $G$ is exact, we conclude that $\rho_v(\bdry A) =0$.

Suppose now that $A$ can be represented by a polygon $u$ that has one corner at the vertex $c$. 
Let $E $ and $E'$ be the edges of $\alpha \cup \beta$ visited by $\bdry u$ before and after passing through the corner $c$, and suppose that their endpoints are $\{ c, p \}$ and $\{ c, q \}$, respectively. 
Since $\bdry u$ has a corner at $c$, one of these edges has $\rho_c = \pm 1$, and the other one has $\rho_c = 0$.
The loop $\bdry u$ is homotopic to the concatenation of a simple path going from $p$ to $c$ along $E$, a simple path going from $c$ to $q$ along $E'$, and a path going from $q$ back to $p$ that has no corners at $c$. Hence $\rho_c([\bdry u]) = \rho_c(E) \pm \rho_c(E') = \pm 1$.
\end{proof}

\subsection{Disks on surgery cobordisms}
\label{subsection:disks}

As a consequence of Lemma \ref{lemma:deformation_surgery}, we also obtain
the following criterion for the existence of non-trivial disks on surgery cobordisms.

\begin{proposition}
\label{prop:relativepi2}
Let $\alpha$ and $\beta$ be immersed curves intersecting at $s$. 
Then the surgery cobordism $S(\alpha, \beta; s)$ is incompressible
if and only if  
$\alpha$ and $\beta$ span a free group of rank $2$ in $\pi_1(\Sigma, s)$. 
\end{proposition}

\begin{proof}
We equip $S$ with the basepoint $x$ which is the intersection of $A$ and $K$.
By retracting $S$ onto $A \cup B \cup K$ as in Lemma \ref{lemma:deformation_surgery} and then collapsing $K$ to a point, 
we obtain a homotopy equivalence $S \to S^1 \vee S^1$ that makes the following diagram commute up to homotopy:
\begin{equation}
\label{cd:homotopy_eq}
\begin{tikzcd}
(S, x) \arrow{d}[swap]{\simeq} \arrow{r}{\iota} 	&(\C \times \Sigma \arrow{d}{\pi_{\Sigma}}, \iota(x)) \\
(S^1 \vee S^1, x_0) \arrow{r}{\alpha \vee \beta} 	&(\Sigma,s).
\end{tikzcd}
\end{equation}
Here, $\alpha$ and $\beta$ are seen as loops based at $s$, the wedge sum identifies the two preimages of $s$, and $x_0$ is the wedge point.

Since the vertical maps in Diagram \eqref{cd:homotopy_eq} are homotopy equivalences, we
deduce that $\iota$ is incompressible if and only if the map 
$\mu: \pi_1(S^1 \vee S^1, x_0) \to \pi_1(\Sigma, s)$ induced by $\alpha \vee \beta$
is injective.

The image of $\mu$ is the subgroup $\langle [\alpha], [\beta] \rangle$ of $\pi_1(\Sigma, s)$ generated by $\alpha$ and $\beta$.
Therefore, if $\mu$ is injective then $\langle [\alpha], [\beta] \rangle \iso \pi_1( S^1 \vee S^1 ,x_0)$, which is a free group of rank $2$.

Conversely, suppose that $\langle [\alpha], [\beta] \rangle$ is a free group of rank $2$. 
Then $\mu$ maps a free group of rank $2$ onto a free group of rank $2$.
By a theorem of Nielsen (see e.g. \cite[Proposition 3.5]{Lyndon-Schupp}), free groups of finite rank are Hopfian\footnote{Recall that a group $G$ is \emph{Hopfian} if any surjective morphism $G \to G$ is an isomorphism.}.
Hence $\mu$ is an isomorphism on its image,
so that $\mu$ is injective.
\end{proof}

For a surface of genus $g \geq 2$, the subgroup of $\pi_1$ generated by two elements is always free. 
This is a special case of the following result of Jaco.
\begin{theorem}[\protect{\cite[Corollary 2]{Jaco70}}]
\label{thm:Jaco}
Let $\Sigma$ be a surface with $\chi(\Sigma) \leq 0$. 
Then any subgroup of $\pi_1(\Sigma)$ generated by $k$ elements, where $k < 2 - \chi(\Sigma)$, is a free group.
\end{theorem}

In the setting of Proposition \ref{prop:relativepi2}, this implies that the subgroup of $\pi_1(\Sigma,s)$ generated by $\alpha$ and $\beta$ is a free group of rank $2$ or less.
The rank is $0$ or $1$ if and only if there is some non-trivial relation $[\alpha]^p = [\beta]^q$ in $\pi_1(\Sigma,s)$. 
Hence we can reformulate Proposition \ref{prop:relativepi2} as follows.

\begin{proposition}
\label{prop:relativepi2version2}
Let $\alpha$ and $\beta$ be immersed curves intersecting at $s$. 
Then the surgery cobordism $S(\alpha, \beta; s)$ is incompressible if and only if 
there are no integers $p$ and $q$ with $(p,q) \neq (0,0)$ such that $[\alpha]^p = [\beta]^q$ in $\pi_1(\Sigma,s)$. 
\end{proposition} 

We will often make use of the following special case of the preceding results.
\begin{corollary}
\label{cor:nonzerointersectionimpliesnodisks}
Let $\alpha$ and $\beta$ be immersed curves intersecting at $s$. 
If $\alpha \cdot \beta \neq 0$, then the surgery cobordism $S(\alpha, \beta; s)$ is incompressible.
\end{corollary}
\begin{proof}
The hypothesis implies that the intersection pairing restricted to $\langle [\alpha], [\beta] \rangle$ is non-trivial. 
Hence, $\langle [\alpha], [\beta] \rangle$ cannot be a free group of rank $0$ or $1$. 
By Theorem \ref{thm:Jaco}, it must be a free group of rank $2$, and the conclusion follows from Proposition \ref{prop:relativepi2}.
\end{proof}

%% file: sections/invariants.tex
In this section, we define the Lagrangian cobordism invariants of curves on symplectic surfaces that give rise to the morphism
\[
\Gunob(\Sigma) \to H_1( S\Sigma; \Z) \oplus \R
\]
of Theorem \ref{thm:computation_Gunob}. 
These invariants are the same as those considered by Abouzaid \cite{Abouzaid08},
who showed that they induce well-defined morphisms on $\Kgrp(\Sigma)$.
Our aim in this section is to show that these invariants also descend to morphisms on $\Gimm(\Sigma)$ (and hence on $\Gunob(\Sigma)$).
Moreover, in Section \ref{subsection:bounding_curves}, we compute these invariants in some cases which are relevant to the proof of Theorem \ref{thm:computation_Gunob}.

\subsection{Discrete invariants}
\label{subsection:discrete_invariants}

For a symplectic manifold $M^{2n}$, we write $\GrassOr(M)$ for the Grassmannian of oriented Lagrangian subspaces of $TM$.
This is a fiber bundle with fiber $\Lambda_n$, the oriented Lagrangian Grassmannian of $\C^n$. 
Recall that $\Lambda_n \iso U(n)/SO(n)$.
In the case of a surface $\Sigma$, we identify $\GrassOr(\Sigma)$ with the unit tangent bundle of $\Sigma$, which we denote $S\Sigma$.
We associate to an immersed curve $\gamma: S^1 \to \Sigma$ the class $[\wtilde{\gamma}] \in H_1(S \Sigma ; \Z)$,
where $\wtilde{\gamma}: S^1 \to S\Sigma$ is the Gauss map of $\gamma$.
We will show that this gives rise to a morphism $\Gimm(\Sigma) \to H_1(S \Sigma; \Z)$.

To see this, consider the stabilization map $i: \GrassOr(\Sigma) \to \GrassOr(\C \times \Sigma)$
that sends an oriented Lagrangian subspace $U \subset T_p \Sigma$ to the Lagrangian subspace 
$\R  \oplus U$ of $T_{(z,p)}(\C \times \Sigma)$, equipped with the product orientation. 
Here, $z \in \C$ is fixed, but the map does not depend on this choice up to homotopy.

\begin{lemma}
\label{lemma:stable_homology}
The stabilization map $i$ induces an isomorphism on $H_1(- ; \Z)$.
\end{lemma}
\begin{proof}
It suffices to prove that stabilization induces an isomorphism on $\pi_1$.
To see this, recall that the determinant map $\Lambda_n \to S^1$ coming from the identification $\Lambda_n \iso U(n)/SO(n)$ 
induces an isomorphism on $\pi_1$. 
The stabilization map $\Lambda_n \to \Lambda_{n+1}$ commutes with the determinant, hence it also induces an isomorphism on $\pi_1$.
The result now follows from comparison of the long exact sequences of homotopy groups associated to the fiber bundles $\GrassOr(\Sigma)$ and 
$\GrassOr(\C \times \Sigma)$.
\end{proof}

It follows from Lemma \ref{lemma:stable_homology} that the class $[\wtilde{\gamma}]$ is invariant under Lagrangian cobordisms in the following sense.

\begin{proposition}
\label{prop:invariance_homology_class}
Let $V: (\gamma_1, \ldots, \gamma_r) \cob \emptyset$ be an oriented Lagrangian cobordism between immersed curves in $\Sigma$. 
Let $\wtilde{\gamma}_1, \ldots, \wtilde{\gamma}_r$ be the Gauss maps of the ends of $V$.
Then
\begin{equation*}
\sum_{k=1}^{r} \; [\widetilde{\gamma}_k] =0 \qquad \text{in } H_1(S\Sigma; \Z).
\end{equation*}
\end{proposition}
\begin{proof}
The cobordism has a Gauss map $V \to \GrassOr(\C \times \Sigma)$
which extends the stabilizations of the Gauss maps of its ends.
This yields the relation
\[
\sum_{k=1}^{r} \; i_*[\wtilde{\gamma}_k] = 0 \qquad \text{in } H_1( \GrassOr(\C \times \Sigma); \Z).
\]
The conclusion now follows from Lemma \ref{lemma:stable_homology}.
\end{proof}

It follows from Proposition \ref{prop:invariance_homology_class} that the homology class of the Gauss map 
descends to a morphism $\Gimm(\Sigma) \to H_1(S\Sigma; \Z)$. 

Since $S\Sigma$ is an oriented circle bundle, its homology can be computed from the Gysin sequence
\begin{equation}
\label{cd:homology_unit_bundle}
\begin{tikzcd}
0 \arrow{r} &\Z_{\chi(\Sigma)} \arrow{r} &H_1(S\Sigma; \Z) \arrow{r}{p_*} &H_1(\Sigma; \Z) \arrow{r} &0,
\end{tikzcd}
\end{equation}
where the first map sends $1$ to the class of a fiber and $p: S \Sigma \to \Sigma$ is the projection. 
This sequence splits since $H_1(\Sigma; \Z)$ is a free abelian group.
It will be useful for us to fix a choice of splitting of \eqref{cd:homology_unit_bundle}.
A splitting map $\mu: H_1(S\Sigma; \Z) \to \Z_{\chi(\Sigma)}$ will be called
a \Def{winding number morphism}.

\begin{remark}
The set of splittings of \eqref{cd:homology_unit_bundle} is a torsor over 
\[
\Hom(H_1(\Sigma; \Z), \Z_{\chi(\Sigma)} ) \iso H^1(\Sigma; \Z_{\chi(\Sigma)} ).
\]
In the terminology of Seidel \cite{Seidel2000}, they correspond to the $\Z_{2\chi(\Sigma)}$-gradings of $\Sigma$ that factor through the 
oriented Lagrangian Grassmannian. 
Winding numbers of curves on closed surfaces first appeared in the works of Reinhart \cite{Reinhart} and Chillingworth \cite{Chillingworth}. 
\end{remark}

\subsection{Holonomy}
\label{subsection:holonomy}

We recall the definition of the holonomy of an immersed curve in a surface of genus $g \neq 0$, following \cite{Abouzaid08}.
First, fix a primitive $\lambda$ of $p^*\omega$, where as before $p: S \Sigma \to \Sigma$
denotes the projection of the unit tangent bundle. 
Such a primitive exists whenever
$\chi(\Sigma) \neq 0$ (see Corollary \ref{cor:pullback_omega_exact} for a proof of a more general claim).
The holonomy of an immersed curve $\gamma$ is then defined as
\[
\hol_{\lambda}(\gamma) := \int_{S^1} \wtilde{\gamma}^* \lambda,
\]
where $\wtilde{\gamma}$ is the lift of $\gamma$ to $S\Sigma$.

The holonomy of curves admits a straightforward generalization to a class of real-valued cobordism invariants of Lagrangian immersions
in \emph{monotone} symplectic manifolds. 
For the benefit of the reader, this theory is developed in Appendix \ref{appendix:action}.
As in the case of surfaces, the basic observation that leads to these invariants is that
for a monotone manifold $M$
the form $p^*\omega$ is exact, 
where $p: \Grass M \to M$ is the Lagrangian Grassmannian bundle (see Corollary \ref{cor:pullback_omega_exact}).
As a special case of the general theory developed in Appendix \ref{appendix:action}, 
we prove that holonomy is invariant under oriented 
Lagrangian cobordisms. The precise statement is as follows.

\begin{proposition}
\label{prop:invariance_holonomy}
Suppose that there is an oriented Lagrangian cobordism $(\gamma_1, \ldots, \gamma_r) \cob \emptyset$. 
Then
\begin{equation*}
\sum_{k=1}^{r} \; \hol_{\lambda}(\gamma_k) = 0.
\end{equation*}
\end{proposition}
\begin{proof}
This is a special case of Corollary \ref{cor:characteristic_numbers}.
\end{proof}

By Proposition \ref{prop:invariance_holonomy}, holonomy descends to a morphism $\hol_{\lambda}: \Gimm(\Sigma) \to \R$.
From now on, we fix a choice of primitive $\lambda$ and suppress it from the notation.

We will often make use of the following relationship between holonomy and regular homotopies.
Recall that a regular homotopy $\phi:S^1 \times I \to \Sigma$ is a Lagrangian homotopy and therefore has
a well-defined flux class $\flux(\phi) \in H^1(S^1; \R)$, which by definition is the class that satisfies
\[
\flux(\phi) \cdot [S^1] = \int_{S^1 \times I} \phi^*\omega.
\]

\begin{lemma}
\label{lemma:holonomy_primitive_of_flux}
Let $\phi:S^1 \times I \to \Sigma$ be a regular homotopy from $\alpha$ to $\beta$. Then
\[
\hol(\beta) - \hol(\alpha) = \int_{S^1 \times I} \phi^*\omega = \flux(\phi) \cdot [S^1]. 
\]
\end{lemma}
\begin{proof}
The regular homotopy $\phi$ lifts to a homotopy from $\wtilde{\alpha}$ to $\wtilde{\beta}$ in $S\Sigma$. 
The result then follows from the Stokes Theorem.
\end{proof}

\subsection{Winding and holonomy of bounding curves}
\label{subsection:bounding_curves}

In this section, we compute the winding number and holonomy of curves that bound surfaces in $\Sigma$.
We start with the following lemma.

\begin{lemma}
\label{lemma:holonomy_fiber}
Suppose that $\sigma$ is a positively oriented fiber of $S\Sigma$. 
Then
\[
\int_{\sigma} \lambda = -\frac{\area(\Sigma)}{\chi(\Sigma)}.
\]
\end{lemma}
\begin{proof}
Suppose that $\sigma$ is the fiber over $x$.
Let $v$ be a vector field on $\Sigma$ that has a unique zero at $x$.
By the Poincaré-Hopf Theorem, the index of this zero is $\chi(\Sigma)$.

Let $D$ be a smoothly embedded closed disk containing $x$.
Consider the section $\wtilde{v}: \Sigma \setminus \Int D \to S\Sigma$ given by $\wtilde{v} = v/|v|$.
By choosing a trivialization of $S\Sigma$ over $D$, we can find a homotopy in $p^{-1}(D)$ from the loop $\wtilde{v}|_{\bdry D}$ 
to the loop $\sigma$ iterated $-\chi(\Sigma)$ times.
Moreover, the area of this homotopy with respect to $p^*\omega$ is $\area(D)$.
Hence, by the Stokes Theorem, we have
\[
-\chi(\Sigma) \, \int_{\sigma} \lambda = \int_{\wtilde{v}|_{\bdry D}} \lambda + \area(D) = \area(\Sigma \setminus D) + \area(D) = \area(\Sigma).
\]
\end{proof}

\begin{lemma}
\label{lemma:holonomy+winding_bounding_curves}
Suppose that the curves $\gamma_k$, $k= 1, \ldots, r$, form the oriented boundary of a compact surface $P \subset \Sigma$. 
Then,
\begin{align*}
&\sum_{k=1}^{r} \hol(\gamma_k) = \area(P) - \frac{\chi(P)}{\chi(\Sigma)} \, \area(\Sigma), \\
&\sum_{k=1}^{r} \mu(\gamma_k) = \chi(P) \mod \chi(\Sigma).
\end{align*}
\end{lemma}
\begin{proof}
Let $v$ be a vector field on $P$ that agrees with the derivatives $\gamma_k'$ over $\bdry P$ and has a unique zero at a point $x \in \Int P$.
By the Poincaré-Hopf Theorem, the index of $v$ at $x$ is $\chi(P)$.

Let $\sigma$ be the fiber of $S\Sigma$ over $x$, with the positive orientation.
Arguing as in the proof of Lemma \ref{lemma:holonomy_fiber}, 
the chain $\sum_k \wtilde{\gamma}_k - \chi(P)\sigma$ 
is the boundary of a surface whose area with respect to $p^*\omega$ equals $\area(P)$.
As a consequence, using Lemma \ref{lemma:holonomy_fiber} we obtain
\[
\sum_{k=1}^r \hol(\gamma_k) = \area(P) + \chi(P) \, \int_{\sigma} \lambda = \area(P) - \frac{\chi(P)}{\chi(\Sigma)} \area(\Sigma).
\]
Likewise, since $\mu(\sigma) = 1$ by definition, we have
\[
\sum_{k=1}^r \mu(\gamma_k)= \chi(P) \, \mu(\sigma) = \chi(P).
\]
\end{proof}

As a consequence, we obtain the following

\begin{corollary}
\label{cor:homology_map_surjective}
The morphism $\Gunob(\Sigma) \to H_1(S\Sigma; \Z)$ is surjective.
\end{corollary}

%% file: sections/isotopies.tex
In this section, we describe how isotopic curves are related in $\Gunob(\Sigma)$.
We let $\mathcal{I}$ be the subgroup of $\Gunob(\Sigma)$
generated by all the elements of the form $[\wtilde{\alpha}] - [\alpha]$, where $\wtilde{\alpha}$ and $\alpha$ are isotopic curves.
The main result of this section is the following computation of $\mathcal{I}$.

\begin{proposition}
\label{prop:holonomy_iso}
The holonomy morphism restricts to an isomorphism $\mathcal{I} \iso \R$.
\end{proposition}

As an immediate consequence, we deduce that
\[
\Gunob(\Sigma) \iso \R \oplus \Gunob(\Sigma)/\mathcal{I}.
\]
Therefore, the computation of $\Gunob(\Sigma)$ reduces to the computation of the quotient
$\Gunob(\Sigma)/ \mathcal{I}$.

\begin{definition}
\label{def:reduced_group}
The quotient 
\[
\Gred(\Sigma) = \Gunob(\Sigma) / \mathcal{I}
\]
is called the \Def{reduced unobstructed cobordism group} of $\Sigma$.
\end{definition}

We emphasize that, by definition, isotopic curves become equal in $\Gred(\Sigma)$.
This will play a crucial role in the proofs of Section \ref{section:action_mcg} and Section \ref{section:computation} by allowing us
to put curves in minimal position by isotopies.

\subsection{Holonomy and Hamiltonian isotopies}
\label{subsection:ham_isotopies}

As a first step
towards the proof of Proposition \ref{prop:holonomy_iso},
we show that holonomy completely determines whether two isotopic curves are Hamiltonian isotopic.

\begin{lemma}
\label{lemma:ham_isotopic_embedded_case}
Let $\alpha$ and $\beta$ be isotopic simple curves with $\hol(\alpha) = \hol(\beta)$. 
Then $\alpha$ and $\beta$ are Hamiltonian isotopic.
\end{lemma}
\begin{proof}
Suppose first that the curves are non-separating. 
By Lemma \ref{lemma:holonomy_primitive_of_flux}, 
the holonomy condition implies that any
isotopy from $\alpha$ to $\beta$ has zero flux.
Moreover, since the curves are non-separating, they induce surjective maps $H^1(\Sigma; \R) \to H^1(S^1; \R)$.
By a standard argument, this condition implies that an isotopy from $\alpha$ to $\beta$
extends to an ambient symplectic isotopy with zero flux; see Lemma 6.6 of \cite{Solomon} for a detailed proof.
The result now follows from the well-known fact that a symplectic isotopy with zero flux
is homotopic relative to its endpoints to a Hamiltonian isotopy; see Theorem 10.2.5 of \cite{McDuffSalamon17}.

Suppose now that the curves are separating.
Then, by Lemma \ref{lemma:holonomy+winding_bounding_curves},
the surfaces bounded by the curves have the same area.
The conclusion then follows from a standard Moser isotopy argument.
In the case of contractible curves, a detailed proof is given in \cite[Appendix A]{AkveldSalamon}.
The same proof applies to the case of separating curves; for completeness, we sketch the argument from \cite{AkveldSalamon}.

Choose a diffeomorphism $\phi$ isotopic to the identity such that $\phi(\alpha) = \beta$.
Let $S$ and $S'$ be the surfaces bounded by $\alpha$ and $\beta$, respectively.
Then $\phi(S) = S'$, hence
\[
\int_S \phi^* \omega = \int_{S'} \omega = \int_{S} \omega
\]
since $S$ and $S'$ have the same area.
By a Moser isotopy argument (see \cite[Lemma A.3]{AkveldSalamon}), 
we can find an isotopy $\psi_t$ such that $\psi_1^*( \phi^* \omega ) = \omega$
and $\psi_t(S) = S$.
Then the composition $\theta = \phi \circ \psi_1$
is symplectic, isotopic to the identity and satisfies $\theta(S) = S'$.
It follows again from Moser isotopy that the inclusion $\Symp(\Sigma) \hookrightarrow \Diff^+(\Sigma)$
is an isomorphism on $\pi_0$, hence there is a path $(\theta_t)$ in $\Symp(\Sigma)$
with $\theta_0 = \id$ and $\theta_1 = \theta$.
The restriction of $(\theta_t)$ to $\alpha$ is an exact Lagrangian path,
hence $\alpha$ and $\beta$ are Hamiltonian isotopic.

\end{proof}

For future use, we also record a version of Lemma \ref{lemma:ham_isotopic_embedded_case} that applies to topologically unobstructed curves.

\begin{lemma}
\label{lemma:ham_isotopic_unob_case}
Let $\alpha$ and $\beta$ be homotopic topologically unobstructed curves with $\hol(\alpha) = \hol(\beta)$.
Then there is an exact homotopy $(\gamma_t)_{t \in [0,1]}$ such that $\gamma_0 = \alpha$, $\gamma_1 = \beta$
and each curve $\gamma_t$ is topologically unobstructed.
\end{lemma}
\begin{proof}
Consider the covering map $\rho: \what{\Sigma} \to \Sigma$ associated to the subgroup of $\pi_1(\Sigma, \alpha(0))$ generated by $[\alpha]$.
Equip $\what{\Sigma}$ with the symplectic form $\rho^*\omega$.

By definition, $\alpha$ lifts to a closed curve $\what{\alpha}$ in $\what{\Sigma}$. 
Moreover, a homotopy from $\alpha$ to $\beta$ lifts to a homotopy from 
$\what{\alpha}$ to a closed curve $\what{\beta}$ which is a lift of $\beta$.
Since $\alpha$ and $\beta$ are unobstructed, by Lemma \ref{lemma:criterion_top_unob} their lifts to the universal cover of $\Sigma$ are embedded. 
This implies that $\what{\alpha}$ and $\what{\beta}$ are simple curves.

Since $\what{\alpha}$ and $\what{\beta}$ are homotopic non-contractible simple curves, they are isotopic by Theorem 2.1 of \cite{Epstein66}.
Any isotopy between them has zero flux since its projection to $\Sigma$ is a regular homotopy from $\alpha$ to $\beta$, 
and this homotopy has zero flux by the holonomy assumption and Lemma \ref{lemma:holonomy_primitive_of_flux}.
Moreover, $\what{\alpha}$ represents a generator of $\pi_1(\what{\Sigma}) \iso \Z$, hence the map $H^1(\what{\Sigma}; \R) \to H^1(S^1; \R)$
induced by $\what{\alpha}$ is an isomorphism.
As in the first case of the proof of Lemma \ref{lemma:ham_isotopic_embedded_case}, 
it follows that $\what{\alpha}$ and $\what{\beta}$ are Hamiltonian isotopic.
The projection of a Hamiltonian isotopy from $\what{\alpha}$ to $\what{\beta}$ yields an exact homotopy from $\alpha$ to $\beta$ through unobstructed curves.
\end{proof}

\subsection{Proof of Proposition \ref{prop:holonomy_iso}}
\label{subsection:proof_holonomy_iso}

In this section, we prove Proposition \ref{prop:holonomy_iso}.
First, we show that $\hol|_{\mathcal{I}}$ is surjective, which
is a consequence of the following lemma.

\begin{lemma}
\label{lemma:isotopy_arbitrary_holonomy}
Let $\alpha$ be a non-separating simple curve. 
Then for any $x \in \R$, there exists 
a simple curve $\wtilde{\alpha}$
isotopic to $\alpha$
such that 
$
\hol(\wtilde{\alpha}) - \hol(\alpha) = x.
$
\end{lemma}
\begin{proof}
There is a curve $\beta$ that intersects $\alpha$ transversely in a single point.
Let $A \iso S^1 \times [0,1]$ be a small tubular neighborhood of $\beta$ such that $\alpha \cap A$ is a ray of the annulus.
Then, the required curve $\wtilde{\alpha}$ can be obtained by wrapping $\alpha$ around the annulus by the flow of a vector field of the form
$f(r) \frac{\del}{\del \theta}$, where $f(r) \geq 0$ is a non-constant function that vanishes near $r=0$ and $r=1$.
\end{proof}

It remains to prove that $\hol|_{\mathcal{I}}$ is injective, which
is the hardest part of the proof. 
We will deduce it from the following proposition, which is a generalization of Lemma \ref{lemma:ham_isotopic_embedded_case}.

\begin{proposition}
\label{prop:isotopy_relation}
Suppose that $(\alpha, \wtilde{\alpha})$ and $(\beta, \wtilde{\beta})$ are two pairs of isotopic simple curves with 
$\hol(\wtilde{\alpha}) - \hol(\alpha) = \hol(\wtilde{\beta}) - \hol(\beta)$. 
Then, in $\Gunob(\Sigma)$ we have
\[
[\wtilde{\alpha}] - [\alpha] = [\wtilde{\beta}] - [\beta].
\]
\end{proposition}

This result appears as
Proposition 5.4 and Lemma 5.15 in \cite{Perrier19}.
As explained in Remark \ref{rmk:def_unob_perrier},
the extension of Perrier's result to 
the present case requires additional verifications to show unobstructedness of the relevant cobordisms.
These verifications are straightforward applications of the obstruction criteria proved in Section \ref{section:top_unob_cobordisms}
and Section \ref{section:obstruction_dim2}.
For the benefit of the reader, we include in Appendix \ref{appendix:isotopy} a proof of Proposition \ref{prop:isotopy_relation}
that closely follows the proof in \cite{Perrier19} and incorporates these verifications.

Using Proposition \ref{prop:isotopy_relation}, we can now finish the proof of Proposition \ref{prop:holonomy_iso}.

\begin{proof}[Proof of Proposition \ref{prop:holonomy_iso}.]
By Lemma \ref{lemma:isotopy_arbitrary_holonomy}, $\hol|_{\mathcal{I}}$ is surjective.
To prove injectivity, we claim that any element of $\mathcal{I}$ 
can be written as $[\wtilde{\alpha}] - [\alpha]$ for some isotopic curves $\wtilde{\alpha}$ and $\alpha$ (in other words, the
subset of elements of the form $[\wtilde{\alpha}] - [\alpha]$ is already a subgroup). 
Indeed, a general element of $\mathcal{I}$ can be written as a sum $\sum_{k=1}^N [\wtilde{\beta}_k] - [\beta_k]$,
where $\wtilde{\beta}_k$ and $\beta_k$ are isotopic. 
Fix a non-separating curve $\alpha_0$. 
By Lemma \ref{lemma:isotopy_arbitrary_holonomy}, 
we can find curves $\alpha_k$ for $k=1, \ldots, N$,
such that $\alpha_{k}$ is obtained from $\alpha_{k-1}$ by an isotopy of area $\hol(\wtilde{\beta}_k) - \hol(\beta_k)$.
Then, by Proposition \ref{prop:isotopy_relation}, we have 
\[
\sum_{k=1}^N [\wtilde{\beta}_k] - [\beta_k] = \sum_{k=1}^{N} [\alpha_k] - [\alpha_{k-1}] = [\alpha_N] - [\alpha_0].
\]
This proves the claim.
The injectivity of $\hol|_{\mathcal{I}}$ now follows directly from Proposition \ref{prop:isotopy_relation}.
\end{proof}

%% file: sections/action_mcg.tex
In Section \ref{section:isotopies}, we showed how to reduce the computation of $\Gunob(\Sigma)$
to the computation of the reduced cobordism group 
$\Gred(\Sigma)$ (see Definition \ref{def:reduced_group}).
The action of the group of symplectic diffeomorphisms $\Symp(\Sigma)$ on $\Gunob(\Sigma)$ descends
to an action on $\Gred(\Sigma)$.
By definition, 
the component of the identity $\Symp_0(\Sigma) \subset \Symp(\Sigma)$ 
acts trivially on $\Gred(\Sigma)$.
Therefore, the action of $\Symp(\Sigma)$ descends to an action of the \emph{symplectic mapping class group}
$\Smcg(\Sigma) = \Symp(\Sigma)/\Symp_0(\Sigma)$ on $\Gred(\Sigma)$.
The goal of this section is to describe this action.
This will be a key ingredient in the computation of $\Gred(\Sigma)$, as it will allow us to determine a simple set
of generators of this group (see Proposition \ref{prop:generators}).

The main result of this section is the computation of the action of \emph{symplectic Dehn twists}, which are generators of $\Smcg(\Sigma)$.
For a simple curve $\alpha$, we denote by $T_{\alpha} \in \Smcg(\Sigma)$ the Dehn twist around $\alpha$ (the definition of $T_{\alpha}$ is recalled
in Section \ref{section:def_dehn_twists}).

\begin{theorem}
\label{thm:dehntwistaction}
Let $\alpha$ and $\beta$ be two simple curves in $\Sigma$. 
Then in $\Gred(\Sigma)$ there is the relation
\begin{equation}
\label{eq:dehntwistformula}
T_{\alpha}[\beta] = [\beta] + (\alpha \cdot \beta) [\alpha].
\end{equation}
\end{theorem}

Theorem \ref{thm:dehntwistaction} is proved in two steps.
The first step is to prove the theorem in the special case where the curves intersect twice with zero algebraic intersection.
This is carried out in Section \ref{section:DT_specialcase}.
Then, in Section \ref{section:DT_generalcase}, we show how to deduce the general case of Theorem \ref{thm:dehntwistaction} from
the special case by an inductive argument.

\begin{remark}
On the side of $\Kgrp$, the existence of the Dehn twist relation \eqref{eq:dehntwistformula} is a consequence
of the well-known \emph{Seidel exact triangle} in $\DFuk$
\begin{equation}
\label{eq:seidel_triangle}
\begin{tikzcd}
HF^*(S, L) \tens S \arrow{r} &L \arrow{r} & T_{S} L \arrow{r}{\left[1\right]} &{}
\end{tikzcd}
\end{equation}
where $T_S$ is the Dehn twist around a framed Lagrangian sphere $S$; see \cite[Corollary 17.18]{SeidelBook}.
See also \cite{Abouzaid08} for the existence of the triangle \eqref{eq:seidel_triangle} in the case of curves 
on surfaces that intersect minimally.
We remark that the methods used in the present work
do not recover the exact triangle \eqref{eq:seidel_triangle} on the level of $\DFuk$, but
only the induced relation in $K$-theory.
\end{remark}

\subsection{Dehn twists}
\label{section:def_dehn_twists}

We recall the definition of the Dehn twist $T_{\alpha}$ around a simple curve $\alpha$.
Choose a symplectic embedding $\psi: [-\delta, \delta] \times S^1 \to \Sigma$ such that $\psi|_{ \{ 0 \} \times S^1} = \alpha$,
where the annulus is equipped with the standard symplectic form $dt \wedge d\theta$.
Consider the symplectic diffeomorphism of the annulus given by 
\[
(t, \theta) \longmapsto (t, \theta - f(t)),
\]
where $f: [-\delta, \delta] \to \R$ is a smooth increasing function which is $0$ near $-\delta$ and $1$ near $\delta$. 
Using the embedding $\psi$, this map extends by the identity to a symplectomorphism $T_{\alpha}$ of $\Sigma$, 
called a (right-handed) \emph{symplectic Dehn twist} around $\alpha$. 
The class of this map in $\Smcg(\Sigma)$ does not depend on the choices of $\psi$ and $f$ above. 
We will also denote this class by the symbol $T_{\alpha}$, whenever no confusion arises.

\begin{remark}
We warn the reader that our convention for the direction of the Dehn twist is opposite 
to that of \cite{Abouzaid08} and \cite{Perrier19}.
\end{remark}

By Moser's Theorem, the inclusion of $\Symp(\Sigma)$ inside the group of orientation-preserving diffeomorphisms 
$\Diff^+(\Sigma)$
induces an isomorphism between
$\Smcg(\Sigma)$ and the classical mapping class group 
$\Mcg(\Sigma) = \Diff^+(\Sigma)/\Diff_0(\Sigma)$,
where $\Diff_0(\Sigma)$ is the subgroup of diffeomorphisms isotopic to the identity.
In what follows, we will identify $\Smcg(\Sigma)$ and $\Mcg(\Sigma)$ using this isomorphism.

It is a well-known theorem of Dehn and Lickorish that $\Mcg(\Sigma)$ is generated by the classes of finitely many Dehn twists. 
(An explicit set of generators is introduced in Section \ref{subsection:generators_reduced_group}.)
It follows that $\Smcg(\Sigma)$ is generated by symplectic Dehn twists.

\subsection{Proof of Theorem \ref{thm:dehntwistaction}: curves intersecting twice}
\label{section:DT_specialcase}

In this section, we prove the following special case of Theorem \ref{thm:dehntwistaction}.

\begin{proposition}
\label{prop:dehntwist2pts}
Let $\alpha$ and $\beta$ be simple curves which are in minimal position and intersect in two points of opposite signs.
Then, in $\Gred(\Sigma)$ we have
$
T_{\alpha}[\beta] = [\beta] + (\alpha \cdot \beta) [\alpha].
$
\end{proposition}

Proposition \ref{prop:dehntwist2pts} is proved by realizing the Dehn twist as an iterated surgery.
This is a well-known procedure; see for example \cite[Section 3.1.1]{FarbMargalit}.
We will follow the description given by Perrier in \cite[Proposition 2.15]{Perrier19}, which we recall in order to fix some notation.

Let $\alpha$ and $\beta$ be curves satisfying the hypothesis of Proposition \ref{prop:dehntwist2pts}.
Let $N$ be a closed regular neighborhood of $\alpha \cup \beta$.\footnote{To be more precise, we mean that $N = A \cup B$, where
$A$ and $B$ are closed tubular neighborhoods of $\alpha$ and $\beta$, respectively. 
Moreover, $A$ and $B$ are chosen small enough so that $A \cap B$ is a disjoint
union of disks, and inside each such disk the curves are in standard position (i.e. look like the coordinate axes of $\R^2$).}
Note that $N$ is a four-holed sphere.
We perform the following surgeries, which are illustrated in Figure \ref{fig:surgery_DT}.
Let $s_1$ be the positive intersection of $(\alpha, \beta)$, and $q_1$ the negative intersection. 
The first step is to form the surgery $\sigma := \alpha \#_{s_1} \beta$.
Next, let $\wtilde{\alpha}$ be a $C^1$-small Hamiltonian perturbation of $\alpha^{-1}$ which intersects $\sigma$ in two points.
This perturbation is chosen so that one of these intersection points, denoted $s_2$, lies on the arc of $\beta$ near $q_1$, and the other, denoted $q_2$, lies on the arc of $\alpha$ near $q_1$.
Finally, we let $\tau := \widetilde{\alpha} \#_{s_2} \sigma$. 
The size of the surgery handle is chosen big enough so that $\tau$ is embedded.

\begin{figure}
	\centering
	\includegraphics[width=\textwidth]{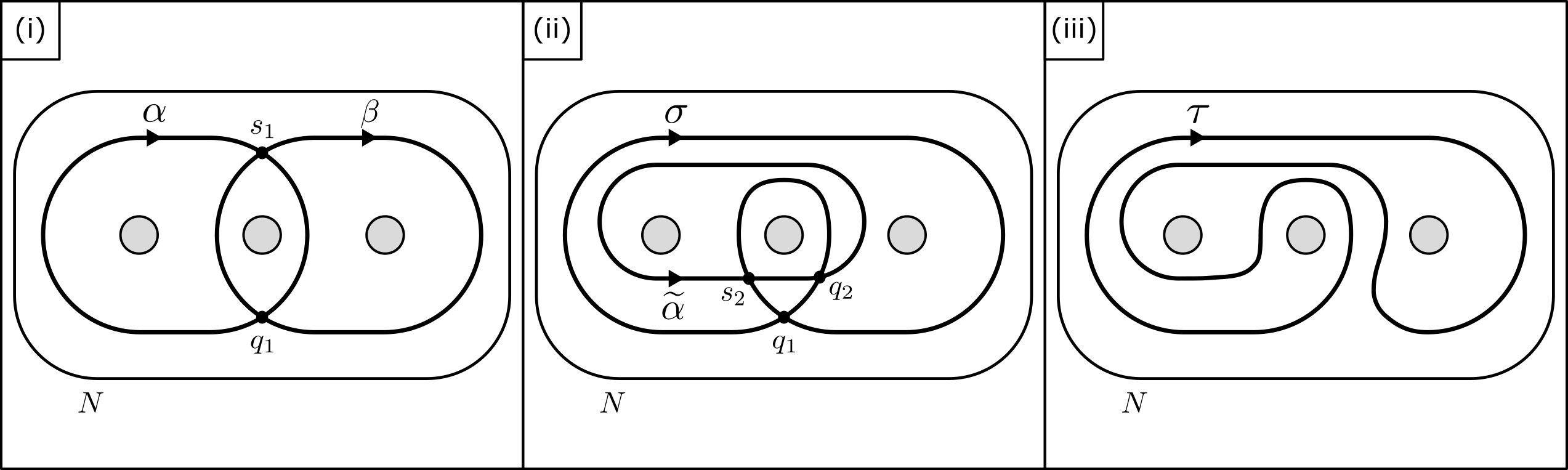}
	\caption{The surgery procedure used in the proof of Proposition \ref{prop:dehntwist2pts}.}
	\label{fig:surgery_DT}
\end{figure}
 
It is easy to check that $\tau$ is isotopic to $T_{\alpha} \beta$.
We denote the traces of the surgeries by $S_1 = S(\alpha, \beta; s_1)$ and $S_2 = S(\widetilde{\alpha}, \sigma; s_2)$. 
Concatenating $S_1$ and $S_2$ produces an immersed cobordism $\tau \cob (\widetilde{\alpha}, \alpha, \beta)$.
To prove Proposition \ref{prop:dehntwist2pts}, we will show that there is an unobstructed cobordism with the same ends.
By Proposition \ref{prop:concatenation_unobstructed}, it suffices to check that the surgery cobordisms
$S_1$ and $S_2$ constructed above are topologically unobstructed. 
Hence, to finish the proof of Proposition \ref{prop:dehntwist2pts}, it suffices to show the following lemma.

\begin{lemma}
\label{lemma:surgeriesareunobstructed}
Assume that $\alpha$ and $\beta$ are in minimal position. 
Then the surgery cobordisms $S_1$ and $S_2$ constructed above are topologically unobstructed.
\end{lemma}

Before giving the proof of the lemma, we make some preliminary observations. 
It will be convenient to perform the necessary computations inside the regular neighborhood $N$.
To relate computations in $N$ to those in $\Sigma$, we will use the fact 
that $N$ is incompressible in $\Sigma$ (recall from Definition \ref{def:incompressible}
that this means that the maps $\pi_1(N) \to \pi_1(\Sigma)$ induced by the inclusion are injective).

\begin{lemma}
Assume that $\alpha$ and $\beta$ are in minimal position. 
Then the regular neighborhood $N$ of $\alpha \cup \beta$ is incompressible in $\Sigma$.
\end{lemma}
\begin{proof}
Since the curves $\alpha$ and $\beta$ are in minimal position, it follows from the
Bigon Criterion \cite[Proposition 1.7]{FarbMargalit} that they do not bound bigons.
In particular, this implies that the boundary circles of $N$ are non-contractible in $\Sigma$. 
The incompressibility of $N$ is then a consequence of the following well-known fact.

\begin{lemma}
\label{lemma:incompressible}
Let $M$ be a smooth compact manifold and $X \subset M$ a smooth codimension $0$ compact submanifold with boundary. 
If $\bdry X$ is incompressible in $M$, then $X$ is incompressible in $M$.
\end{lemma}

We refer to the appendix of \cite{GanorTanny} for an elegant proof of this lemma.
In the case of surfaces, $\bdry X$ is the union of embedded circles. 
Since the fundamental group of an orientable surface has no torsion, the condition that
$\bdry X$ be incompressible is equivalent to the non-contractibility of each component of $\bdry X$. 
Therefore, we conclude from Lemma \ref{lemma:incompressible} that
the regular neighborhood $N$ above is incompressible.
\end{proof}

The incompressibility of $N$ implies that for a subset $A \subset N$, the inclusions induce
an isomorphism $\pi_2(N, A, *) \to \pi_2(\Sigma, A, *)$. 
This follows from comparison of the homotopy exact sequences of the pairs $(N,A)$ and $(\Sigma, A)$.

We now proceed to the proof that the cobordisms $S_i$ are topologically unobstructed.

\begin{proof}[Proof of Lemma \ref{lemma:surgeriesareunobstructed}]
\hfill

\noindent\textit{$S_1$ is topologically unobstructed.}

The group $\pi_1(N, s_1)$ is free of rank $3$. 
The classes $[\alpha]$ and $[\beta]$ belong to a basis of $\pi_1(N, s_1)$, hence span a free group of rank $2$. 
Since $N$ is incompressible, they also span a free group of rank $2$ in $\pi_1(\Sigma, s_1)$. 
Hence, by Proposition \ref{prop:relativepi2}, $S_1$ is incompressible.

Moreover, we have $\pi_2(N, \alpha \cup \beta, s_1) = 0$, hence also $\pi_2(\Sigma, \alpha \cup \beta, s_1) =0$. 
By Lemma \ref{lemma:countingcorners}, $\alpha$ and $\beta$ do not bound $s_1$-marked teardrops. 
We conclude from Proposition \ref{prop:teardropcriterion} that $S_1$ does not bound teardrops.

\noindent\textit{$S_2$ is topologically unobstructed.}

Again, it can be checked directly that $[\wtilde{\alpha}]$ and $[\sigma]$ belong to a basis of $\pi_1(N,s_2)$.
By the same argument, it follows that $S_2$ is incompressible.

It remains to show that $\wtilde{\alpha}$ and $\sigma$ do not bound $s_2$-marked teardrops.
The group $\pi_2(\Sigma, \wtilde{\alpha} \cup \sigma, s_2) \iso \pi_2(N, \wtilde{\alpha} \cup \sigma, s_2)$ is isomorphic to $\Z$ and is spanned by the class of the obvious triangle with boundary on $\widetilde{\alpha} \cup \sigma$. 
Since this triangle has three distinct corners, Lemma \ref{lemma:countingcorners} ensures 
that no classes in $\pi_2(\Sigma, \wtilde{\alpha} \cup \sigma, s_2)$ can be represented by $s_2$-marked teardrops.

\end{proof}

\subsection{Proof of Theorem \ref{thm:dehntwistaction}: general case}
\label{section:DT_generalcase}

We now show how to deduce the general case of Theorem \ref{thm:dehntwistaction} from the special case of Proposition \ref{prop:dehntwist2pts}.
The proof uses an inductive argument based on the following lemma, due to Lickorish.

\begin{lemma}[Lickorish \cite{Lick64}]
\label{lemma:lickorish_trick}
Let $\alpha$ and $\beta$ be simple curves in general position with $N \geq 2$ intersection points.
In the case $N = 2$, assume that the intersections have the same sign.
Then there exists a simple curve $\gamma$ with the following properties:
\begin{enumerate}[label = (\roman*), font=\normalfont]
	\item $\gamma$ intersects $\beta$ transversely in one point or two points of opposite signs.
	\item $\gamma$ intersects $\alpha$ transversely in at most $N-1$ points.
	\item The Dehn twist $T_{\gamma} \beta$ is isotopic to a curve $\delta$ that intersects $\alpha$ transversely in at most $N-1$ points.
\end{enumerate}
\end{lemma}
\begin{proof}
The idea of the proof is that $\alpha$ and $\beta$ either have 
\begin{enumerate}[label = (\alph*)] 
	\item two intersections that are consecutive along $\alpha$ and have the same sign, or
	\item three intersections that are consecutive along $\alpha$ and have alternating signs.
\end{enumerate} 
In each case, a curve $\gamma$ satisfying the above conditions is represented in Figure \ref{fig:lickorish_trick}.
We refer to the proof of Lemma 2 of \cite{Lick64} for more details.

\begin{figure}
	\centering
	\includegraphics[width=\textwidth]{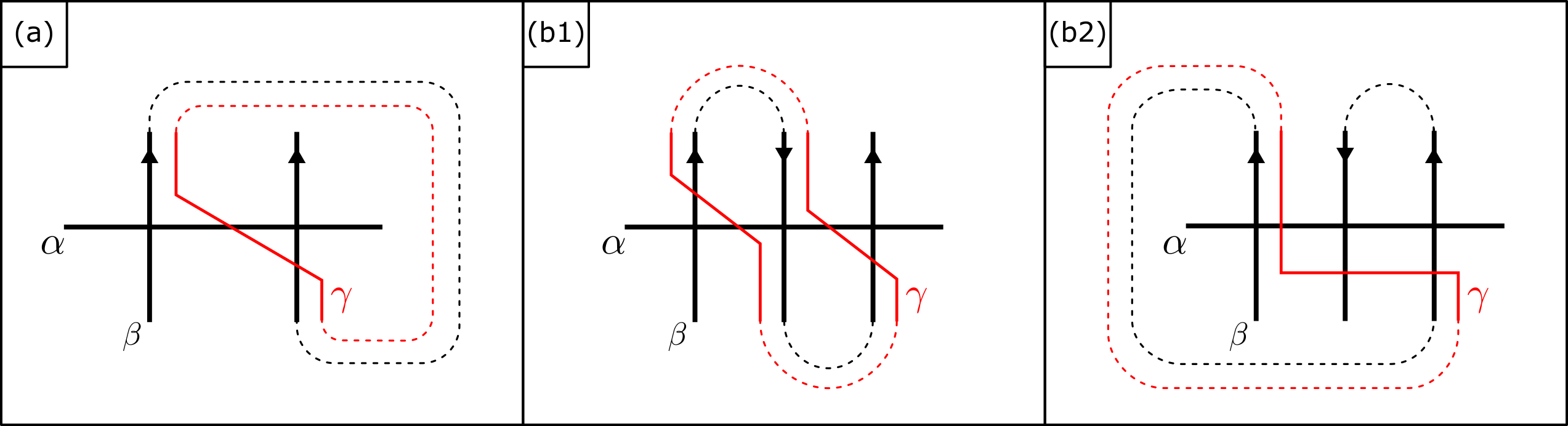}
	\caption{The choice of the curve $\gamma$ in the proof of Lemma \ref{lemma:lickorish_trick}.}
	\label{fig:lickorish_trick}
\end{figure}
\end{proof}

\begin{proof}[Proof of Theorem \ref{thm:dehntwistaction}.]
The proof is by induction over the geometric intersection number $N$ of the curves. 
By moving $\alpha$ and $\beta$ by isotopies if necessary, we may assume that they are in minimal position.

We start by handling a few base cases.
For $N=0$, Formula \eqref{eq:dehntwistformula} follows from the fact that $T_{\alpha} \beta$ is isotopic to $\beta$.

For $N=1$,
Formula \eqref{eq:dehntwistformula} follows from the fact that, in this case, the Dehn twist $T_{\alpha} \beta$ 
is isotopic either
to the surgery $\alpha \# \beta$ if $\alpha \cdot \beta > 0$, 
or to the surgery $\alpha^{-1}\# \beta$ if $\alpha \cdot \beta < 0$. 
The resulting surgery cobordism is embedded and therefore bounds no teardrops. 
Moreover, it follows from Corollary \ref{cor:nonzerointersectionimpliesnodisks} that the cobordism is incompressible.
Hence, the surgery cobordism is topologically unobstructed, and therefore unobstructed.

In the case where $N=2$ and $\alpha \cdot \beta = 0$, the
theorem reduces to Proposition \ref{prop:dehntwist2pts}.

We deal with the remaining cases by induction, using Lemma \ref{lemma:lickorish_trick} to introduce auxiliary curves whose role is to
reduce the number of intersections.
Suppose then that $\alpha$ and $\beta$ have geometric intersection number $N \geq 2$, and in the case $N=2$ that $\alpha \cdot \beta \neq 0$.
As above, we assume that $\alpha$ and $\beta$ are in minimal position.
Let $\gamma$ and $\delta$ be curves satisfying the conclusion of Lemma \ref{lemma:lickorish_trick}.

We claim that $\delta$ and $\gamma$ are non-contractible.
Indeed, the curve $\delta$ is isotopic to $T_{\gamma} \beta$, which is non-contractible since $\beta$ is non-contractible.
To see that $\gamma$ is non-contractible, note that the geometric intersection number of $T_{\gamma} \beta$ and $\alpha$ is at most $N-1$, 
which implies that $T_{\gamma} \beta$ is not isotopic to $\beta$. 
In turn, this implies that $T_{\gamma}$
is not isotopic to the identity, hence that $\gamma$ is non-contractible.

Using the $N=1$ and $N=2$ cases proven before, we have $[T_{\gamma} \beta] = [\beta] + (\gamma \cdot \beta) [\gamma]$.
Since $T_{\gamma} \beta$ is isotopic to $\delta$, this gives the relation 
\begin{equation}
\label{eq:relationdelta}
[\delta] = [\beta] + (\gamma \cdot \beta) [\gamma].
\end{equation}
Applying the Dehn twist along $\alpha$, we then obtain:
\[
T_{\alpha} [\beta] = T_{\alpha} [\delta] - (\gamma \cdot \beta) T_{\alpha} [\gamma].
\]
Since $\alpha$ intersects $\delta$ and $\gamma$ in less than $N$ points, we can use the induction hypothesis and deduce that
\[
T_{\alpha} [\beta] = [\delta] + (\alpha \cdot \delta) [\alpha] - (\gamma \cdot \beta) \left([\gamma] + (\alpha \cdot \gamma)[ \alpha] \right).
\]
Using relation \eqref{eq:relationdelta} again, this last equality simplifies to $T_{\alpha}[\beta] = [\beta] + (\alpha \cdot \beta) [\alpha]$.
\end{proof}

%% file: sections/computation.tex
In this section, we combine the results of Sections \ref{section:obstruction_dim2} -- \ref{section:action_mcg} to complete the computation of 
the cobordism group $\Gunob(\Sigma)$.
We will show that the map
\[
\Gunob(\Sigma) \to H_1(S \Sigma; \Z) \oplus \R  
\]
defined in Section \ref{section:invariants} is an isomorphism.

Given the results of Sections \ref{section:invariants} -- \ref{section:action_mcg}, 
the structure of the proof is identical to the computation of $K_0 \DFuk(\Sigma)$ in Section 7 and Section 8 of \cite{Abouzaid08}.
The only new ingredient that is needed is the verification that the surgeries appearing in Lemma \ref{lemma:pair_of_pants_relation} below, which corresponds to Lemma 7.6 in \cite{Abouzaid08}, give rise to unobstructed cobordisms.
In the rest of this section, we recall Abouzaid's computation and indicate the adjustments needed to deal with obstruction.

\begin{figure}
	\centering
	\includegraphics[width = 0.8\textwidth]{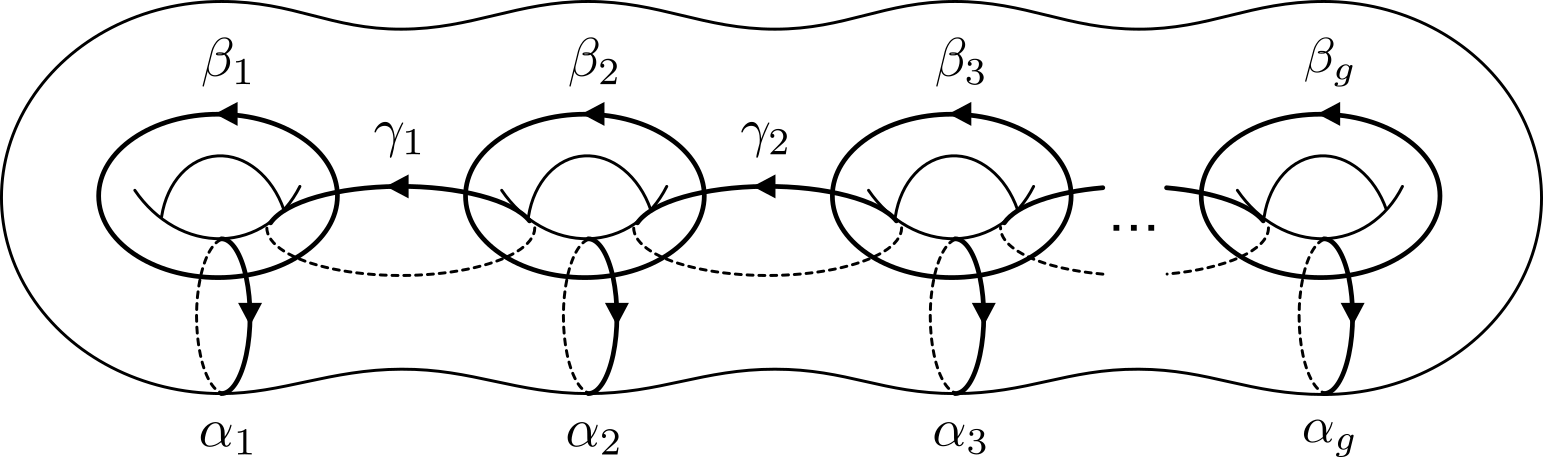}
	\caption{The Lickorish generators of $\Mcg(\Sigma)$.}
	\label{fig:Lickorish}
\end{figure}

\subsection{Generators of $\Gred(\Sigma)$}
\label{subsection:generators_reduced_group}

The first step is to determine a simple set of generators for $\Gred(\Sigma)$.
These generators will be obtained from a set of generators of the mapping class group of $\Sigma$.
By a well-known theorem of Lickorish \cite{Lick64}, $\Mcg(\Sigma)$ is generated by the Dehn twists about the $3g-1$ curves 
$\alpha_1, \ldots, \alpha_g$, $\beta_1, \ldots, \beta_g$ and $\gamma_1, \ldots, \gamma_{g-1}$ that are represented in Figure \ref{fig:Lickorish}.

To describe the generators of $\Gred(\Sigma)$, we will first need the following lemma.

\begin{lemma}
\label{lemma:class_of_boundary}
Suppose that $\gamma$ and $\gamma'$ are the oriented boundaries of diffeomorphic compact surfaces in $\Sigma$.
Then $[\gamma] = [\gamma']$ in $\Gred(\Sigma)$.
\end{lemma}
\begin{proof}
There is an orientation-preserving diffeomorphism of $\Sigma$ that takes $\gamma$ to $\gamma'$. 
This diffeomorphism can be written
as the composition of an ambient isotopy and a sequence of symplectic Dehn twists (see Section \ref{section:def_dehn_twists}). 
It follows from the Dehn twist relation \eqref{eq:dehntwistformula} that $[\gamma] = [\gamma']$ in $\Gred(\Sigma)$. 
\end{proof}

We let $T \in \Gred(\Sigma)$ be the class represented by curves bounding genus $1$ surfaces in $\Sigma$.
This is well-defined by Lemma \ref{lemma:class_of_boundary}. 
We can now state the main result of this section.

\begin{proposition}
\label{prop:generators}
$\Gred(\Sigma)$ is generated by $T$ and the Lickorish generators $[\alpha_i]$ and $[\beta_i]$ for $i=1, \ldots, g$. 
\end{proposition}

The proof of the proposition relies on Lemma \ref{lemma:pair_of_pants_relation} and Lemma \ref{lemma:class_separating_curve} below.

\begin{lemma}
\label{lemma:pair_of_pants_relation}
Suppose that the curves $\sigma_1$, $\sigma_2$ and $\sigma_3$ form the oriented boundary of a pair of pants.
Moreover, assume that at most one of the three curves is separating.
Then, in $\Gred(\Sigma)$ there is the relation
\begin{equation}
\label{eq:pair_of_pants_relation}
[\sigma_1] + [\sigma_2] + [\sigma_3] = T.
\end{equation}
\end{lemma}

\begin{proof}
The proof uses the same surgeries as the proof of Lemma 7.6 in \cite{Abouzaid08}, which we recall for convenience.
We need to show that the associated surgery cobordisms are unobstructed.
The proof splits in two cases: the case where $\Sigma$ has genus $g \geq 3$ and the case where $\Sigma$ has genus $2$.
We consider first the case $g \geq 3$.

\begin{figure}[b]
	\centering
	\includegraphics[scale = 1]{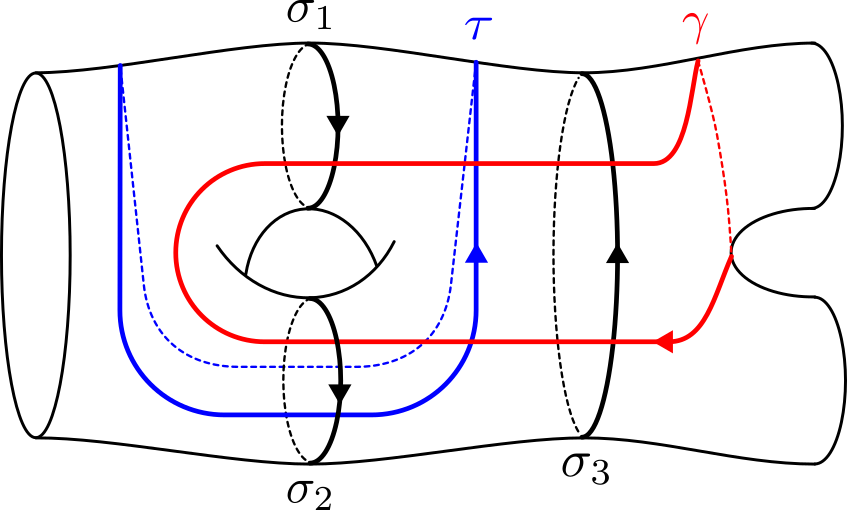}
	\caption{The curves used in the proof of Lemma \ref{lemma:pair_of_pants_relation} in the case $g \geq 3$.}
	\label{fig:curves_lemma_pop}
\end{figure}

First, we can assume that none of the curves $\sigma_i$ are isotopic to each other.
Indeed, if two of the curves are isotopic then the third curve bounds a torus
and \eqref{eq:pair_of_pants_relation} follows directly
from the definition of $T$.

Let $P$ be the pair of pants bounded by the curves $\sigma_i$.
By assumption, at most one of the curves is separating. 
Without loss of generality, assume
that $\sigma_1$ and $\sigma_2$ are non-separating.
Then there is a curve $\sigma_4$ such that $\sigma_1$, $\sigma_2$ and $\sigma_4$ together bound
a pair of pants $P'$ whose interior is disjoint from $P$.

Let $\tau$ be a curve in $P \cup P'$ that bounds a torus $T$ and whose intersection pattern with the $\sigma_i$ is specified in 
Figure \ref{fig:curves_lemma_pop}.
We can assume that $\tau$ satisfies the holonomy condition
\begin{equation}
\label{eq:holonomy_tau}
\hol(\tau) = \hol(\sigma_1) + \hol(\sigma_2) + \hol(\sigma_3).
\end{equation}
Indeed, by Lemma \ref{lemma:holonomy+winding_bounding_curves}, this condition is equivalent to $\area(T) = \area(P)$, which can
be arranged by an isotopy of $\tau$ without changing the intersection pattern.

Since $\Sigma$ has genus at least $3$, we can find a further curve $\gamma$ which is in minimal position with respect
to the curves $\sigma_i$ and $\tau$, and
whose intersection pattern is specified in
Figure \ref{fig:curves_lemma_pop}.

Next, we perform the surgeries represented in Figure \ref{fig:surgery_pop}.
In the left column of Figure \ref{fig:surgery_pop}, we perform surgeries between $\sigma_1^{-1}$, $\gamma$ and $\tau$, 
producing a curve $\delta$.
In the right column of Figure \ref{fig:surgery_pop}, we perform surgeries
between $\gamma$, $\sigma_2$ and $\sigma_3$, producing a curve $\delta'$.

The curves $\delta$ and $\delta'$ are topologically unobstructed and regularly homotopic.
Moreover, by \eqref{eq:holonomy_tau} their holonomies satisfy
\[
\hol(\delta') = \hol(\gamma) + \hol(\sigma_2 ) + \hol(\sigma_3) = \hol(\gamma) + \hol(\tau) - \hol(\sigma) = \hol(\delta).
\]
By Lemma \ref{lemma:ham_isotopic_unob_case}, $\delta$ and $\delta'$ are exact homotopic through topologically unobstructed curves.
By concatenating the surgery cobordisms obtained above with the suspension of an unobstructed exact homotopy from $\delta$ to $\delta'$, we obtain an immersed Lagrangian cobordism $(\sigma_2, \gamma, \sigma_3) \cob (\tau, \gamma, \sigma_1^{-1} )$. 

To finish the proof, it remains to show that there is an unobstructed cobordism with the same ends.
By Proposition \ref{prop:concatenation_unobstructed}, it suffices to prove that each step in the above procedure produces a topologically unobstructed cobordism.
The surgeries labelled L1 and R1 in Figure \ref{fig:surgery_pop} involve curves intersecting once, hence they are topologically unobstructed by Corollary \ref{cor:nonzerointersectionimpliesnodisks}.
The surgeries labelled L2 and R2 involve curves which intersect twice.
Moreover, these curves are in minimal position since 
$\gamma$ was chosen to be in minimal position with respect to $\tau$ and $\sigma_3$.
Hence, we are in the same situation as in Section \ref{section:DT_specialcase}, and we conclude that these surgeries are topologically unobstructed
by Lemma \ref{lemma:surgeriesareunobstructed}.
Finally, the exact homotopy between $\delta$ and $\delta'$ is topologically unobstructed by Lemma \ref{lemma:obstruction_suspension}.

Suppose now that $\Sigma$ has genus $2$.
If one of the $\sigma_i$ is separating, then it bounds a genus $1$ surface. In this case, the other two curves
are isotopic and \eqref{eq:pair_of_pants_relation} follows.
Suppose now that all the $\sigma_i$ are non-separating.
Then, up to a diffeomorphism, we are in the situation of Figure \ref{fig:curves_pop_genus2}.
To obtain Equation \eqref{eq:pair_of_pants_relation}, apply the same surgery procedure as above using the curves 
$\tau$ and $\gamma$ indicated on the right of Figure \ref{fig:curves_pop_genus2}.

\begin{figure}[hp]
	\centering
	\includegraphics[width = \textwidth]{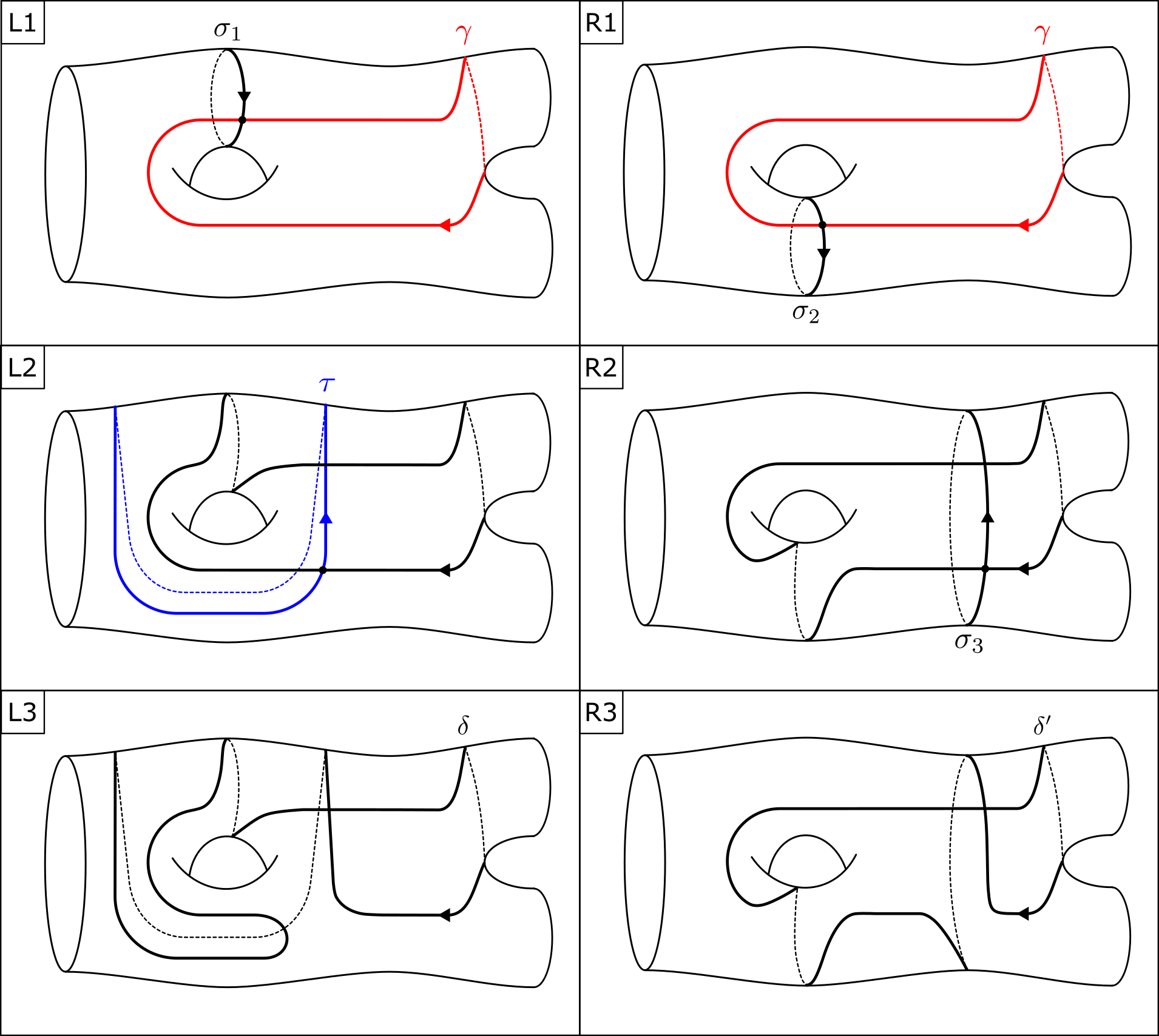}
	\caption{The surgeries used in the proof of Lemma \ref{lemma:pair_of_pants_relation} in the case $g \geq 3$. This figure is adapted from Figure 15 of \cite{Abouzaid08}.}
	\label{fig:surgery_pop}
\end{figure}

\begin{figure}[hp]
	\centering
	\includegraphics[width = 0.9\textwidth]{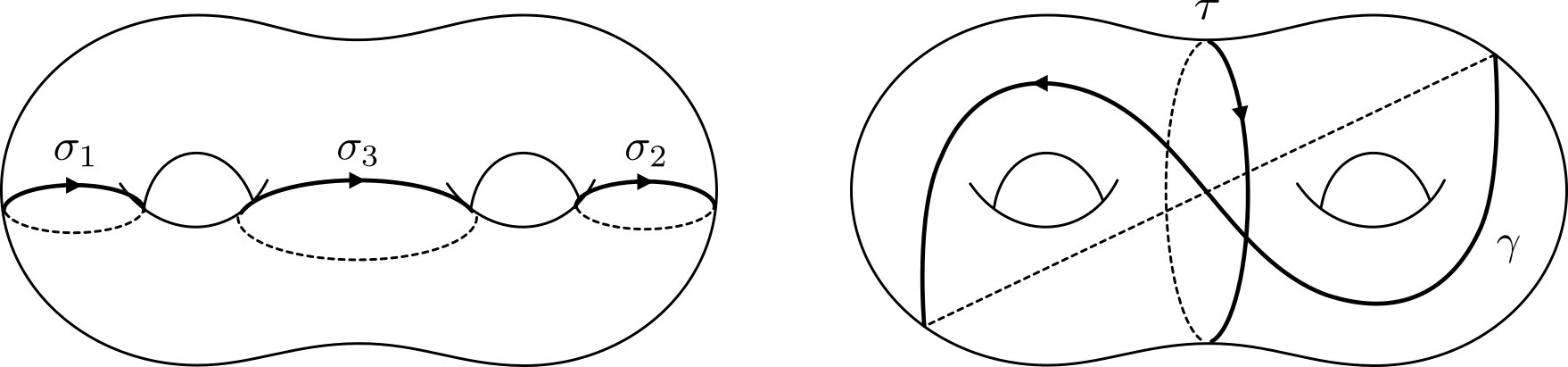}
	\caption{The curves used in the proof of Lemma \ref{lemma:pair_of_pants_relation} in the case $g = 2$. This figure is adapted from Figure 16 of \cite{Abouzaid08}.}
	\label{fig:curves_pop_genus2}
\end{figure}

\end{proof}

\begin{lemma}
\label{lemma:class_separating_curve}
Suppose that $\gamma$ is the oriented boundary of an embedded surface $S \subset \Sigma$.
Then $[\gamma] = -\chi(S) \, T$ in $\Gred(\Sigma)$.
\end{lemma}
\begin{proof}
The proof is by induction on the genus of $S$. 
If $S$ has genus $1$, this follows from the definition of $T$.
Suppose now that $S$ has genus $g(S) > 1$.
We can write $S = S' \cup Q$, where
\begin{enumerate}
	\item $S'$ is a compact surface of genus $g(S') = g(S) -1$ and $1$ boundary component $\gamma'$.
	\item $Q$ is a compact surface of genus $1$ with $2$ boundary components $\gamma$ and $\gamma'$.
	\item $Q \cap S' = \gamma'$.
\end{enumerate}
Furthermore, the surface $Q$ can be decomposed as the union of two
pairs of pants whose common boundary components are two non-separating curves $\alpha$ and $\beta$.
Applying Lemma \ref{lemma:pair_of_pants_relation} to these pairs of pants and using the induction hypothesis, it follows that
\begin{equation*}
[\gamma] 
= T - [\alpha] - [\beta]
= 2 T + [\gamma']
= 2T -\chi(S')T
= -\chi(S) T.
\end{equation*}

\end{proof}

\begin{corollary}
\label{cor:order_of_T}
$T$ has order $\chi(\Sigma)$ in $\Gred(\Sigma)$.
\end{corollary}
\begin{proof}
Suppose that $\gamma$ bounds a torus in $\Sigma$. Then $\gamma^{-1}$ bounds a surface $S$ with $\chi(S) = \chi(\Sigma)-1$. 
By Lemma \ref{lemma:class_separating_curve}, we then have
\[
-T = [\gamma^{-1}] = -( \chi(\Sigma) - 1 ) \, T,
\]
from which it follows that $\chi(\Sigma) T = 0$. 
Moreover, $\mu(T) = -1 \in \Z/ \chi(\Sigma)\Z$ by Lemma \ref{lemma:holonomy+winding_bounding_curves}. 
Hence $T$ has order $\chi(\Sigma)$.
\end{proof}

We can now complete the proof of Proposition \ref{prop:generators}.

\begin{proof}[Proof of Proposition \ref{prop:generators}.]

If $\gamma$ is a separating curve, then by Lemma \ref{lemma:class_separating_curve} $[\gamma]$ lies in the subgroup generated by $T$.

Suppose now that $\gamma$ is a non-separating curve. 
Since $\Diff^+(\Sigma)$ acts transitively on non-separating curves, $\gamma$
can be obtained from $\alpha_1$ by an isotopy and a sequence of Dehn twists about the Lickorish generators.
By the Dehn twist formula \eqref{eq:dehntwistformula},
$[\gamma]$ belongs to the subgroup generated by the classes of the Lickorish generators. 
By Lemma \ref{lemma:pair_of_pants_relation}, the Lickorish generators satisfy the relation $-[\alpha_i] + [\gamma_i] + [\alpha_{i+1}] = T$.
Hence, we conclude that $[\gamma]$ belongs to the subgroup generated by $T$ and $[\alpha_i]$, $[\beta_i]$, $i=1 , \ldots , g$.
\end{proof}

\subsection{End of the proofs of the main results}

\begin{proof}[Proof of Theorem \ref{thm:computation_Gunob}]

By the results of Section \ref{section:isotopies}, there is a diagram with split-exact rows
\[
\begin{tikzcd}
0 \arrow{r} &\R \arrow{r}{i} \arrow{d}{\id}  &\Gunob(\Sigma) \arrow{r} \arrow{d}            &\Gred(\Sigma) \arrow{r} \arrow{d}{\vphi}   &0 \\
0 \arrow{r} &\R \arrow{r}			         &H_1(S \Sigma; \Z) \oplus \R  \arrow{r} 		&H_1(S \Sigma; \Z) \arrow{r}   				&0
\end{tikzcd}
\]
Hence, it suffices to prove that the map $\vphi: \Gred(\Sigma) \to H_1(S \Sigma; \Z)$ is an isomorphism.
This map is surjective by Corollary \ref{cor:homology_map_surjective}.
To prove that it is injective, suppose that $A$ belongs to the kernel of $\vphi$.
By Proposition \ref{prop:generators}, we can write $A$ in terms of the generators
\[
A = k \, T + \sum_{i=1}^g n_i \, [\alpha_i] + m_i [\beta_i].
\]
By projecting to $H_1(\Sigma; \Z)$, we have the relation
\[
\sum_{i=1}^g n_i \, [\alpha_i] + m_i [\beta_i] = 0 \in H_1(\Sigma; \Z).
\]
Since the classes $[\alpha_i]$ and $[\beta_i]$ form a basis of $H_1(\Sigma; \Z)$, 
we deduce that $n_i = m_i =0$ for all $i=1, \ldots , g$. 
Hence $A = k \, T$.
By Lemma \ref{lemma:holonomy+winding_bounding_curves}, we then have $0 = \mu(A) = -k \mod \chi(\Sigma)$.
Since $T$ has order $\chi(\Sigma)$ by Corollary \ref{cor:order_of_T}, we conclude that $A = 0$.
\end{proof}

\begin{proof}[Proof of Corollary \ref{cor:isomorphism_Kgroup}]
By Proposition 6.1 of \cite{Abouzaid08}, the invariants introduced in Section \ref{section:invariants}
descend to surjective morphisms $K_0 \DFuk(\Sigma) \to \R$ and $K_0 \DFuk(\Sigma) \to H_1(S\Sigma; \Z)$.
Hence, there is a diagram
\[
\begin{tikzcd}
\Gunob(\Sigma) \arrow[twoheadrightarrow]{r}{\Theta} \arrow{dr}[swap]{\iso} &K_0 \DFuk(\Sigma) \arrow[twoheadrightarrow]{d} \\
& H_1(S \Sigma; \Z) \oplus \R
\end{tikzcd}
\]
It follows that $\Theta$ is an isomorphism.
\end{proof}

\begin{proof}[Proof of Corollary \ref{cor:computation_Gimm}]
We have the following diagram:
\[
\begin{tikzcd}
\Gunob(\Sigma) \arrow{r} \arrow{dr}[swap]{\iso}	&\Gimm(\Sigma) \arrow{d} \\
												&H_1(S\Sigma; \Z) \oplus \R
\end{tikzcd}
\]
This shows that the map $\Gunob(\Sigma) \to \Gimm(\Sigma)$ is injective.
To show surjectivity, it suffices to prove that $\Gimm(\Sigma)$ is generated
by the classes of non-contractible simple curves.
To see this, note first that $\Gimm(\Sigma)$ is generated by simple curves.
Indeed, any class in $\Gimm(\Sigma)$ is represented by an immersed curve in general position.
Doing surgery at each self-intersection point of this curve yields a Lagrangian cobordism to a collection of
disjoint simple curves.

Next, we prove that any contractible simple curve is a sum of non-contractible simple curves.
To see this, pick two simple curves $\alpha$ and $\beta$ that are independent in $H_1(\Sigma; \Z)$,
have zero geometric intersection and intersect in two points. 
After reversing the orientation of one of the curves if necessary, doing the surgery at the two intersection points
yields a simple curve $\gamma$ in the homology class $[\alpha]+ [\beta]$ and a contractible simple curve $\delta$.
There is a Lagrangian cobordism $(\alpha, \beta) \cob (\gamma, \delta)$, which shows
that $[\delta]$ is a sum of non-contractible simple curves.
By Lemma \ref{lemma:ham_isotopic_embedded_case}, any contractible simple curve which has
the same holonomy as $\delta$ is Hamiltonian isotopic to $\pm \delta$.
Moreover, by isotoping $\alpha$ and $\beta$, we may produce in this way contractible simple curves of arbitrarily small area.
This proves the claim for contractible curves of sufficiently small area.
The general case follows by writing an arbitrary contractible curve
as a sum of curves of small area.
\end{proof}

%% file: sections/appendix_action.tex
In this appendix, we define real-valued Lagrangian cobordism invariants of oriented Lagrangian immersions $L \looparrowright M$,
where $M$ is a monotone symplectic manifold.
These invariants arise as characteristic numbers associated to a cohomology class
which generalizes the \emph{action class} $[\lambda|_L] \in H^1(L; \R)$ associated to a Lagrangian immersion into an
exact symplectic manifold $(M, d\lambda)$.
The invariants defined here also generalize the \emph{holonomy} of an immersed curve in a closed surface of genus $g \neq 1$, as defined
in Section \ref{subsection:holonomy}.

\begin{remark}
The reader should be aware that we will only make use of the special case of surfaces of genus $g \geq 2$ in this paper.
However, the case of general monotone manifolds is of independent interest and, to the best of the author's knowledge, 
does not appear in the litterature.
\end{remark}

For the rest of this appendix, we assume that $(M, \omega)$ is a \emph{monotone} symplectic manifold, in the sense that
there is a constant $\tau \in \R$ such that $[\omega] = \tau \, c_1(M)$ as elements of $H^2(M; \R)$.

\subsection{Definition of the action class}

Let $\Grass M$ be the (unoriented) Lagrangian Grassmannian of $M$ and $\pi: \Grass M \to M$ be the projection.
The key ingredient for the definition of the action is the following elementary fact.

\begin{lemma}
\label{lemma:pullback_c1}
For any symplectic manifold $M$, $2 \, \pi^*c_1(M) = 0$ in $H^2(\Grass M; \Z)$. 
\end{lemma}
\begin{proof}
Recall first that if a complex vector bundle $E$ is the complexification of a real bundle,
then $E \iso \overline{E}$. Therefore, we have $c_1(E) = -c_1( \overline{E}) = -c_1(E)$,
so that $2 c_1(E) = 0$.

We apply the previous observation to the bundle $\pi^*TM$. 
Equip $TM$ with a compatible almost complex structure. 
Then the Hermitian bundle $\pi^*TM$ has a Lagrangian subbundle given by the
tautological bundle $\gamma$ over $\Grass M$. 
Hence, there is an isomorphism of complex bundles $\pi^*TM \iso \gamma \tens \C$.
We conclude that $2 \pi^* c_1(M) = 2 c_1( \pi^*TM ) = 0$.
\end{proof}

\begin{corollary}
\label{cor:pullback_omega_exact}
Let $(M, \omega)$ be a monotone symplectic manifold. 
Then $\pi^*\omega$ is exact.
\end{corollary}

Although Corollary \ref{cor:pullback_omega_exact} will be sufficient for our purposes, it
is interesting to note that the exactness of $\pi^*\omega$ actually characterizes monotonicity.

\begin{proposition}
$(M,\omega)$ is monotone if and only if $\pi^*\omega$ is exact.
\end{proposition}
\begin{proof}
It suffices to show that the kernel of $\pi^*: H^2(M; \R) \to H^2(\Grass M; \R)$ is the subspace
generated by $c_1(M)$.
To do this, consider the $E_2$ page of the Serre spectral sequence with $\R$ coefficients of the fibration $\pi$.
Note that $\pi$ is orientable (in the sense that the associated local coefficient system is simple) 
since it is a fiber bundle associated to a principal $U(n)$-bundle over $M$.

By a well-known property (see e.g. \cite[Theorem 5.9]{McCleary}), the kernel of $\pi^*$ is identified
with the image of the transgression
$d_2^{0,1}: E_2^{0,1} \to E_2^{2,0}$.
In the present case, we have $E_2^{0,1} \iso H^1( U(n)/O(n); \R) \iso \R$, generated by the Maslov class $\mu \in H^1(U(n)/O(n); \R)$.
Moreover, $E_2^{2,0} \iso H^2(M; \R)$.
Finally, under these identifications we have that $\mu$
transgresses to $\pm 2 c_1(M)$ (this appears to be well-known; see e.g. Proposition 1 of the appendix of \cite{Viterbo87}).
We conclude that $\Ker \pi^*$ is the subspace spanned by $c_1(M)$.
\end{proof}

A primitive of $\pi^*\omega$ will be called an \Def{action form}.
Using such a primitive, we can define the action of Lagrangian immersions in $M$.

\begin{definition}
\label{def:action}
Let $f: L \to M$ be a Lagrangian immersion.
The \Def{action} of $f$ with respect to the action form $\lambda$ is the class
\[
a_{\lambda}(f) = [\wtilde{f}^*\lambda] \in H^1(L; \R),
\]
where $\wtilde{f}: L \to \Grass M$ is the Gauss map of $f$.
\end{definition}

\begin{remark}
The action depends on the choice of the primitive $\lambda$. 
If $\lambda$ and $\lambda'$ are two primitives,
then $\lambda - \lambda'$ is closed and defines a class
$\delta = [\lambda - \lambda'] \in H^1(\Grass M; \R)$. 
Then, for a Lagrangian immersion $f$ we have
\[
a_{\lambda}(f) - a_{\lambda'}(f) = \wtilde{f}^* \delta.
\]
In particular, if $\delta = 0$, i.e. $\lambda$ and $\lambda'$ differ by an exact form, then $\lambda$ and $\lambda'$ define the same action class.
\end{remark}

\begin{remark}\label{rmk:oriented_vs_unoriented}
In the case where the manifold $L$ is oriented, the action of a Lagrangian immersion $f: L \to M$ could alternatively be defined as
the class $[\wtilde{f}^*\lambda]$, where ${p : \GrassOr(M) \to M}$ is the \emph{oriented} Lagrangian Grassmannian bundle, $\lambda$ is
a primitive of $p^*\omega$ and $\wtilde{f}: L \to \GrassOr(M)$ is the oriented Gauss map.
The two definitions are naturally related by the two-fold covering $\GrassOr(M) \to \Grass(M)$.
This covering induces an isomorphism on $H^1(- ; \R)$, hence it induces a
bijection between the cohomology classes of primitives of $\pi^*\omega$ and those of $p^*\omega$.
As a consequence, for oriented Lagrangians the two definitions are essentially interchangeable.
\end{remark}

\subsection{Invariance under Lagrangian cobordisms}

Our next aim is to prove that action is invariant under Lagrangian cobordisms, in the
sense that any Lagrangian cobordism carries a cohomology class that
extends the action classes of its ends.
To show this, we extend the definition of the action class to Lagrangian cobordisms. 

Observe that for a monotone manifold $(M, \omega)$, the product $(\C \times M, \omega_{\C} \oplus \omega)$ is also monotone.
Hence, by applying Corollary \ref{cor:pullback_omega_exact}, we deduce that $\wtilde{\pi}^*(\omega_{\C} \oplus \omega)$ is exact, where
$\wtilde{\pi}: \Grass(\C \times M) \to \C \times M$ is the projection of the Lagrangian Grassmannian.

\begin{lemma}
\label{lemma:compatible_primitive}
Given an action form $\lambda$, there exists a primitive $\wtilde{\lambda}$ of $\wtilde{\pi}^*( \omega_{\C} \oplus \omega)$ such that
$i^* \wtilde{\lambda} - \lambda$ is exact, where $i: \Grass(M) \to \Grass(\C \times M)$ is the stabilization map.
\end{lemma}
\begin{proof}
The map $i$ induces an isomorphism on $\pi_1$ (see the proof of Lemma \ref{lemma:stable_homology}, which also applies to the unoriented case). 
Hence, $i$ also induces an isomorphism on $H^1(- ; \R)$. 
Therefore, $i$ induces a bijection between the cohomology classes of primitives of $\wtilde{\pi}^*(\omega_{\C} \oplus \omega)$ and those of 
$\pi^*\omega = \iota^* \wtilde{\pi}^*(\omega_{\C} \oplus \omega)$.
\end{proof}

We fix a primitive $\wtilde{\lambda}$ which is compatible with $\lambda$ as in Lemma \ref{lemma:compatible_primitive}.
We can now define the action class of a Lagrangian cobordism $F: V \to \C \times M$ by setting
\[
a_{\tilde{\lambda}}(F) = [\wtilde{F}^* \wtilde{\lambda}] \in H^1(V; \R),
\]
where $\wtilde{F}: V \to \Grass(\C \times M)$ is the Gauss map of $F$.

\begin{lemma}
\label{lemma:invariance_action}
Let $f_k : L_k \to M$, $k=1, \ldots, r$, be Lagrangian immersions and $F:(f_1, \ldots, f_r) \cob \emptyset$
be a Lagrangian cobordism.
Then
\[
a_{\tilde{\lambda}}(F)|_{L_k} =  a_{\lambda}(f_k).
\]
\end{lemma}

\begin{proof}
By definition, the Gauss map of the cobordism satisfies $\wtilde{F}|_{L_k} = i \circ \wtilde{f}_k$.
Therefore, since $i^*\wtilde{\lambda} - \lambda$ is exact, we have
\[
a_{\tilde{\lambda}}(F)|_{L_k} = [ f_{k}^* i^* \wtilde{\lambda}] = [f_k^* \lambda] = a_{\lambda}(f_k).
\]
\end{proof}

In the usual fashion, Lemma \ref{lemma:invariance_action} allows us to define cobordism invariants of oriented Lagrangian immersions
as characteristic numbers involving the action class.

\begin{corollary}
\label{cor:characteristic_numbers}
Let $x \in H^{n-1}(\Grass M; \R)$ be a stable class, i.e. $x$ is in the image of the stabilization map $\Grass M \to \Grass(\C \times M)$. 
Let $f: L \to M$ be an oriented Lagrangian immersion.
Then the number
\[
\langle  a_{\lambda}(f) \cup \wtilde{f}^* x , [L] \rangle \in \R
\]
is an invariant of oriented Lagrangian cobordism.
\end{corollary}

As a special case, when $M$ is a closed surface of genus $g \neq 1$ and we take $x = 1$, we recover 
the holonomy of immersed curves in $M$, as defined in Section \ref{subsection:holonomy}.

\subsection{Action is a primitive of flux}
\label{appendix:flux}

As in the exact case, the action class determines the flux of Lagrangian paths in $M$.
Recall that the \emph{flux} of a path of Lagrangian immersions $f: L \times I \to M$ is the class $\flux(f) \in H^1(L; \R)$
whose value on a loop $\gamma$ in $L$ is given by
\[
\flux(f) \cdot \gamma = \int_{\gamma \times I} f^*\omega.
\]
See for example \cite{Solomon} for a more detailed discussion of Lagrangian flux.

The flux of a path is invariant under homotopies with fixed endpoints.
Hence, for a given manifold $L$, one can see flux as an $H^1(L; \R)$-valued cocycle on the space of Lagrangian immersions $L \to M$.
The next proposition asserts that the action is a primitive of this cocycle.

\begin{proposition}
\label{prop:action_primitive_of_flux}
Let $f : L \times I \to M$ be a Lagrangian path. Then
\[
\flux(f) = a_{\lambda}(f_1) - a_{\lambda}(f_0).
\]
\end{proposition}
\begin{proof}
By taking Gauss maps, $f$ lifts to a path $\wtilde{f} : L \times I \to \Grass M$. 
Then, for any loop $\gamma$ in $L$ we have
\[
\flux(f) \cdot \gamma 
= \int_{\gamma \times I} f^*\omega
= \int_{\gamma \times I} \wtilde{f}^* \pi^* \omega
\stackrel{\text{Stokes}}{=} \int_{\gamma} \wtilde{f}_1^* \lambda - \int_{\gamma} \wtilde{f}_0^*\lambda = a_{\lambda}(f_1) \cdot \gamma - a_{\lambda}(f_0) \cdot \gamma.
\]
\end{proof}

\begin{corollary}
Let $(M, \omega)$ be a monotone manifold and $f: L \times I \to M$ be a Lagrangian loop. Then $\flux(f) = 0$.
\end{corollary}

%% file: sections/appendix_isotopy.tex
In this appendix, we prove Proposition \ref{prop:isotopy_relation}.
The proofs given here elaborate and simplify the arguments used by Perrier in the proofs of Proposition 5.4 and Lemma 5.15 of \cite{Perrier19}.
Our main purpose is to give complete proofs that the cobordisms constructed
in \cite{Perrier19} are unobstructed in the sense of Definition \ref{def:unob_cobordisms}.

We start by proving the proposition in the case of non-separating curves in Section \ref{appendix:nonsep}.
The case of separating curves will be deduced from the non-separating case in Section \ref{appendix:sep}.

\subsection{Non-separating case}
\label{appendix:nonsep}

Let $(\alpha, \wtilde{\alpha})$ and $(\beta, \wtilde{\beta})$ be two pairs of isotopic non-separating curves with 
$\hol(\wtilde{\alpha}) - \hol(\alpha) = \hol(\wtilde{\beta}) - \hol(\beta) = x$. 
We wish to show that 
\[
[\wtilde{\alpha}] - [\alpha] = [\wtilde{\beta}] - [ \beta ]
\]
in $\Gunob(\Sigma)$.

The first case to consider is when $x=0$. 
By Lemma \ref{lemma:ham_isotopic_embedded_case}, we then have that $\alpha$ is Hamiltonian isotopic to $\wtilde{\alpha}$ and $\beta$ is Hamiltonian isotopic to $\wtilde{\beta}$.
It follows that $[\wtilde{\alpha}] - [\alpha] = 0 = [\wtilde{\beta}] - [\beta]$ and Proposition \ref{prop:isotopy_relation} is proved in this case. 
Hence, from now on we will assume that $x \neq 0$.

The general case of Proposition \ref{prop:isotopy_relation} will be deduced from two special cases, which
are stated as Lemma \ref{lemma:disjointcase} and Lemma \ref{lemma:disjointcase2} below.

\begin{lemma}
\label{lemma:disjointcase}
Let $(\alpha, \wtilde{\alpha})$ and $(\beta, \wtilde{\beta})$ be two pairs of isotopic non-separating curves such that 
$\hol(\wtilde{\alpha}) - \hol(\alpha) = \hol(\wtilde{\beta}) - \hol(\beta)$.
Furthermore, suppose that
\begin{enumerate}[label = (\roman*), font=\normalfont]
	\item $\beta$ and $\wtilde{\beta}$ are disjoint,
	\item $\alpha$ is disjoint from $\beta \cup \wtilde{\beta}$.
\end{enumerate}
Then $[\wtilde{\alpha}] - [\alpha] = [\wtilde{\beta}] - [\beta]$ in $\Gunob(\Sigma)$.
\end{lemma}

\begin{proof}

We use the same surgery procedure as in the proof of Proposition 5.4 in \cite{Perrier19}.
The hypothesis implies that there is a curve $\gamma$ that intersects each of the three curves $\alpha$, $\beta$ and $\wtilde{\beta}$ in
one point.
We may assume that $\gamma$ intersects each curve positively.
We then perform the following surgeries, which are illustrated in Figure \ref{fig:surgery_lemmaB1}.

\begin{figure}
	\centering
	\includegraphics[width=\textwidth]{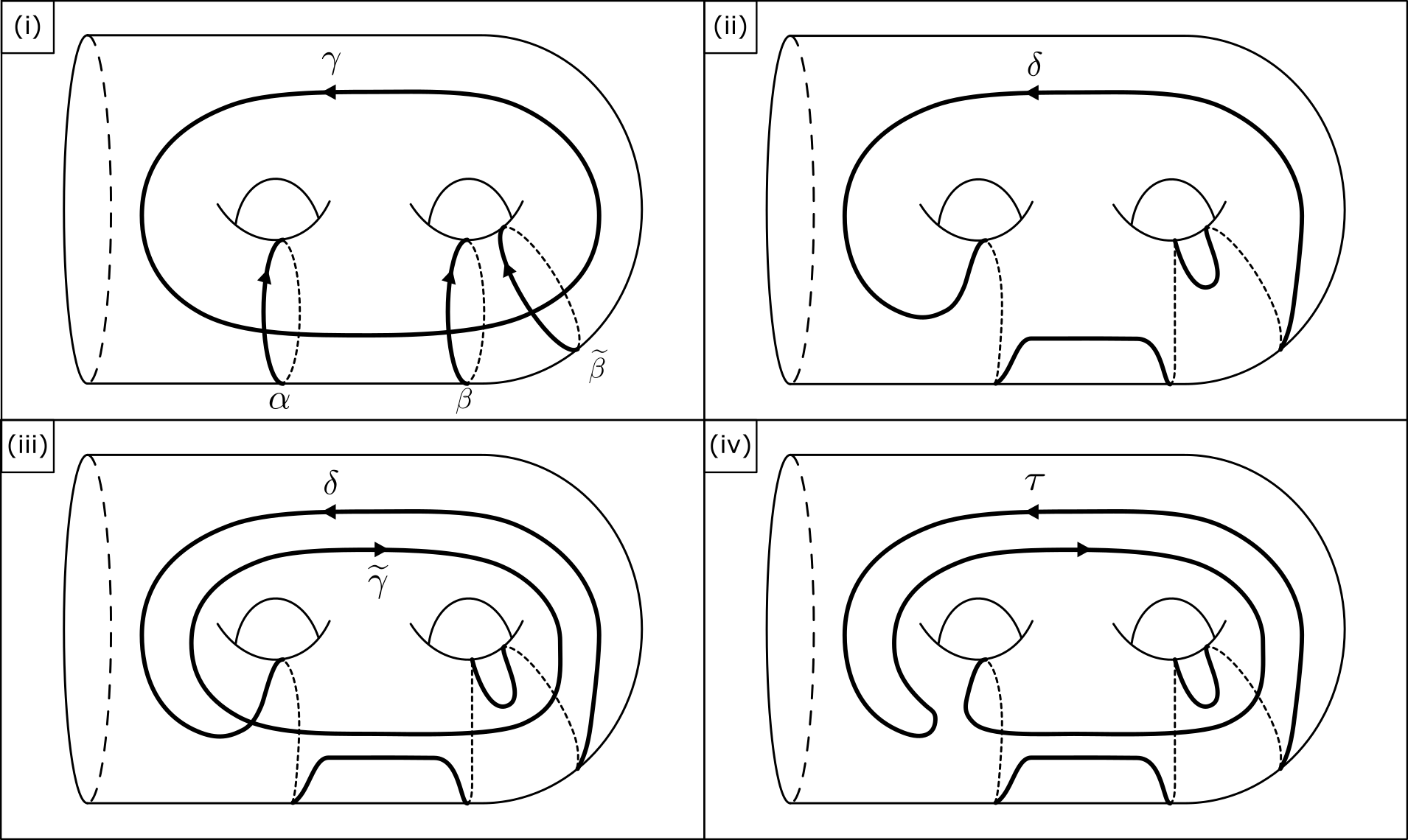}
	\caption{The surgery procedure used in the proof of Lemma \ref{lemma:disjointcase}.}
	\label{fig:surgery_lemmaB1}
\end{figure}

\begin{enumerate}
	\item The successive surgeries of $\gamma$ with $\alpha$, $\beta^{-1}$ and $\wtilde{\beta}$ at their unique intersection point, 
	producing a curve $\delta$.
	\item The surgery of $\delta$ with a curve $\wtilde{\gamma}$ which is Hamiltonian isotopic to $\gamma^{-1}$ and intersects $\delta$ in one point lying on $\alpha$. This produces a curve $\tau$.
\end{enumerate}

Since all the curves involved in the above surgeries intersect only once, the associated surgery cobordisms are embedded.
Moreover, the cobordisms are incompressible by Lemma \ref{cor:nonzerointersectionimpliesnodisks}. 
Hence, by concatenating these cobordisms and using Proposition \ref{prop:concatenation_unobstructed}, we obtain an unobstructed cobordism 
$\tau \cob (\beta^{-1}, \gamma, \alpha, \wtilde{\beta}, \wtilde{\gamma})$.
This gives the relation $[\tau] = [\alpha] + [\wtilde{\beta}] - [\beta]$.

It is easy to check that the curve $\tau$ is isotopic to $\alpha$, hence also isotopic to $\wtilde{\alpha}$. 
Moreover, by the holonomy assumption we have
\[
\hol( \tau ) = \hol(\alpha) + \hol(\wtilde{\beta}) - \hol ( \beta) = \hol(\widetilde{\alpha}).
\]
Hence, by Lemma \ref{lemma:ham_isotopic_embedded_case}, $\tau$ is Hamiltonian isotopic to $\wtilde{\alpha}$ and the conclusion follows.
\end{proof}

\begin{lemma}
\label{lemma:disjointcase2}
Let $(\alpha, \wtilde{\alpha})$ and $(\beta, \wtilde{\beta})$ be two pairs of isotopic non-separating curves such that 
$\hol(\wtilde{\alpha}) - \hol(\alpha) = \hol(\wtilde{\beta}) - \hol(\beta)$. 
Suppose that $\alpha \cap \wtilde{\alpha} = \emptyset = \beta \cap \wtilde{\beta}$.

Then  $[\wtilde{\alpha}] - [\alpha] = [\wtilde{\beta}] - [\beta]$ in $\Gunob(\Sigma)$.
\end{lemma}
\begin{proof}

By Theorem 4.4 of \cite{FarbMargalit}, there is a sequence of non-separating curves $(\gamma_j)_{j=0}^{\ell}$ such that 
$\gamma_0 = \alpha$, $\gamma_{\ell} = \beta$ and $\gamma_j \cap \gamma_{j+1} = \emptyset$ for $j = 0, \ldots, \ell - 1$.

For any integer $N$, we can write
\begin{align*}
[\wtilde{\alpha}] - [\alpha] &= N \;  \left( \, [\alpha'] - [\alpha]  \,  \right)
\end{align*}
where $\alpha'$ is isotopic to $\alpha$ and $\alpha \cap \alpha' = \emptyset$.
Indeed, since $\alpha \cap \wtilde{\alpha} = \emptyset$, there is an isotopy $(\alpha_t)_{t \in [0,1]}$ such that 
$\alpha_0 = \alpha$, $\alpha_1 = \wtilde{\alpha}$ and the 
corresponding map $S^1 \times [0,1] \to \Sigma$ is an embedding (see Lemma 2.4 of \cite{Epstein66}).
Pick a subdivision $t_0 =0 < t_1 < \ldots < t_N = 1$ of the interval $[0,1]$ with the property that 
\[
\hol( \alpha_{t_{i+1}} )  - \hol ( \alpha_{t_i} ) = \frac{x}{N}
\]
for $i = 0 , \ldots, N-1$.
Applying Lemma \ref{lemma:disjointcase} to the pairs $(\alpha, \alpha_{t_1})$ and $(\alpha_{t_i}, \alpha_{t_{i+1}})$, we have
\begin{align*}
[\wtilde{\alpha}] -  [\alpha] 
&= \sum_{i=0}^{N-1} [\alpha_{t_{i+1}}] - [\alpha_{t_i}] \\
&= N \; \left( \, [\alpha_{t_1}] - [\alpha] \, \right)
\end{align*}
Likewise, we can write $[\wtilde{\beta}] - [\beta] = N  \, \left( [\beta'] - [\beta] \right)$.

By choosing $N$ large enough, we can assume that each curve $\gamma_j$ has a tubular neighbordhood $A_j$ such that the two components of $A_j \setminus \gamma_j$ have area greater than $|x|/ N$, and moreover that $A_j \cap A_{j+1} = \emptyset$. 
In particular, for each $j = 0 , \ldots, \ell$ there is a curve $\wtilde{\gamma}_j$ with the following properties:
\begin{enumerate}[label = (\roman*), font=\normalfont]
	\item $\hol(\wtilde{\gamma}_j) - \hol(\gamma_j ) = \frac{x}{N}$,
	\item $\wtilde{\gamma}_j \cap \gamma_j = \emptyset$,
	\item $(\gamma_j, \wtilde{\gamma}_j)$ is disjoint from $(\gamma_{j+1}, \wtilde{\gamma}_{j+1})$.
\end{enumerate}

The result now follows from applying Lemma \ref{lemma:disjointcase} successively to the pairs 
$(\alpha, \alpha')$, $(\gamma_1, \wtilde{\gamma}_1)$, $(\gamma_2, \wtilde{\gamma}_2)$, $\ldots$ , $(\beta, \beta')$.

\end{proof}

\begin{lemma}
\label{lemma:pushoffwithgivenholonomy}
There is a constant $c > 0$ with the following property: 
for each non-separating curve $\alpha$ and each $x \in \R$ with $|x| < c$, there is a
curve $\wtilde{\alpha}$ isotopic to $\alpha$ such that $\alpha \cap \wtilde{\alpha} = \emptyset$ and $\hol(\wtilde{\alpha}) - \hol(\alpha) = x$.
\end{lemma}
\begin{proof}
For a given curve $\alpha$, there is a Weinstein neighborhood $U$ and a symplectomorphism $\vphi: U \iso (-c,c) \times S^1$
such that $\vphi(\alpha) = \{0 \} \times S^1$.
Then the required curves $\wtilde{\alpha}$ are obtained by pushing $\alpha$ in the radial direction of the annulus.
The same constant $c$ works for all non-separating curves, since all non-separating curves on a closed surface are symplectomorphic.
\end{proof}

\begin{remark}
The optimal constant in the preceding lemma is $c = \area(\Sigma)$, but we will not make use of this.
\end{remark}

We now proceed to the proof of Proposition \ref{prop:isotopy_relation} in the case of non-separating curves.

\begin{proof}[Proof of Proposition \ref{prop:isotopy_relation} -- non-separating case.]

Pick a constant $c$ as in Lemma \ref{lemma:pushoffwithgivenholonomy}. 
By choosing isotopies from $\alpha$ to $\wtilde{\alpha}$ and from $\beta$ to $\wtilde{\beta}$
and breaking them into isotopies of small flux, it suffices to prove the proposition when the holonomy difference satisfies $|x| < c$.

By our choice of $c$, there is a curve $\what{\alpha}$ isotopic to $\alpha$ such that $\what{\alpha} \cap \alpha = \emptyset$ and
\[
\hol(\what{\alpha}) - \hol(\alpha) = x.
\]
In particular, $\what{\alpha}$ is Hamiltonian isotopic to $\wtilde{\alpha}$, so $[\what{\alpha}] = [\wtilde{\alpha}]$.
Likewise, there is a curve $\what{\beta}$ isotopic to $\beta$ such that $\what{\beta} \cap \beta = \emptyset$ and $[\what{\beta}] = [\wtilde{\beta}]$.
The proposition now follows by applying Lemma \ref{lemma:disjointcase2} to the pairs $(\alpha, \what{\alpha})$ and $(\beta, \what{\beta})$.
\end{proof}

\subsection{Separating case}
\label{appendix:sep}

The following lemma reduces Proposition \ref{prop:isotopy_relation} in the case of separating curves to the non-separating case.

\begin{lemma}
\label{lemma:reduction_to_nonsep_case}
Let $\alpha$ and $\wtilde{\alpha}$ be isotopic non-contractible separating curves.
Then there are isotopic non-separating curves $\gamma$ and $\wtilde{\gamma}$ such that
\[
[\wtilde{\alpha}] - [\alpha] = [\wtilde{\gamma}] - [\gamma].
\]
\end{lemma}
\begin{proof}
The proof uses a simpler version of the surgery procedure described in the proof of Lemma 5.15 of \cite{Perrier19}.
We start by showing that for a given separating curve $\alpha$, there is a number $\eps(\alpha) > 0$ such that the lemma holds 
for any curve $\wtilde{\alpha}$
isotopic to $\alpha$ that satisfies $|\hol(\wtilde{\alpha}) - \hol(\alpha)| < \eps(\alpha)$.

Since $\alpha$ is separating and non-contractible, there is a non-separating curve $\gamma$ such that 
$\alpha$ and $\gamma$ intersect twice and are in minimal position. 
We are therefore in the setting considered in Section \ref{section:DT_specialcase}.
We perform the following surgeries, which are
a small modification of the procedure described in Section \ref{section:DT_specialcase}.
The whole procedure is illustrated in Figure \ref{fig:surgery_lemmaB5}.

\begin{figure}
	\centering
	\includegraphics[width=\textwidth]{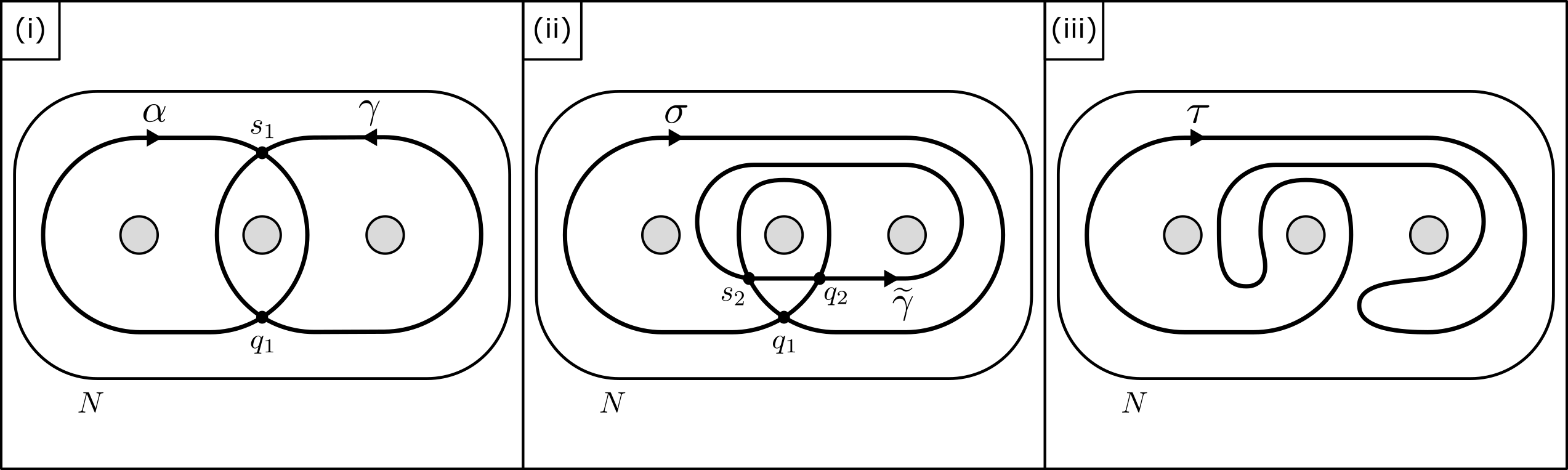}
	\caption{The surgery procedure used in the proof of Lemma \ref{lemma:reduction_to_nonsep_case}.}
	\label{fig:surgery_lemmaB5}
\end{figure}

Let $s_1$ be the negative intersection point of $(\alpha, \gamma)$, and $q_1$ be the positive intersection. 
The first step is to form the surgery $\sigma = \alpha \#_{s_1} \gamma^{-1}$.
Next, by considering small pushoffs of $\gamma$, we can find an $\eps > 0$ with the property that for any $x$ with $|x|< \eps$, there is a curve 
$\wtilde{\gamma}$ isotopic to $\gamma$ that satisfies
\begin{enumerate}[label = (\roman*), font=\normalfont]
	\item $\hol(\wtilde{\gamma}) - \hol(\gamma)  = x$,
	\item $\wtilde{\gamma}$ intersects $\sigma$ in two points;
	the first one, denoted $s_2$, lies on the arc of $\gamma$ near $q_1$, and the other, denoted $q_2$, lies on the arc of $\alpha$ near $q_1$.
\end{enumerate}
Given such a curve $\wtilde{\gamma}$, we let $\tau = \wtilde{\gamma} \#_{s_2} \sigma$.
The size of the surgery handle is chosen big enough so that $\tau$ is embedded. 

The curve $\tau$ is isotopic to $\alpha$. 
Moreover, there is an immersed cobordism $\tau \cob (\wtilde{\gamma}, \alpha, \gamma^{-1})$.
The unobstructedness arguments used in Section \ref{section:DT_specialcase} apply verbatim to the present case, and
we conclude that there is an unobstructed cobordism with the same ends. 
Hence, we obtain the relation
$[\tau] - [\alpha] = [\wtilde{\gamma}]- [\gamma]$.
Moreover, we have
\[
\hol(\tau) = \hol(\alpha) + \hol(\wtilde{\gamma}) - \hol(\gamma) = \hol(\alpha) +  x.
\]

Suppose now that $\wtilde{\alpha}$ is any curve isotopic to $\alpha$ with $\hol(\wtilde{\alpha}) - \hol (\alpha) = x$ and $|x| < \eps$.
Then we have
\[
\hol(\wtilde{\alpha}) = \hol(\alpha) + x = \hol(\tau),
\]
so $\tau$ is Hamiltonian isotopic to $\wtilde{\alpha}$ by Lemma \ref{lemma:ham_isotopic_embedded_case}. 
Hence, we obtain the required relation $[\wtilde{\alpha}] - [\alpha] = [\gamma] - [\wtilde{\gamma}]$.

If $\beta$ is a curve sufficiently $C^1$-close to $\alpha$, then the preceding construction still works with the same choices of $\gamma$ and 
$\wtilde{\gamma}$. 
In particular, this means that the function $\eps(\alpha)$ can be chosen to be lower-semicontinuous in the $C^1$ topology.

Consider now the general case where $\alpha$ and $\wtilde{\alpha}$ are arbitrary isotopic separating curves. 
Choose an isotopy $(\alpha_t)_{t \in [0,1]}$ from $\alpha$ to $\wtilde{\alpha}$.
Then $\eps(\alpha_t)$ is bounded away from $0$, hence there is a subdivision $t_0 =0 < t_1 < \ldots < t_N = 1$
of the interval
such that 
\[
|\hol(\alpha_{t_{i+1}}) - \hol(\alpha_{t_i}) | < \eps(\alpha_{t_i})
\]
for $i = 0, \ldots, N-1$. 
The lemma now follows by applying the preceding case to each pair $(\alpha_{t_{i}}, \alpha_{t_{i+1}})$.
\end{proof}

%% file: sections/appendix_teardrops.tex
In this appendix,
we establish a necessary condition for the existence of marked teardrops
with boundary on immersed curves in a surface $\Sigma$ of genus $g \geq 2$.
In the following, we fix immersed curves $\alpha$ and $\beta$ in $\Sigma$
that are in general position and intersect at a point $s \in \Sigma$.

\begin{proposition}
\label{prop:marked_teardrops_criterion}
Suppose that the curves $\alpha$ and $\beta$ are topologically unobstructed and bound an $s$-marked teardrop.
Then there is a bigon with boundary on $\alpha \cup \beta$.
\end{proposition}

As a sample application, we mention the following corollary.

\begin{corollary}
Let $\alpha$ and $\beta$ be simple curves in minimal position.
Then there is no marked teardrop with boundary on $\alpha \cup \beta$.
\end{corollary}
\begin{proof}
If $\alpha$ and $\beta$ bound a marked teardrop, then by Proposition \ref{prop:marked_teardrops_criterion}
they bound a bigon. By Lemma 1.8 of \cite{FarbMargalit}, they
also bound an embedded bigon.
By the Bigon Criterion \cite[Proposition 1.7]{FarbMargalit}, the curves are not in minimal position.
\end{proof}

Before giving the proof of the proposition, we make some remarks.

\begin{remark}
\hfill
\begin{enumerate}[label = (\roman*)]
	\item The bigon obtained from Proposition \ref{prop:marked_teardrops_criterion} 
	may have boundary on only one of the curves. See Figure \ref{fig:marked_teardrops} for an example.
	\item Under the assumptions of Proposition \ref{prop:marked_teardrops_criterion},
	it is not necessarily the case that the curves bound an embedded bigon.
	However, we may assume that the bigon is the projection of an embedded bigon in the universal cover.
	\item Proposition \ref{prop:marked_teardrops_criterion} hints
	at a deeper relation between marked teardrops and the Floer differential.
	It would be interesting to describe this relation in a more precise way.
\end{enumerate}
\end{remark}

\begin{proof}[Proof of Proposition \ref{prop:marked_teardrops_criterion}]
Write $\HH$ for the hyperbolic plane and fix a universal covering $\pi: \HH \to \Sigma$, 
where the deck group acts by isometries of the hyperbolic metric.
Choose $s$ as a basepoint of $\Sigma$, and fix a basepoint $\tilde{s}_0 \in \pi^{-1}(s)$.
The choices of basepoints give an identification of $\pi_1(\Sigma, s)$ with the deck group of $\pi$.
Let $\wtilde{\alpha}$ and $\wtilde{\beta}$ be the lifts of $\alpha$ and $\beta$ that
start at $\tilde{s}_0$.
Since $\alpha$ and $\beta$ are topologically unobstructed, these lifts
are properly embedded lines in $\HH$, see Lemma \ref{lemma:criterion_top_unob}.

Let $a$ and $b$ be the hyperbolic isometries associated to $[\alpha]$ and $[\beta]$, respectively.
Let $A$ (resp. $B$) be the axis of $a$ (resp. $b$).
The proof splits into three cases, depending on the behaviour of the axes.

\begin{description}
	\item[Case 1] $A=B$.
\end{description}

In this case $a$ and $b$ commute, so that $\alpha$ and $\beta$ generate a cyclic subgroup of $\pi_1(\Sigma, s)$.
This implies that the lifts $\wtilde{\alpha}$ and $\wtilde{\beta}$ intersect infinitely many times.
In particular, the lifts bound an embedded bigon, which projects to a bigon on $\alpha$ and $\beta$.

\begin{description}
	\item[Case 2] $A$ and $B$ are disjoint.
\end{description}

The lift $\wtilde{\alpha}$ has the same endpoints at infinity as $A$, and likewise for $\wtilde{\beta}$ and $B$.
Since $A$ and $B$ are disjoint, $\wtilde{\alpha}$ and $\wtilde{\beta}$ must
intersect at least twice. Hence the lifts bound an embedded bigon, which 
projects to a bigon
on $\alpha$ and $\beta$.

\begin{description}
	\item[Case 3] $A$ and $B$ are distinct and intersect.
\end{description}

In this case, the axes $A$ and $B$ intersect exactly once.
Let $\Gamma$ be the subgroup of $\pi_1(\Sigma, s)$ generated
by $\alpha$ and $\beta$.
Let $T = \HH/\Gamma$ be the covering of $\Sigma$ associated to $\Gamma$.
The hypothesis on the axes implies that $T$ is a punctured torus
and that $A$ and $B$ project to
simple closed geodesics $\gamma_A$ and $\gamma_B$ in $T$ that intersect once and form a basis of $\pi_1(T)$.
These properties can be proved in an elementary way by constructing an explicit fundamental domain for the action of $\Gamma$; see
the proof of Theorem 8 of \cite{Purzitsky}.

On the other hand, the lifts $\wtilde{\alpha}$ and $\wtilde{\beta}$ project 
to closed curves $\overline{\alpha}$ and $\overline{\beta}$ in $T$
that are lifts of $\alpha$ and $\beta$.
These curves are freely homotopic to $\gamma_A$ and $\gamma_B$, respectively.
Moreover, the existence of the $s$-marked teardrop on $\alpha$ and $\beta$
implies that either the lifts $\overline{\alpha}$ and $\overline{\beta}$
intersect at least twice, or that
one of the lifts has a self-intersection.

If either $\overline{\alpha}$ or $\overline{\beta}$ is not simple,
then since it is freely homotopic to a simple geodesic it bounds an embedded
bigon by Theorem 2.7 of \cite{HassScott}.
Otherwise, the curves $\overline{\alpha}$ and $\overline{\beta}$
are simple and intersect at least twice.
Since the geodesics $\gamma_A$ and $\gamma_B$ intersect only once, 
$\overline{\alpha}$ and $\overline{\beta}$
are not in minimal position.
Hence they bound an embedded bigon by the bigon criterion \cite[Proposition 1.7]{FarbMargalit}.
\end{proof}